\documentclass[11pt]{article}
\usepackage[T1]{fontenc}
\usepackage{latexsym,amssymb,amsmath,amsfonts,amsthm}
\usepackage{graphics}
\usepackage{graphicx}
\usepackage{mathrsfs}
\usepackage{subfigure}
\usepackage{hyperref}
\hypersetup{
            colorlinks=true,
            linkcolor=blue,
            anchorcolor=blue,
            citecolor=blue}
\usepackage{refcheck}
\usepackage{color}
\usepackage{epstopdf}
\usepackage{bookmark}
\newcommand{\R}{{\mathbb R}}
\newcommand{\Z}{{\mathbb Z}}

\newcommand{\C}{{\mathbb C}}

\newcommand{\be}{\begin{eqnarray}}
\newcommand{\ben}{\begin{eqnarray*}}
\newcommand{\en}{\end{eqnarray}}
\newcommand{\enn}{\end{eqnarray*}}
\newcommand{\ba}{\backslash}
\newcommand{\pa}{\partial}

\newcommand{\ov}{\overline}

\newcommand{\Om}{\Omega}
\newcommand{\om}{\omega}

\newcommand{\ra}{\rightarrow}
\newtheorem{theorem}{Theorem}[section]
\newtheorem{lemma}[theorem]{Lemma}

\newtheorem{remark}[theorem]{Remark}

\newtheorem{example}{Example}[]

\definecolor{xxx}{rgb}{0,0,0}
\newcommand{\x}{\color{xxx}}
\definecolor{blau}{rgb}{0,0,0}
\newcommand{\bl}{\color{blau}}
\begin{document}
\renewcommand{\theequation}{\arabic{section}.\arabic{equation}}
\begin{titlepage}

\title{Simultaneously recover two constant coefficients and a polygon with a single pair of Cauchy data for the Helmholtz equation}

\author{Xiaoxu Xu\thanks{School of Mathematics and Statistics, Xi'an Jiaotong University, Xi'an, Shaanxi, 710049, China (\sf{xuxiaoxu@xjtu.edu.cn})}
\and Guanghui Hu\thanks{School of Mathematical Sciences and LPMC, Nankai University, Tianjin 300071, China ({\sf ghhu@nankai.edu.cn}).}}

\date{}
\end{titlepage}
\maketitle
\vspace{.2in}
\begin{abstract}
This paper is concerned with an inverse boundary value problem for the Helmholtz equation over a bounded domain.
The aim is to reconstruct two constant coefficients together with the location and shape of a Dirichlet polygonal obstacle from a single pair of Cauchy data.
Uniqueness results are verified under some a priori assumptions and the one-wave factorization method has been adapted to recover the polygonal obstacle as well as the two coefficients.
{\bl A modified factorization using the Dirichlet-to-Neumann operator is employed to overcome difficulties arising from possible eigenvalues.}
Intensive numerical examples indicate that our method is efficient.

\vspace{.2in} {\bf Keywords}: Inverse boundary value problem, uniqueness, single Cauchy data, coefficient recovery, one-wave factorization method.
\end{abstract}

\section{Introduction}

Suppose that $D\subset\R^2$ is a convex polygon that contained in the interior of a disk $B$.
We consider the following boundary value problem:
\be\label{-1}
\nabla\cdot(\sigma\nabla u)+qu=0 & \text{in}\;B\ba\ov{D},\\ \label{-2}
u=f & \text{on}\;\pa B,\\ \label{-3}
u=0 & \text{on}\;\pa D.
\en
Here, $f$ is called Dirichlet boundary value of $u$ on $\pa B$ and and $D$ is assumed to be a Dirichlet obstacle which means that $u\!=\!0$ on $\pa D$. We restrict ourselves to the special case that both $\sigma\in\C\ba\{0\}$ and $q\in\C$ are constants. In this case,
the boundary value problem (\ref{-1})--(\ref{-3}) can be rewritten as
\be\label{1}
\Delta u+k^2u=0 & \text{in } B\ba\ov{D},\\ \label{2}
u=f & \text{on }\pa B,\\ \label{3}
u=0 & \text{on }\pa D,
\en
where $k=\sqrt{q/\sigma}$ with the branch of the square root being chosen such that ${\rm Im}\,k\geq0$.
The equation (\ref{1}) is well known as the Helmholtz equation.
According to \cite[Chapter 4]{WM} (see also \cite[Section 3.3]{Monk}), the boundary value problem (\ref{1})--(\ref{3}) is uniquely solvable in $H^1(B\backslash\overline{D})$ for all $f\in H^{1/2}(\pa D)$ and $k^2\in(0,+\infty)\ba\mathcal S$, where $\mathcal S\subset (0,+\infty)$ denotes a discrete set of Dirichlet eigenvalue with the only accumulating point $+\infty$ (see \cite[Theorem 5.18]{Cakoni14}).
Note that $k^2$ is called an eigenvalue of the boundary value problem (\ref{1})--(\ref{3}) if (\ref{1})--(\ref{3}) admits a nontrivial solution with $f=0$ on $\pa B$. In this paper, we always assume that $k^2$ is not an eigenvalue of the boundary value problem (\ref{1})--(\ref{3}).

The inverse boundary value problem concerned in this paper is to determine the coefficients $\sigma$, $q$ and the obstacle $D$ from a single Cauchy data $(f,\sigma\pa_\nu u|_{\pa B})$, where $\nu$ denotes the unit normal on $\pa B$ directed into the exterior of $B$.
This is an extension {\bl to the Helmholtz equation} of our previews work \cite{XH24} where an inverse boundary problem for {\bl the Laplace's equation} has been studied.
Analogously, we shall adapt the one-wave factorization method proposed in \cite{EH19,HL,MaHu} for inverse scattering problems to the inverse boundary value problem under question. The original version of the factorization method, proposed by A. Kirsch (see \cite{Kirsch98, Kirsch08}), requires measurement data over all observation and incident directions but does not rely on physical properties of the target. It is based on designing a point-wisely defined indicator function which decides on whether a sampling point lies inside or outside of the target. The one-wave factorization method belongs to the class of domain-defined non-iterative sampling schemes (see \cite{NP2013} for an overview), which does not rely on forward solvers and usually involves synthetic data for testing targets. There are some other non-iterative approaches with a single far-field pattern or Cauchy data, for example, the enclosure method \cite{Ik1998,I99,Ik1999}, the one-wave range test \cite{Lin2021,Lin2023,rangetest,Sun2023}, one-wave non-response test \cite{P2003,Lin2021,Lin2023,Sun2023} and the extended linear sampling method \cite{LS2018}.
The advantages of the one-wave factorization lies in the fact that
it yields a sufficient and necessary condition to characterize a convex obstacle of polygonal/polyhedral type based on the corner singularity analysis \cite{BLS,ElHu2015,hu2018,HL,KS,MaHu}.
The computational criterion from the one-wave factorization looks more explicit and can be easily implemented when sampling disks or straight lines are chosen as testing targets \cite{XH2024}.

The inverse boundary value problem for (\ref{-1}) has been investigated in \cite{Kirsch05}, where {\bl a factorization method using the boundary Neumann-to-Dirichlet operator} has been employed to recover a general interface. In the first part of this paper, we shall establish a factorization scheme for a testing domain which is  supposed to be occupied by either an impedance obstacle or an absorbing medium.
Instead of the Neumann-to-Dirichlet operator used in \cite{Kirsch05} and \cite[Chapter 6]{Kirsch08}, we will establish a framework based on the Dirichlet-to-Neumann operator in this paper (see \cite{XH24}).
The one-wave factorization method is a direct consequence of the factorization scheme for testing domains and the singularity of the solution around a corner point of the obstacle. Both uniqueness results and numerical methods will be considered in our paper.
{\x Promising features of this article include the design of indicator functions for recovering two constant coefficients and an efficient implementation of the one-wave factorization method.}
{\bl Our numerical tests show that, an optimization-based iterative approach with an initial guess given by the result of this method will lead to a more precise inversion result.}

The remaining part of this paper is organized as follows.
Some preliminary results from factorization method based on Dirichlet-to-Neumann operators will be given in Section \ref{s2}.
Sections \ref{s3} and \ref{sec4} are devoted to the uniqueness results and numerical methods of the inverse problem based on a single Cauchy data, respectively.
Numerical examples will be given in Section \ref{s4}.
Finally, a conclusion will be given in Section \ref{s5}.

\section{Factorization of Dirichlet-to-Neumann operators}\label{s2}
\setcounter{equation}{0}

In order to reconstruct the coefficients $\sigma$, $q$ and the inclusion $D$ from a single Cauchy data, we will apply the one-wave factorization method developed in \cite{MaHu}.
To this end, we first introduce the classical factorization method based on Cauchy data to the boundary value problem of the Helmholtz equation.
{\bl To overcome difficulties due to possible eigenvalues, a modified factorization using the Dirichlet-to-Neumann operator will be prosed in the second subsection.}

\subsection{Dirichlet-to-Neumann operators and the range identity}

By the trace theorem for $H_\Delta^1(\Om)$ and $L_\Delta^2(\Om)$ (see \cite[Theorem 2.1]{XH24}) and proceeding as in the proof of \cite[Theorems 2.2 and 2.3]{XH24} and \cite[Example 5.17]{Cakoni14}, we can show the existence of a unique solution to \eqref{1}--\eqref{3} and the following boundary value problems:
\be\label{5}
\Delta u_0+k^2u_0=0 & \text{in}\; B,\\ \label{6}
u_0=f & \text{on}\;\pa B,
\en
\be\label{1'}
\Delta v+k^2v=0 & \text{in } B\ba\ov{\Om},\\ \label{2'}
v=f & \text{on }\pa B,\\ \label{3'}
\pa_\nu v+\eta v=0 & \text{on }\pa \Om,
\en
and
\be\label{m1}
\Delta w+k^2nw=0 & \text{in}\; B,\\ \label{m2}
w=f & \text{on}\;\pa B.
\en
{\bl Here,} $\Om\subset B$ is a Lipschitz domain, $\nu$ denotes the unit normal on $\pa\Om$ directed into the exterior of $\Om$, $\eta\in C(\pa\Om)$ in \eqref{1'}--\eqref{3'},  $n\in L^\infty(B)$ and ${\rm supp}\,(n-1)=:\ov{\mho}$ is contained in $B$ {\bl and} $k^2$ is not an eigenvalue of the corresponding boundary value problem.
{\bl Note that} $k^2$ is called an eigenvalue of the boundary value problem \eqref{5}--\eqref{6} if \eqref{5}--\eqref{6} admits a nontrivial solution with $f=0$ on $\pa B$.
The eigenvalues of \eqref{1'}--\eqref{3'} and \eqref{m1}--\eqref{m2} are defined analogously.
{\bl Several remarks concerning the above boundary value problems are summarized below.}

\begin{remark}\label{rem-}
(i) If $k^2$ is not an eigenvalue of \eqref{5}--\eqref{6}, then $\|u_0\|_{H^1_\Delta(B)}\leq C\|f\|_{H^{1/2}(\pa B)}$ for all $f\in H^{1/2}(\pa B)$ and $\|u_0\|_{L^2_\Delta(B)}\leq C\|f\|_{H^{-1/2}(\pa B)}$ for all $f\in H^{-1/2}(\pa B)$, where $C>0$ is a constant independent of $f$.
Define the Dirichlet-to-Neumann operator $A_0:=A_0(k^2)$ corresponding to (\ref{5})--(\ref{6}) by $A_0f:=\pa_\nu u_0$ on $\pa B$.
By the trace theorem \cite[Theorem 2.1]{XH24}, we know $A_0:H^{1/2}(\pa B)\ra H^{-1/2}(\pa B)$ and $A_0:H^{-1/2}(\pa B)\ra H^{-3/2}(\pa B)$ are bounded.
By the interpolation property of the Sobolev
spaces (see \cite[Theorem 8.13]{Kress14}), $A_0:H^s(\pa B)\ra H^{s-1}(\pa B)$ is bounded for $s\in[-\frac12,\frac12]$.
By trace theorem and the regularity of elliptic equations (see e.g., \cite{Evans,GT}), we know $A_0:H^s(\pa B)\ra H^{s-1}(\pa B)$ is bounded for $s\geq-\frac12$.

(ii) If $k^2$ is not an eigenvalue of \eqref{1}--\eqref{3}, then $\|u\|_{H^1_\Delta(B\ba\ov{D})}\leq C\|f\|_{H^{1/2}(\pa B)}$ for all $f\in H^{1/2}(\pa B)$.
Assume further that $B_1$ is an open subset of $B$ such that $\ov{D}\subset B_1\subset\ov{B_1}\subset B$.
Then $\|u\|_{H_\Delta^1(B_1\ba\ov{D})}+\|u\|_{L^2_\Delta(B\ba\ov{B_1})}\leq C\|f\|_{H^{-1/2}(\pa B)}$ for all $f\in H^{-1/2}(\pa B)$.
Here, $C>0$ is a constant independent of $f$.
Define the Dirichlet-to-Neumann operator $A:=A(k^2,D)$ corresponding to \eqref{1}--\eqref{3} by $Af:=\pa_\nu u$ on $\pa B$.
By the trace theorem \cite[Theorem 2.1]{XH24}, we know $A:H^{1/2}(\pa B)\ra H^{-1/2}(\pa B)$ and $A:H^{-1/2}(\pa B)\ra H^{-3/2}(\pa B)$ are bounded.
By the interpolation property of the Sobolev
spaces (see \cite[Theorem 8.13]{Kress14}), $A:H^s(\pa B)\ra H^{s-1}(\pa B)$ is bounded for $s\in[-\frac12,\frac12]$.
By trace theorem and the regularity of elliptic equations (see e.g., \cite{Evans,GT}), we know $A:H^s(\pa B)\ra H^{s-1}(\pa B)$ is bounded for $s\geq-\frac12$.

(iii) If $k^2$ is not an eigenvalue of \eqref{1'}--\eqref{3'}, then $\|v\|_{H^1_\Delta(B\ba\ov{\Om})}\leq C\|f\|_{H^{1/2}(\pa B)}$ for all $f\in H^{1/2}(\pa B)$.
Assume further that $B_1$ is an open subset of $B$ such that $\ov{\Om}\subset B_1\subset\ov{B_1}\subset B$.
Then $\|v\|_{H_\Delta^1(B_1\ba\ov{\Om})}+\|v\|_{L^2_\Delta(B\ba\ov{B_1})}\leq C\|f\|_{H^{-1/2}(\pa B)}$ for all $f\in H^{-1/2}(\pa B)$.
Here, $C>0$ is a constant independent of $f$.
Define the Dirichlet-to-Neumann operator $A:=A(k^2,\Om)$ corresponding to \eqref{1'}--\eqref{3'} by $Af:=\pa_\nu v$ on $\pa B$.
By the trace theorem \cite[Theorem 2.1]{XH24}, we know $A:H^{1/2}(\pa B)\ra H^{-1/2}(\pa B)$ and $A:H^{-1/2}(\pa B)\ra H^{-3/2}(\pa B)$ are bounded.
By the interpolation property of the Sobolev
spaces (see \cite[Theorem 8.13]{Kress14}), $A:H^s(\pa B)\ra H^{s-1}(\pa B)$ is bounded for $s\in[-\frac12,\frac12]$.
By trace theorem and the regularity of elliptic equations (see e.g., \cite{Evans,GT}), we know $A:H^s(\pa B)\ra H^{s-1}(\pa B)$ is bounded for $s\geq-\frac12$.

(iv) If $k^2$ is not an eigenvalue of \eqref{m1}--\eqref{m2}, then $\|w\|_{H^1_\Delta(B)}\leq C\|f\|_{H^{1/2}(\pa B)}$ for all $f\in H^{1/2}(\pa B)$.
Assume further that $B_1$ is an open subset of $B$ such that $\ov{\mho}\subset B_1\subset\ov{B_1}\subset B$.
Then $\|w\|_{H_\Delta^1(B_1)}+\|w\|_{L^2_\Delta(B\ba\ov{B_1})}\leq C\|f\|_{H^{-1/2}(\pa B)}$ for all $f\in H^{-1/2}(\pa B)$.
Define the Dirichlet-to-Neumann operator $A:=A(k^2,\mho)$ corresponding to \eqref{m1}--\eqref{m2} by $Af:=\pa_\nu w$ on $\pa B$.
By the trace theorem \cite[Theorem 2.1]{XH24}, we know $A:H^{1/2}(\pa B)\ra H^{-1/2}(\pa B)$ and $A:H^{-1/2}(\pa B)\ra H^{-3/2}(\pa B)$ are bounded.
By the interpolation property of the Sobolev
spaces (see \cite[Theorem 8.13]{Kress14}), $A:H^s(\pa B)\ra H^{s-1}(\pa B)$ is bounded for $s\in[-\frac12,\frac12]$.
By trace theorem and the regularity of elliptic equations (see e.g., \cite{Evans,GT}), we know $A:H^s(\pa B)\ra H^{s-1}(\pa B)$ is bounded for $s\geq-\frac12$.

(v) The solution $u_0$ to \eqref{5}--\eqref{6} is analytic in $B$; the solution $u$ to \eqref{1}--\eqref{3} is analytic in $B\ba\ov{D}$; the solution $v$ to \eqref{1'}--\eqref{3'} is analytic in $B\ba\ov{\Om}$; the solution $w$ to \eqref{m1}--\eqref{m2} is analytic in $B\ba\ov{\mho}$; see \cite[Theorem 3.2]{Cakoni14}.

(vi) If ${\rm Im}\,\eta>0$ on $\pa B$, then $k^2>0$ is not an eigenvalue of \eqref{1'}--\eqref{3'} (see \cite[Theorem 8.2]{Cakoni14}).
Taking the complex conjugate of \eqref{1'}--\eqref{3'}, we know $k^2>0$ is not an eigenvalue of \eqref{1'}--\eqref{3'} provided ${\rm Im}\,\eta<0$ on $\pa B$.
If ${\rm Im}\,n>0$ in $\ov{B}$, then $k^2>0$ is not an eigenvalue of \eqref{m1}--\eqref{m2} (see \cite[Theorem 8.12]{CK19}).
Similarly, taking the complex conjugate of \eqref{m1}--\eqref{m2}, we know $k^2>0$ is not an eigenvalue of \eqref{m1}--\eqref{m2} provided ${\rm Im}\,n<0$ in $\ov{B}$.
\end{remark}

Now we return to the boundary value problem \eqref{-1}--\eqref{-3}.
We define the Dirichlet-to-Neumann operator $\Lambda(\sigma,q,D)$ by $\Lambda(\sigma,q,D)f=\sigma\pa_\nu u|_{\pa B}$, where $u$ solves \eqref{-1}--\eqref{-3}.
With this notation, the Cauchy data $(u|_{\pa B},\pa_\nu u|_{\pa B})$ to \eqref{-1}--\eqref{-3} can be represented as $(f,\Lambda(\sigma,q,D)f)$, or, equivalently, $(f,\sigma A(q/\sigma,D)f)$.

Below we shall derive factorizations of $A(k^2,\Om)-A_0(k^2)$ and $A(k^2,\mho)-A_0(k^2)$.
%
For this purpose, we define the operator $G(k^2,\Om):H^{-1/2}(\pa\Om)\ra H^{-1/2}(\pa B)$ by $G(k^2,\Om)g=\pa_\nu v$ on $\pa B$ where $v$ is the unique solution to the boundary value problem
\be\label{7}
\Delta v+k^2v=0 & \text{in}\;B\ba\ov{\Om},\\ \label{8}
v=0 & \text{on}\;\pa B,\\ \label{9}
\pa_\nu v+\eta v=g & \text{on}\;\pa\Om,
\en
{\bl and $\Om$, $\eta$ are defined as the} same as in (\ref{1'})--(\ref{3'}).
Similarly, define the operator $G(k^2,\mho):L^2(\mho)\ra H^{-1/2}(\pa B)$ by $G(k^2,\mho)\rho=\pa_\nu w$ where $w$ solves
\be\label{G-1}
\Delta w+k^2nw=(1-n)\rho & \text{in }B,\\ \label{G-2}
w=0 & \text{on }\pa B,
\en
{\bl and} $n$ is the same as in \eqref{m1}--\eqref{m2}.
If $k^2$ is not an eigenvalue of \eqref{1'}--\eqref{3'}, then there exists a unique solution to (\ref{7})--(\ref{9}) for any $g\in H^{-1/2}(\pa B)$.
Similarly, if $k^2$ is not an eigenvalue of \eqref{m1}--\eqref{m2}, then there exists a unique solution to (\ref{G-1})--(\ref{G-2}) for any $\rho\in L^2(\mho)$.
Therefore, $G(k^2,\Om):H^{-1/2}(\pa\Om)\ra H^{-1/2}(\pa B)$ and $G(k^2,\mho):L^2(\mho)\ra H^{-1/2}(\pa B)$ are well-defined.
Analogously to the proof of \cite[Theorem 2.5]{XH24}, {\bl we have} mapping properties of $G(k^2,\Om)$ and $G(k^2,\mho)$ summarized as follows.

\begin{theorem}\label{t3.3}
For any integer $m\geq1$ the following hold:

(a) If $k^2$ is not an eigenvalue of \eqref{1'}--\eqref{3'}, then $G(k^2,\Om):H^{-1/2}(\pa\Om)\ra H^{m-3/2}(\pa B)$ is bounded and $G(k^2,\Om):H^{-1/2}(\pa\Om)\ra L^2(\pa B)$ is compact.

(b) If $k^2$ is not an eigenvalue of \eqref{m1}--\eqref{m2}, then $G(k^2,\mho):L^2(\mho)\ra H^{m-3/2}(\pa B)$ is bounded and $G(k^2,\mho):L^2(\mho)\ra L^2(\pa B)$ is compact.
\end{theorem}

{\bl To factorize the operators} $A(k^2,\Om)-A_0(k^2)$ and $A(k^2,\mho)-A_0(k^2)$, we need to introduce several integral operators.
We begin with the Green function to the Helmholtz equation in $B$ with the homogeneous Dirichlet boundary condition
\ben
\Psi_k(x,y):=\Phi_k(x,y)+\psi_k(x,y),\quad x,y\in B,\;x\neq y,
\enn
where $\Phi_k(x,y):=\frac{i}{4}H_0^{(1)}(k|x-y|)$ is the fundamental solution to the Helmholtz equation in $\R^2$ with wave number $k$ (see \cite[(3.106)]{CK19}) and $u_0=\psi_k(\cdot,y)$ is the unique solution to (\ref{5})--(\ref{6}) with $f:=-\Phi_k(\cdot,y)$ on $\pa B$. {\bl Note that $k^2$ is not an eigenvalue of \eqref{5}--\eqref{6}.
We observe} that $\Psi_k(x,y)$ is analytic in $x\in B\ba\{y\}$ and $\Psi_k(\cdot,y)=0$ on $\pa B$.
Moreover, proceeding as in \cite{XH24} {\bl one can prove} the symmetric property for $\Psi_k(x,y)$ and {\bl get} an explicit representation of $u_0$ in terms of the boundary data $f$ as follows.

\begin{lemma}\label{lem211025}
Assume $k^2$ is not an eigenvalue of \eqref{5}--\eqref{6}. Then the following hold:

(i) $\Psi_k(x,y)=\Psi_k(y,x)$ for all $x,y\in B$.

(ii) For $f\in H^s(\pa B)$ with $s\in[-\frac12,\frac12]$, the solution $u_0\in L^2_\Delta(B)$ to (\ref{5})--(\ref{6}) is given by
\be\label{211025-1}
u_0(x)=-\int_{\pa B}\frac{\pa \Psi_k(x,y)}{\pa\nu(y)}f(y)ds(y),\quad x\in B.
\en
In particular, $u_0\in H^1_\Delta(B)$ provided $f\in H^{1/2}(\pa B)$.
\end{lemma}

\begin{remark}\label{rem3.4}
$\Psi_k(x,y)$ is {\bl real valued} for all $x,y\in B$.
Actually, for any fixed $y\in B$ define $\tilde u(x,y):=\frac{i}2J_0(k|x-y|)+\psi_k(x,y)-\ov{\psi_k(x,y)}$, $x\in B$, where $J_0$ denotes the Bessel function of order zero.
It is easy to verify that $\tilde u(x,y)=\Psi_k(x,y)-\ov{\Psi_k(x,y)}$ for $x\neq y$.
Hence, $\tilde u(\cdot,y)$ solves the Helmholtz equation in $B$ and vanishes on $\pa B$.
Therefore, $\tilde u(\cdot,y)=0$ since $k^2$ is not an eigenvalue of (\ref{5})--(\ref{6}).
\end{remark}


Analogously to \cite[(3.8)--(3.11) and (11.62)]{CK19}, we define the boundary integral operators $S:=S(k^2,\pa\Om)$, $K:=K(k^2,\pa\Om)$, $K':=K'(k^2,\pa\Om)$, and $T:=T(k^2,\pa\Om)$ by
\be\label{bi-1}
&(S\varphi)(x):=\int_{\pa\Om}\Psi_k(x,y)\varphi(y)ds(y),\quad x\in\pa\Om,&\\ \label{bi-2}
&(K\varphi)(x):=\int_{\pa\Om}\frac{\pa\Psi_k(x,y)}{\pa\nu(y)}\varphi(y)ds(y),\quad x\in\pa\Om,&\\ \label{bi-3}
&(K'\varphi)(x):=\frac{\pa}{\pa\nu(x)}\int_{\pa\Om}\Psi_k(x,y)\varphi(y)ds(y),\quad x\in\pa\Om,&\\ \label{bi-4}
&(T\varphi)(x):=\frac{\pa}{\pa\nu(x)}\int_{\pa\Om}\frac{\pa\Psi_k(x,y)}{\pa\nu(y)}\varphi(y)ds(y),\quad x\in\pa\Om,&
\en
and the volume integral operator $S_\mho:=S_\mho(k^2)$ by
\ben
(S_\mho\psi)(x)=-\frac{\psi}{1-n}-k^2\int_{\mho}\Psi_k(x,y)\psi(y)dy,\quad x\in\mho.
\enn
{\bl Properties} of the above integral operators are given in the following theorems.
{\x The proofs of these properties are omitted, because they are similar to those in \cite{Cakoni14,CK19,CK83,Costabel,Kirsch08,WM,XH24}.}
\begin{theorem}\label{t+}
Assume $k^2$ is not an eigenvalue of \eqref{5}--\eqref{6}.

(a) {\bl The integral operator $H:L^2(\pa B)\ra H^{-1/2}(\pa\Om)$ defined by}
\ben
(Hf)(x):=-\left(\frac{\pa}{\pa\nu(x)}+\eta\right)\int_{\pa B}\frac{\pa \Psi_k(x,y)}{\pa\nu(y)}f(y)ds(y),\quad x\in\pa\Om.
\enn
is bounded.
If $k^2$ is not an eigenvalue of
\be\label{imp1}
\Delta u_0+k^2u_0=0 && \text{in }\Om,\\ \label{imp2}
\pa_\nu u_0+\eta u_0=0 && \text{on }\pa\Om,
\en
i.e., the unique solution of \eqref{imp1}--\eqref{imp2} is the zero solution $u_0=0$, then $H$ is injective.

(b) $S:H^{-1/2}(\pa\Om)\ra H^{1/2}(\pa\Om)$ is bounded  and self-adjoint.

(c) $K:H^{-1/2}(\pa\Om)\ra H^{1/2}(\pa\Om)$, $K':H^{-1/2}(\pa\Om)\ra H^{1/2}(\pa\Om)$ are bounded and adjoint.

(d) $T:H^{1/2}(\pa\Om)\ra H^{-1/2}(\pa\Om)$ are bounded and self-adjoint.

(e) Define $T_{imp}:=T+i{\rm Im}\,\eta I+K'\ov{\eta}+\eta K+\eta S\ov{\eta}$ for some $\eta\in C(\pa\Om)$ satisfying ${\rm Im}\,\eta>0$ or ${\rm Im}\,\eta<0$ on $\pa\Om$.
Then $T_{imp}:H^{1/2}(\pa\Om)\ra H^{-1/2}(\pa\Om)$ is bounded.
Define $T_i:H^{1/2}(\pa\Om)\ra H^{-1/2}(\pa\Om)$ by
\ben
(T_i\varphi)(x)=\frac{\pa}{\pa\nu(x)}\int_{\pa\Om}\frac{\pa\Phi_i(x,y)}{\pa\nu(y)}\varphi(y)ds(y),\quad x\in\pa\Om.
\enn
Then $-T_i$ is coercive, i.e., there exists a constant $c>0$ such that
\ben
\langle-T_i\varphi,\ov{\varphi}\rangle\geq c\|\varphi\|_{H^{1/2}(\pa\Om)}^2\quad\text{for all }\varphi\in H^{1/2}(\pa\Om),
\enn
and $(T_{imp}-T_i):H^{1/2}(\pa\Om)\ra H^{-1/2}(\pa\Om)$ is compact.
Moreover, it holds that
\ben
{\rm sign}({\rm Im}\,\eta){\rm Im}\langle\,T_{imp}\varphi,\ov{\varphi}\rangle>0\quad\text{for all }\varphi\in H^{1/2}(\pa\Om)\ba\{0\}.
\enn
\end{theorem}

The coercivity of the boundary integral operator $T_{imp}$ plays an important role in the factorization method.
It should be remarked that we consider the impedance boundary condition with ${\rm Im}\,\eta>0$ or ${\rm Im}\,\eta<0$ on $\pa\Om$ since, in general, the boundary integral operators are not coercive under the Dirichlet or Neumann boundary condition, even if $k^2$ is not a corresponding eigenvalue.
Moreover, analogously to Remark \ref{rem-} (vi), if ${\rm Im}\,\eta>0$ or ${\rm Im}\,\eta<0$ on $\pa\Om$,  then $k^2>0$ is {\bl cannot be} an eigenvalue of \eqref{imp1}--\eqref{imp2}.

\begin{theorem}\label{t+m}
Assume $k^2$ is not an eigenvalue of \eqref{5}--\eqref{6}.

(a) Define $H_\mho:L^2(\pa B)\ra L^2(\mho)$ by
\ben
(H_\mho f)=-\int_{\pa B}\frac{\pa\Psi_k(x,y)}{\pa\nu(y)}f(y)ds(y),\quad x\in\mho.
\enn
Then $H_\mho:L^2(\pa B)\ra L^2(\mho)$ is bounded, compact, and injective.


(b) If ${\rm Re}\,n>1$ or ${\rm Re}\,n<1$ in $\ov{\mho}$ and ${\rm Im}\,n>0$ or ${\rm Im}\,n<0$ in $\ov{\mho}$, then $S_\mho:L^2(\mho)\ra L^2(\mho)$ is bounded.
Define $S_{\mho,0}:L^2(\mho)\ra L^2(\mho)$ by
\ben
(S_{\mho,0}\psi)(x):=-\frac{\psi}{1-n},\quad x\in\mho.
\enn
Then ${\rm sign}({\rm Re}\,n-1){\rm Re}\,S_{\mho,0}$ is coercive, i.e., there exists a constant $c>0$ such that
\ben
{\rm sign}({\rm Re}\,n-1){\rm Re}\langle S_{\mho,0}\psi,\ov{\psi}\rangle\geq c\left\|\psi\right\|_{L^2(\mho)}^2\quad\text{for all }\psi\in L^2(\mho),
\enn
and $S_\mho-S_{\mho,0}:L^2(\mho)\ra L^2(\mho)$ is compact.
Moreover, it holds that
\ben
-{\rm sign}({\rm Im}\,n){\rm Im}\langle S_\mho\psi,\ov{\psi}\rangle>0\quad\text{for all }\psi\in L^2(\mho)\ba\{0\}.
\enn
\end{theorem}

The coercivity of the boundary integral operator $S_\mho$ plays an important role in the factorization method.
It should be remarked that we consider a medium with refractive index $n$ satisfying ${\rm Im}\,n>0$ or ${\rm Im}\,n<0$ in $\ov{\mho}$ since, in general, $S_\mho$ is not coercive when $n$ is real-valued, even if $k^2$ is not a corresponding eigenvalue.
Moreover, analogously to  Remark \ref{rem-} (vi), if ${\rm Im}\,n>0$ or ${\rm Im}\,n<0$ in $\ov{\mho}$ then $k^2$ is not an eigenvalue of \eqref{m1}--\eqref{m2}.

{\bl Using the results of previous theorems and proceeding} as in \cite{CK19,Kirsch08}, {\x we can obtain the symmetric factorizations of $A(k^2,\Om)-A_0(k^2)$ and $A(k^2,\mho)-A_0(k^2)$ as follows.}

\begin{theorem}\label{t3.5}
We have
\be\label{f-1}
A(k^2,\Om)-A_0(k^2)=G(k^2,\Om)T_{imp}^*[G(k^2,\Om)]^*,\\ \label{f-2}
A(k^2,\mho)-A_0(k^2)=k^2G(k^2,\mho)S_\mho^*[G(k^2,\mho)]^*,
\en
where the superscript $*$ denotes the adjoint operator.
\end{theorem}

\begin{remark}\label{0915}
(i) It follows from Remark \ref{rem-} (ii) and (iii) that $A(k^2,\Om)-A_0(k^2):H^s(\pa B)\ra H^{s-1}(\pa B)$ is bounded for any $s\in[-\frac12,\frac12]$.
Further, since $v_1=v-u_0$ solves (\ref{7})--(\ref{9}) with $g=-(\pa_\nu u_0+\eta u_0)$ on $\pa\Om$, we conclude from Theorem \ref{t3.3} (a) that $A(k^2,\Om)-A_0(k^2):H^s(\pa B)\ra H^{m-3/2}(\pa B)$ is bounded for any $s\in[-\frac12,\frac12]$ and $m\in\Z_+$.
In particular, $A(k^2,\Om)-A_0(k^2):L^2(\pa B)\ra L^2(\pa B)$ is bounded.

(ii) Similarly, it follows from Remark \ref{rem-} (ii) and (iv) that $A(k^2,\mho)-A_0(k^2):H^s(\pa B)\ra H^{s-1}(\pa B)$ is bounded for any $s\in[-\frac12,\frac12]$.
Further, since $w_1=w-u_0$ solves \eqref{G-1}--\eqref{G-2} with $\rho=k^2u_0$ in $\mho$, we conclude from Theorem \ref{t3.3} (b) that $A(k^2,\mho)-A_0(k^2):H^s(\pa B)\ra H^{m-3/2}(\pa B)$ is bounded for any $s\in[-\frac12,\frac12]$ and $m\in\Z_+$.
In particular, $A(k^2,\mho)-A_0(k^2):L^2(\pa B)\ra L^2(\pa B)$ is bounded.

(iii) $H:L^2(\pa B)\ra H^{-1/2}(\pa\Om)$ and $H_\mho:L^2(\pa B)\ra L^2(\mho)$ are compact.

(iv) $G(k^2,\Om):H^{-1/2}(\pa\Om)\ra L^2(\pa B)$ and $G(k^2,\mho):L^2(\mho)\ra L^2(\pa B)$ have dense ranges.
\end{remark}

For $A\in L^2(\pa B)\ra L^2(\pa B)$ define $A_\#:=\left|{\rm Re}A\right|+\left|{\rm Im}A\right|$, where ${\rm Re}A=(A+A^*)/2$ and ${\rm Im}A=(A-A^*)/(2i)$.
The following theorem on ``range identity'' is a direct corollary of Theorems \ref{t3.3}, \ref{t+}, \ref{t+m} and \ref{t3.5}, Remark \ref{0915} (iv) and \cite[Theorem 3.3]{Kirsch05} (see also \cite[Theorem 2.15]{Kirsch08}).

\begin{theorem}\label{thm3.8}
Assume that $k^2>0$ is not an eigenvalue of (\ref{5})--(\ref{6}).

(i) If ${\rm Im}\,\eta>0$ or ${\rm Im}\,\eta<0$ on $\pa\Om$, then the ranges of $G(k^2,\Om)$ and $[A(k^2,\Om)-A_0(k^2)]_\#^{1/2}$ coincide, i.e., ${\rm Ran}\,G(k^2,\Om)={\rm Ran}\,[A(k^2,\Om)-A_0(k^2)]_\#^{1/2}$.

(ii) If ${\rm Re}\,n>1$ or ${\rm Re}\,n<1$ in $\ov{\mho}$ and ${\rm Im}\,n>0$ or ${\rm Im}\,n<0$ in $\ov{\mho}$, then the ranges of $G(k^2,\mho)$ and $[A(k^2,\mho)-A_0(k^2)]_\#^{1/2}$ coincide, i.e., ${\rm Ran}\,G(k^2,\mho)={\rm Ran}\,[A(k^2,\mho)-A_0(k^2)]_\#^{1/2}$.
\end{theorem}

For a numerical implementation of factorization method, we need the following theorem.
We omit the proof since similar results can be found in \cite{CK19,Kirsch08} and \cite{XH24}.

\begin{theorem}\label{zB}
Let $z\in B$.
(i) $\pa_\nu\Psi_k(\cdot,z)|_{\pa B}\in{\rm Ran}\,G(k^2,\Om)$ if and only if $z\in\Om$.

(ii) $\pa_\nu\Psi_k(\cdot,z)|_{\pa B}\in{\rm Ran}\,G(k^2,\mho)$ if and only if $z\in\mho$.
\end{theorem}

\subsection{Revisited factorizations}

In the sequel, we will determine the value of $k^2=q/\sigma$ from a single pair of Cauchy data by a sampling scheme, {\bl and also calculate}
${\rm Ran}\,G(k^2,\Om)$ and ${\rm Ran}\,G(k^2,\mho)$ for different values of $k^2$.
This section is devoted to represent ${\rm Ran}\,G(k^2,\Om)$ and ${\rm Ran}\,G(k^2,\mho)$ in terms of Dirichlet-to-Neumann maps, without
the assumption that $k^2$ is not an eigenvalue of (\ref{5})--(\ref{6}).
Motivated by the idea of an auxiliary interior disk (see \cite[$\widetilde D$ in Figure 2.1]{QZZ} and \cite{cbj}), we introduce another two Green functions, instead of $\Psi_k(x,y)$, as follows.

Let $\widetilde B$ be a bounded domain such that $\ov{\widetilde B}\subset\Om$.
For any $y\in B\ba\ov{\widetilde B}$, consider the Green function $\widetilde\Psi_k(x,y)=\Phi_k(x,y)+\tilde\psi_k(x,y)$ where $\psi_k(x,y)$ solves
\be\label{tilde1}
\Delta_x\tilde\psi_k(x,y)+k^2\tilde\psi_k(x,y)=0 & \text{in }B\ba\ov{\widetilde B},\\ \label{tilde2}
\tilde\psi_k(\cdot,y)=-\Phi_k(\cdot,y) & \text{on }\pa B,\\ \label{tilde3}
\tilde\psi(\cdot,y)=-\Phi_k(\cdot,y) & \text{on }\pa\widetilde B{\bl .}
\en
Here, $\widetilde B$ is chosen appropriately such that $k^2$ is not an eigenvalue of \eqref{tilde1}--\eqref{tilde3}.
Actually, we may assume that $\widetilde B$ is a disk contained in the interior of $\Om$.
According to the strong monotonicity property of Dirichlet eigenvalues (see \cite[Theorem 5.2]{CK19}), we can find such a disk $\widetilde B$ by a minor adjustment on its radius.
Analogously to Lemma \ref{lem211025} (i) and Remark \ref{rem3.4}, it is easy to verify that $\widetilde\Psi_k(x,y)$ is {\bl real valued} and symmetric.

Define boundary integral operators $\widetilde S$, $\widetilde K$, $\widetilde K'$ and $\widetilde T$ analogous to (\ref{bi-1})--(\ref{bi-4}) with $\Psi_k$ replaced by $\widetilde\Psi_k$.
Moreover, analogous to $T_{imp}$, define $\widetilde T_{imp}:=\widetilde T+i{\rm Im}\,\eta I+\widetilde K'\ov{\eta}+\eta\widetilde K+\eta\widetilde S\ov{\eta}$.

Proceeding as in the proofs of Theorems \ref{t+}, \ref{t3.5} and \ref{thm3.8}, we have
\begin{theorem}
Let $\widetilde B$ be given as above.
Assume ${\rm Im}\,\eta>0$ or ${\rm Im}\,\eta<0$ on $\pa\Om$.

(i) $\widetilde T_{imp}:H^{1/2}(\pa\Om)\ra H^{-1/2}(\pa\Om)$ is bounded and $(\widetilde T_{imp}-T_i):H^{1/2}(\pa\Om)\ra H^{-1/2}(\pa\Om)$ is compact.
Moreover, it holds that
\ben
{\rm sign}({\rm Im}\,\eta){\rm Im}\langle\widetilde T_{imp}\varphi,\ov{\varphi}\rangle>0\quad\text{for all }\varphi\in H^{1/2}(\pa\Om)\ba\{0\}.
\enn

(ii) It holds that
\be\label{220614-2}
A(k^2,\Om)-\widetilde A_0(k^2)=G(k^2,\Om)\widetilde T_{imp}^*[G(k^2,\Om)]^*,
\en
where $\widetilde A_0(k^2)$ is defined as $A(k^2,D)$ with $D$ replaced by $\widetilde B$.

(iii) ${\rm Ran}\,G(k^2,\Om)={\rm Ran}\,[A(k^2,\Om)-\widetilde A_0(k^2)]_\#^{1/2}$.
\end{theorem}

\begin{proof}
We only prove assertion (ii).
Analogous to Lemma \ref{lem211025} (ii), for $f\in H^s(\pa B)$ with $s\in[-\frac12,\frac12]$, the solution $\tilde u_0\in L^2_\Delta(B\ba\ov{\widetilde B})$ to the following boundary value problem
\ben
\Delta\tilde u_0+k^2\tilde u_0=0 && \text{in }B\ba\ov{\widetilde B},\\
\tilde u_0=f && \text{on }\pa B,\\
\tilde u_0=0 && \text{on }\pa\widetilde B,
\enn
can be represented as
\ben
\tilde u_0(x)=-\int_{\pa B}\frac{\pa\widetilde\Psi_k(x,y)}{\pa\nu(y)}f(y)ds(y),\quad x\in B\ba\ov{\widetilde B}.
\enn
Analogous to Theorem \ref{t+} (a), define $\widetilde H:L^2(\pa B)\to H^{-1/2}(\pa\Om)$ by
\ben
(\widetilde Hf)(x):=-\left(\frac{\pa}{\pa\nu(x)}+\eta\right)\int_{\pa B}\frac{\pa\widetilde\Psi_k(x,y)}{\pa\nu(y)}f(y)ds(y),\quad x\in\pa\Om.
\enn
Now we have
\be\label{12tilde}
[A(k^2,\Om)-\widetilde A_0(k^2)]f=G(k^2,\Om)[-(\pa_\nu\tilde u_0+\eta\tilde u_0)|_{\pa\Om}]=-G(k^2,\Om)\widetilde Hf.
\en
Since $\widetilde\Psi_k(x,y)$ is real-valued, the $L^2$ adjoint $\widetilde H^*:H^{1/2}(\pa\Om)\ra L^2(\pa B)$ is given by
\ben
(\widetilde H^*\varphi)(x):=-\frac{\pa}{\pa\nu(x)}\int_{\pa\Om}\left(\frac{\pa}{\pa\nu(y)}+\ov{\eta}\right)\widetilde\Psi_k(x,y)\varphi(y)ds(y),\quad x\in\pa B,
\enn
which represents the Neumann boundary value $\pa_\nu v$ on $\pa B$ with $v$ defined by
\ben
v(x)=-\int_{\pa\Om}\left(\frac{\pa}{\pa\nu(y)}+\ov{\eta}\right)\widetilde\Psi_k(x,y)\varphi(y)ds(y),\quad x\in B\ba\ov{\Om}.
\enn
By jump relations, we have $\widetilde H^*=-G(k^2,\Om)\widetilde T_{imp}$ since $v$ solves (\ref{7})--(\ref{9}) with $g=-\widetilde T_{imp}\varphi$, and consequently
\be\label{13tilde}
\widetilde H=-\widetilde T_{imp}^*[G(k^2,\Om)]^*.
\en
Now, the factorization (\ref{220614-2}) follows by combining (\ref{12tilde}) and (\ref{13tilde}).
\end{proof}

Note that it is not convenient to check whether $k^2$ is a Dirichlet eigenvalue of \eqref{tilde1}--\eqref{tilde3}.
To overcome this difficulty, we can {\bl make use of} an artificial impedance obstacle $\widetilde\Om$ such that $\ov{\widetilde\Om}\subset\Om$.
\begin{theorem}\label{thm3.8_ID}
Assume ${\widetilde\Om}\subset\Om$ and the impedance coefficients $\tilde\eta\in C(\pa\widetilde\Om),\eta\in C(\pa\Om)$ satisfy one of the following assumptions:

(i) ${\rm Im}\,\tilde\eta<0$ on $\pa\widetilde\Om$ and ${\rm Im}\,\eta>0$ on $\pa\Om$;

(ii) ${\rm Im}\,\tilde\eta>0$ on $\pa\widetilde\Om$ and ${\rm Im}\,\eta<0$ on $\pa\Om$.

Then we have
\be\label{220614-1}
{\rm Ran}\,G(k^2,\Om)={\rm Ran}[A(k^2,\Om)-A(k^2,\widetilde\Om)]_\#^{1/2},
\en
where $A(k^2,\widetilde\Om)$ is defined as $A(k^2,\Om)$ with $(\Om,\eta)$ replaced by $(\widetilde\Om,\tilde\eta)$.
\end{theorem}
\begin{proof}
We only consider the case when assumption (i) is satisfied since the proof is similar.
Choose a proper domain $\widetilde B$ such that $\widetilde B\subset\ov{\widetilde B}\subset\widetilde\Om\subset\ov{\widetilde\Om}\subset\Om\subset\ov{\Om}\subset B$ and \eqref{tilde1}--\eqref{tilde3} is uniquely solvable.
It follows from (\ref{220614-2}) that
\ben
A(k^2,\Om)-\widetilde A_0(k^2)=G(k^2,\Om)[\widetilde T_{imp}(k^2,\pa\Om)]^*[G(k^2,\Om)]^*,\\
A(k^2,\widetilde\Om)-\widetilde A_0(k^2)=G(k^2,\widetilde\Om)[\widetilde T_{imp}(k^2,\pa\widetilde\Om)]^*[G(k^2,\widetilde\Om)]^*,
\enn
where the operators are defined analogously.
Note that $G(k^2,\widetilde\Om)=G(k^2,\Om)R_{\pa\widetilde\Om\ra\pa\Om}$, where $R_{\pa\widetilde\Om\ra\pa\Om}$ is defined by $R_{\pa\widetilde\Om\ra\pa\Om}g=\pa_\nu w+\eta w$ on $\pa\Om$ with $w$ solving the following boundary value problem
\ben
\Delta w+k^2w=0 & \text{in}\;B\ba\ov{\widetilde\Om},\\
w=0 & \text{on}\;\pa B,\\
\pa_\nu w+\tilde\eta w=g & \text{on}\;\pa\widetilde\Om.
\enn
Moreover, $R_{\pa\widetilde\Om\ra\pa\Om}:H^{-1/2}(\pa\widetilde\Om)\ra H^{-1/2}(\pa\Om)$ is compact due to the regularity results for elliptic differential equations.
We have
\ben
A(k^2,\Om)\!\!-\!\!A(k^2,\widetilde\Om)\!\!=\!\!G(k^2,\Om)\{[\widetilde T_{imp}(k^2,\pa\Om)]^*\!\!-\!\!R_{\pa\widetilde\Om\ra\pa\Om}[\widetilde T_{imp}(k^2,\pa\widetilde\Om)]^*R_{\pa\widetilde\Om\ra\pa\Om}^*\}[G(k^2,\Om)]^*.
\enn
Note that ${\rm Im}\,\eta>0$ on $\pa\Om$ implies ${\rm Im}\langle\widetilde T_{imp}(k^2,\pa\Om)\varphi,\varphi\rangle>0$ for all $\varphi\in H^{1/2}(\pa\Om)$ with $\varphi\neq0$ and ${\rm Im}\,\tilde\eta<0$ on $\pa\widetilde\Om$ implies ${\rm Im}\,\langle\widetilde T_{imp}(k^2,\pa\widetilde\Om)\tilde\varphi,\tilde\varphi\rangle>0$ for all $\tilde\varphi\in H^{1/2}(\pa\widetilde\Om)$ with $\tilde\varphi\neq0$, respectively.
Analogously to Theorem \ref{thm3.8}, it can be shown that \eqref{220614-1} holds.
\end{proof}

We now turn to the case that the impedance obstacle $\Om$ is replaced by a medium with the refractive index function $n$ given as in \eqref{m1}--\eqref{m2}.
For any $y\in B\ba\ov{\widetilde B}$, we may consider the Green function $\widetilde\Psi'_k(x,y)=\Phi_k(x,y)+\tilde\psi'_k(x,y)$ where $\tilde\psi'_k(x,y)$ solves
\be\label{tilde1m}
\Delta_x\tilde\psi'_k(x,y)+k^2\check n\tilde\psi'_k(x,y)=k^2(1-\check n)\Phi_k(x,y) & \text{in }B,\\ \label{tilde2m}
\tilde\psi'_k(\cdot,y)=-\Phi_k(\cdot,y) & \text{on }\pa B,
\en
where the real-valued refractive index $\check n\in L^\infty(B)$ is appropriately chosen such that ${\rm supp}\,(\check n-1)=:\ov{\widetilde{B}}\subset\mho\subset\ov{\mho}=:{\rm supp}\,(n-1)$ and the above boundary value problem is uniquely solvable, i.e., $k^2$ is not an eigenvalue of \eqref{tilde1m}--\eqref{tilde2m}.

To show that $\widetilde\Psi'_k(x,y)$ is symmetric, it suffices to show $\tilde\psi'_k(x,y)=\tilde\psi'_k(y,x)$ for all $x,y\in B$, which can be justified by an argument similar to the proof of Lemma \ref{lem211025} (i).
To show that $\widetilde\Psi'_k(x,y)$ is real-valued.
For any fixed $y\in B$, we consider $\tilde u(x,y):=\frac{i}2J_0(k|x-y|)+\widetilde\psi'_k(x,y)-\ov{\widetilde\psi'_k(x,y)}$.
Then $\tilde u(x,y)=\widetilde\Psi'_k(x,y)-\ov{\widetilde\Psi'_k(x,y)}$ for $x\neq y$.
Since $\check n$ is real-valued, $\tilde u(\cdot,y)$ solves
\ben
\Delta\tilde u+k^2\check n\tilde u=0 \quad \text{in }B,\qquad
\tilde u=0 \quad \text{on }\pa B.
\enn
Since $k^2$ is not an eigenvalue of \eqref{tilde1m}--\eqref{tilde2m}, we have $\widetilde\Psi_k(x,y)-\ov{\widetilde\Psi_k(x,y)}=\tilde u(x,y)=0$.

Analogous to $S_\mho$ and $S_{\mho,0}$, define $\widetilde S_\mho,\widetilde S_{\mho,0}:L^2(\mho)\ra L^2(\mho)$ by
\ben
(\widetilde S_{\mho}\psi)(x)=-\frac\psi{\check n-n}-k^2\int_{\mho}\widetilde\Psi_k(x,y)\psi(y)dy,\;\;\widetilde S_{\mho,0}\psi=-\frac\psi{\check n-n},\quad x\in\mho,
\enn
and $\widetilde G(k^2,\mho):L^2(\mho)\ra L^2(\pa B)$ by $\widetilde G(k^2,\mho)f=\pa_\nu w|_{\pa B}$ where $w$ solves
\ben
\Delta w+k^2nw=(\check n-n)f \quad \text{in }B,\qquad
w=0 \quad \text{on }\pa B.
\enn
Analogously to Theorems \ref{t+m}, \ref{t3.5} and \ref{thm3.8}, we have
\begin{theorem}\label{thm0816}
Let $\check n$ be given as above.
Assume ${\rm Re}\,(\check n-n)>0$ or ${\rm Re}\,(\check n-n)<0$ in $\ov{\mho}$ and ${\rm Im}\,n>0$ or ${\rm Im}\,n<0$ in $\ov{\mho}$.

(i) $\widetilde S_{\mho}:L^2(\mho)\ra L^2(\mho)$ is bounded and $\widetilde S_\mho-\widetilde S_{\mho,0}:L^2(\mho)\ra L^2(\mho)$ is compact.
${\rm sign}({\rm Re}\,(n-\check n)){\rm Re}\,\widetilde S_{\mho,0}$ is coercive, i.e., there exists a constant $c>0$ such that
\ben
{\rm sign}({\rm Re}(n-\check n)){\rm Re}\,\langle\widetilde S_{\mho,0}\psi,\ov{\psi}\rangle\geq c\left\|\psi\right\|_{L^2(\mho)}^2\quad\text{for all }\psi\in L^2(\mho),
\enn
Moreover, it holds that
\ben
-{\rm sign}({\rm Im}\,n){\rm Im}\langle\widetilde S_\mho\psi,\ov{\psi}\rangle>0\quad\text{for all }\psi\in L^2(\mho)\ba\{0\}.
\enn

(ii) It holds that
\be\label{220622-1}
A(k^2,\mho)-\widetilde A'_0(k^2)=k^2\widetilde G(k^2,\mho)\widetilde S_{\mho}^*[\widetilde G(k^2,\mho)]^*,
\en
where $\widetilde A'_0(k^2)$ is defined as $A(k^2,\mho)$ with $(n,\mho)$ replaced by $(\check n,\widetilde B)$.

(iii) ${\rm Ran}\,G(k^2,\mho)={\rm Ran}\,\widetilde G(k^2,\mho)={\rm Ran}[A(k^2,\mho)-\widetilde A'_0(k^2)]_\#^{1/2}$.
\end{theorem}
\begin{proof}
(i) The assertion follows easily from the assumption on $\check n$ and $n$.

(ii) Note that the solution to (\ref{m1})--(\ref{m2}) has the decomposition $w=\tilde u_0+\tilde w$ where $\tilde u_0$ solves
\ben
\Delta\tilde u_0+k^2\check n\tilde u_0=0 \quad \text{in }B,\qquad
\tilde u_0=f \quad \text{on }\pa B,
\enn
and $\tilde w$ solves
\ben
\Delta\tilde w+k^2n\tilde w=(\check n-n)k^2\tilde u_0 \quad \text{in }B,\qquad
\tilde w=0 \quad \text{on }\pa B.
\enn
Analogous to \eqref{211025-1}, the solution $\tilde u_0$ can be represented as
\ben
\tilde u_0(x)=-\int_{\pa B}\frac{\pa\widetilde \Psi'_k(x,y)}{\pa\nu(y)}f(y)ds(y),\quad x\in B.
\enn
Therefore,
\be\label{240816-1}
[A(k^2,\mho)-\widetilde A'_0(k^2)]f=\pa_\nu w=\widetilde G(k^2,\mho)(k^2\tilde u_0)=k^2\widetilde G(k^2,\mho)\widetilde H_\mho f,
\en
where $\widetilde H_\mho:L^2(\pa B)\ra L^2(\Om)$ is defined by
\ben
(\widetilde H_\mho f)=-\int_{\pa B}\frac{\pa\widetilde\Psi'_k(x,y)}{\pa\nu(y)}f(y)ds(y),\quad x\in\mho.
\enn
Since $\widetilde\Psi'_k(x,y)$ is real-valued, a straightforward calculation shows that
\ben
(\widetilde H^*_\mho\psi)(x)=-\frac{\pa}{\pa\nu(x)}\int_\mho\widetilde\Psi'_k(x,y)\psi(y)dy,\quad x\in\pa B.
\enn
Define
\ben
w=\int_\mho\widetilde\Psi_k(x,y)\psi(y)dy,\quad x\in\pa B.
\enn
We can deduce from \cite[Theorems 8.1 and 8.2]{CK19} and \eqref{tilde1m} that
\ben
(\Delta+k^2\check n)w&=&(\Delta+k^2\check n)\int_\mho\Phi_k(x,y)\psi(y)dy+(\Delta+k^2\check n)\int_\mho\widetilde\psi_k(x,y)\psi(y)dy\\
&=&-\psi+k^2(\check n-1)\int_\mho\Phi_k(x,y)\psi(y)dy+k^2(1-\check n)\int_\mho\Phi_k(x,y)\psi(y)dy\\
&=&-\psi.
\enn
Moreover, it follows from \eqref{tilde2m} that $w=0$ on $\pa B$.
Therefore,
\ben
\Delta w+k^2nw=-\psi+k^2(n-\check n)w=(\check n-n)[-\psi/(\check n-n)-k^2w]\quad\text{in }B.
\enn
This implies $\widetilde H_\mho^*=\widetilde G(k^2,\mho)\widetilde S_\mho$.
By \eqref{240816-1}, we arrive at \eqref{220622-1}.

(iii) Note that $\widetilde G(k^2,\mho)f=G(k^2,\mho)\left(\frac{\check n-n}{1-n}f\right)$ for all $f\in L^2(\mho)$, we conclude from the assumptions on $\check n$ and $n$ that ${\rm Ran}\,\widetilde G(k^2,\mho)={\rm Ran}\,G(k^2,\mho)$.
The remaining part of the proof is similar to that of Theorem \ref{thm3.8}.
\end{proof}

Note that it is not convenient to check whether $k^2$ is a Dirichlet eigenvalue of \eqref{tilde1m}--\eqref{tilde2m}.
To overcome this difficulty, we can put an artificial medium with refractive index $\tilde n$ such that ${\rm supp}(\tilde n-1):=\ov{\widetilde\mho}\subset\mho\subset\ov{\mho}=:{\rm supp}(n-1)$.

\begin{theorem}\label{thm3.8m_ID}
Assume ${\widetilde\mho}\subset\mho$ and the refractive indices $n,\tilde n\in L^\infty(B)$ satisfy either (i) or (ii), and either (iii) or (iv), where the assumptions are:

(i) ${\rm Re}\,n>1$ in $\ov{\mho}$ and ${\rm Re}\,\tilde n<1$ in $\ov{\widetilde\mho}$;

(ii) ${\rm Re}\,n<1$ in $\ov{\mho}$ and ${\rm Re}\,\tilde n>1$ in $\ov{\widetilde\mho}$;

(iii) ${\rm Im}\,n>0$ in $\ov{\mho}$ and ${\rm Im}\,\tilde n<0$ in $\ov{\widetilde\mho}$;

(iv) ${\rm Im}\,n<0$ in $\ov{\mho}$ and ${\rm Im}\,\tilde n>0$ in $\ov{\widetilde\mho}$.

Then we have
\be\label{220623-1}
{\rm Ran}\,G(k^2,\mho)={\rm Ran}\,\widetilde G(k^2,\mho)={\rm Ran}[A(k^2,\mho)-A(k^2,\widetilde\mho)]_\#^{1/2},
\en
where $A(k^2,\widetilde\mho)$ is defined as $A(k^2,\mho)$ with $(\mho,n)$ replaced by $(\widetilde\mho,\tilde n)$.
\end{theorem}
\begin{proof}
We only consider the case when assumptions (i) and (iii) are satisfied since the proofs for other cases are similar.
Choose a proper real-valued refractive index $\check n$ such that ${\rm supp}(\check n-1):=\ov{\widetilde B}\subset\widetilde\Om\subset\ov{\widetilde\Om}\subset\Om\subset\ov{\Om}\subset B$, and \eqref{tilde1m}--\eqref{tilde2m} is uniquely solvable, and ${\rm Re}\,(\check n-n)<0$ in $\ov{\Om}$ and ${\rm Re}\,(\check n-\tilde n)>0$ in $\ov{\widetilde\mho}$.
It follows from (\ref{220622-1}) that
\ben
A(k^2,\mho)-\widetilde A'_0(k^2)=k^2\widetilde G(k^2,\mho)\widetilde S_\mho^*[\widetilde G(k^2,\mho)]^*,\\
A(k^2,\widetilde\mho)-\widetilde A'_0(k^2)=k^2\widetilde G(k^2,\widetilde\mho)\widetilde S_{\widetilde\mho}^*[\widetilde G(k^2,\widetilde\mho)]^*.
\enn
Note that $G(k^2,\widetilde\mho)=G(k^2,\mho)R_{\widetilde\mho\ra\mho}$, where $R_{\widetilde\mho\ra\mho}$ is defined by $R_{\widetilde\mho\ra\mho}g=Eg+k^2(n-\tilde n)w$ in $\mho$ with $Eg$ denoting the extension of $g$ by zero from $\widetilde\mho$ into $\mho$ and $w$ solving
\be\label{220622-2}
\Delta w+k^2\tilde nw=g \quad \text{in}\;B,\qquad
w=0 \quad \text{on}\;\pa B.
\en
Note that
\ben
\widetilde G(k^2,\widetilde\mho)f\!\!\!\!&=&\!\!\!\!G(k^2,\widetilde\mho)\left(\frac{\check n-\tilde n}{1-\tilde n}f\right)=G(k^2,\mho)R_{\widetilde\mho\ra\mho}\left(\frac{\check n-\tilde n}{1-\tilde n}f\right)\\
&=&\!\!\!\!\widetilde G(k^2,\mho)\left(\frac{1-n}{\check n-n}R_{\widetilde\mho\ra\mho}\left(\frac{\check n-\tilde n}{1-\tilde n}f\right)\right)
\enn
for all $f\in L^2(\mho)$.
By defining $\widetilde R_{\widetilde\mho\ra\mho}g=\frac{1-n}{\check n-n}R_{\widetilde\mho\ra\mho}\left(\frac{\check n-\tilde n}{1-\tilde n}g\right)$ for all $g\in L^2(\widetilde\mho)$, we have
\ben
A(k^2,\mho)-A(k^2,\widetilde\mho)=k^2\widetilde G(k^2,\mho)(\widetilde S_\mho^*-\widetilde R_{\widetilde\mho\ra\mho}\widetilde S_{\widetilde\mho}^*\widetilde R_{\widetilde\mho\ra\mho}^*)[\widetilde G(k^2,\mho)]^*.
\enn
By Theorem \ref{thm0816} (i), we know $(\widetilde S_\mho-\widetilde R_{\widetilde\mho\ra\mho}\widetilde S_{\widetilde\mho}\widetilde R_{\widetilde\mho\ra\mho}^*)-(\widetilde S_{\mho,0}-\widetilde R_{\widetilde\mho\ra\mho}\widetilde S_{\widetilde\mho,0}\widetilde R_{\widetilde\mho\ra\mho}^*):L^2(\mho)\to L^2(\mho)$ is compact.
Moreover, it follows from the assumptions (i) and (iii) that there exists a constant $c>0$ such that
\ben
{\rm Re}\left\langle(\widetilde S_{\mho,0}-\widetilde R_{\widetilde\mho\ra\mho}\widetilde S_{\widetilde\mho,0}\widetilde R_{\widetilde\mho\ra\mho}^*)\psi,\ov{\psi}\right\rangle\geq c\|\psi\|^2_{L^2(\mho)}\quad\text{for all }\psi\in L^2(\mho),\\
-{\rm Im}\left\langle(\widetilde S_\mho-\widetilde R_{\widetilde\mho\ra\mho}\widetilde S_{\widetilde\mho}\widetilde R_{\widetilde\mho\ra\mho}^*)\psi,\ov{\psi}\right\rangle>0\quad\text{for all }\psi\in L^2(\mho)\ba\{0\}.
\enn
Combining Theorem \ref{t3.3} (b), Theorem \ref{thm0816} (i),(iii) and Remark \ref{0915} (iv), \eqref{220623-1} follows from \cite[Theorem 3.3]{Kirsch05} (see also \cite[Theorem 2.15]{Kirsch08}).
\end{proof}

For a numerical implementation of factorization method, we need the following theorem analogous to Theorem \ref{zB}.
We omit the proof since it is similar to that of Theorem \ref{zB}.

\begin{theorem}\label{zB_ID}
Let $z\in B$.
(i) $\pa_{\nu}\widetilde\Psi_k(\cdot,z)|_{\pa B}\in{\rm Ran}\,G(k^2,\Om)$ if and only if $z\in\Om$.

(ii) $\pa_{\nu}\widetilde\Psi_k(\cdot,z)|_{\pa B}\in{\rm Ran}\,G(k^2,\mho)$ if and only if $z\in\mho$.
\end{theorem}

\section{Determination of the two coefficients}\label{s3}
\setcounter{equation}{0}

We retain the notations introduced in previous sections.
Below we show that $(\sigma,q)$ can be recovered from a single pair of Cauchy data under some a priori assumptions.
{\bl The following uniqueness results for determining coefficients remain valid even if $D$ is a Lipschitz domain contained in $B$.}

\begin{theorem}\label{thm-1}
Assume that neither $q/\sigma$ nor $\kappa^2$ is an eigenvalue of \eqref{5}--\eqref{6}.
Let $u$ solve \eqref{-1}--\eqref{-3} and $u_0(\cdot,\kappa^2)$ solve \eqref{5}--\eqref{6} with $k=\kappa$ for the same boundary value $f$ on $\pa B$.
Assume that $D\subset\Omega\subset\ov{\Om}\subset B$.
For $\tau\in\C$ and $\kappa\in\C$ define $p(\tau,\kappa^2):=\sigma\pa_\nu u-\tau\pa_\nu u_0(\cdot,\kappa^2)$ on $\pa B$.
Then the following statements are true:

(i) If $(\tau,\kappa^2)=(\sigma,q/\sigma)$, then $p(\tau,\kappa^2)\in{\rm Ran}\,G(\kappa^2,\Om)$.

(ii) $p(\tau,\kappa^2)\in{\rm Ran}\,G(\kappa^2,\Om)$ if and only if $[\sigma A_0(q/\sigma)-\tau A_0(\kappa^2)]f\in{\rm Ran}\,G(\kappa^2,\Om)$.
\end{theorem}

\begin{proof}
(i) If $(\tau,\kappa^2)=(\sigma,q/\sigma)$, then $p(\tau,\kappa^2)=p(\sigma,q/\sigma)=\sigma\pa_\nu(u-u_0(\cdot,q/\sigma))|_{\pa B}$.
Noting that $u-u_0(\cdot,q/\sigma)$ satisfies the Helmholtz equation in $B\ba\ov{D}$ and $u-u_0(\cdot,q/\sigma)=f-f=0$ on $\pa B$, we have $p(\tau,\kappa^2)=G(\kappa^2,\Om)g$ with $g=(\pa_\nu+\eta)[\sigma(u-u_0(\cdot,q/\sigma))]$ on $\pa\Om$.

(ii) Since $p(\sigma,q/\sigma)\in{\rm Ran}\,G(q/\sigma,\Om)$, the assertion follows easily from
\ben
p(\tau,\kappa^2)-p(\sigma,q/\sigma)&=&(\sigma\pa_\nu u-\tau\pa_\nu u_0(\cdot,\kappa^2))-(\sigma\pa_\nu u-\sigma\pa_\nu u_0(\cdot,q/\sigma))\\
&=&\sigma\pa_\nu u_0(\cdot,q/\sigma)-\tau\pa_\nu u_0(\cdot,\kappa^2)\\
&=&[\sigma A_0(q/\sigma)-\tau A_0(\kappa^2)]f.
\enn
\end{proof}

In view of Theorem \ref{thm-1}, if the boundary value $f$ on $\pa B$ satisfies the condition $[\sigma A_0(q/\sigma)-\tau A_0(\kappa^2)]f\notin{\rm Ran}\,G(\kappa^2,\Om)$ whenever $(\tau,\kappa^2)\neq(\sigma,q/\sigma)$, then $(\sigma,q)$ can be recovered by checking whether $p(\tau,\kappa^2)$ belongs to ${\rm Ran}\,G(\kappa^2,\Om)$ if neither $q/\sigma$ nor $\kappa^2$ is an eigenvalue of \eqref{5}--\eqref{6}.
For convenience of numerical implementation, we provide an explicit example of the boundary value $f$ on $\pa B$ such that $[\sigma A_0(q/\sigma)-\tau A_0(\kappa^2)]f\notin{\rm Ran}\,G(\kappa^2,\Om)$ whenever $(\tau,\kappa^2)\neq(\sigma,q/\sigma)$.

\begin{theorem}\label{53-thm1}
Assume that neither $q/\sigma$ nor $\kappa^2$ is an eigenvalue of \eqref{5}--\eqref{6} and $f\in H^s(\pa B)\ba H^{s+\varepsilon}(\pa B)$ for any $s\geq-\frac12$ and $\varepsilon>0$.
Let $u$ solve \eqref{-1}--\eqref{-3} and $u_0(\cdot,\kappa^2)$ solve \eqref{5}--\eqref{6} with $k=\kappa$.
Assume $D\subset\Omega\subset\ov{\Om}\subset B$.
If $(\tau,\kappa^2)\neq(\sigma,q/\sigma)$, then $[\sigma A_0(q/\sigma)-\tau A_0(\kappa^2)]f\notin{\rm Ran}\,G(\kappa^2,\Om)$.
\end{theorem}
\begin{proof}
Without loss of generality, we may assume that $B$ is centered at the origin.
We first note that $f\in H^{s}(\pa B)$ can be expanded into the series
\be\label{53-1}
f(x)=\sum_{n\in\Z}c_ne^{in\theta},\quad x\in\pa B,
\en
where $\theta\in[0,2\pi)$ is given by $x/|x|=(\cos\theta,\sin\theta)$ and
\be\label{220604-3}
\left\|f\right\|_{H^s(\pa B)}^2=\sum_{n\in\Z}(n^2+1)^s|c_n|^2.
\en
Proceeding as in the proof of \cite[Theorem 2.33]{KH}, we have
\be\label{220604-2}
u_0(\cdot,\kappa^2)=\sum_{n\in\Z}c_n\frac{J_n(\kappa|x|)}{J_n(\kappa R)}e^{in\theta},\quad x\in B,
\en
where $R>0$ is the radius of $B$ and $J_n$ denotes the Bessel function of order $n$.
Note that $J_n(\kappa R)\neq0$ and $J_n(kR)\neq0$ for all $n\in\Z$ due to the assumption on $q/\sigma=:k^2$ and $\kappa^2$.
We immediately deduce from (\ref{220604-2}) that
\ben
\sigma\pa_\nu u_0(\cdot,k^2)-\tau\pa_\nu u_0(\cdot,\kappa^2)=\sigma\sum_{n\in\Z}c_n\frac{kJ'_n(kR)}{J_n(kR)}e^{in\theta}-\tau\sum_{n\in\Z}c_n\frac{\kappa J'_n(\kappa R)}{J_n(\kappa R)}e^{in\theta}.
\enn
It follows from the relations in \cite[Section 3.2]{Cakoni14}) and \cite[Section 3.5]{CK19}) that
that
\ben
\frac{J'_n(t)}{J_n(t)}=\frac{J'_{|n|}(t)}{J_{|n|}(t)}=\frac{|n|}t-\frac{J_{|n|+1}(t)}{J_{|n|}(t)}=\frac{|n|}t-\frac{t}{2(|n|+1)}\left\{1+O\left(\frac1{|n|}\right)\right\},\quad |n|\ra\infty.
\enn
Therefore,
\ben
&&[\sigma A_0(q/\sigma)-\tau A_0(\kappa^2)]f=\sum_{n\in\Z}g_ne^{in\theta},\\ 
&&g_n=c_n\left\{(\sigma-\tau)\frac{|n|}{R}+\frac{(\sigma k^2-\tau\kappa^2)R}{2(|n|+1)}+O\left(\frac1{|n|^2}\right)\right\},\quad|n|\ra\infty.
\enn
For $f\in H^s(\pa B)\ba H^{s+\varepsilon}(\pa B)$, it can be easily deduced from (\ref{220604-3}) that
\ben
\begin{cases}
[\sigma A_0(q/\sigma)-\tau A_0(\kappa^2)]f\in H^{s-1}(\pa B)\ba H^{s-1+\varepsilon}(\pa B) & \text{if }\tau\neq\sigma,\\
[\sigma A_0(q/\sigma)-\tau A_0(\kappa^2)]f\in H^{s+1}(\pa B)\ba H^{s+1+\varepsilon}(\pa B) & \text{if }\tau=\sigma\text{ and }\kappa\neq k.
\end{cases}
\enn
This implies $[\sigma A_0(q/\sigma)-\tau A_0(\kappa^2)]f\notin C^\infty(\pa B)$.
However, we conclude from Theorem \ref{t3.3} (a) that ${\rm Ran}\,G(k^2,\Om)\subset C^\infty(\pa B)$.
Therefore, $[\sigma A_0(q/\sigma)-\tau A_0(\kappa^2)]f\notin{\rm Ran}\,G(\kappa^2,\Om)$.
\end{proof}

Our sampling scheme to recover coefficients is based on solving \eqref{5}--\eqref{6} for different values of $k=\kappa$.
This motivates us to remove the assumption on eigenvalues of \eqref{5}--\eqref{6}.

\begin{theorem}\label{thm-2}
Let $\widehat B$ be a domain contained in the interior of $\Om$ and assume that \eqref{1'}--\eqref{3'} with $k=\kappa$ and $\Om=\widehat B$ is uniquely solvable.
Let $u$ solve \eqref{-1}--\eqref{-3} and $v=\tilde u_0(\cdot,\kappa^2)$ solve \eqref{1'}--\eqref{3'} with $k=\kappa$ and $\Om=\widehat B$.
Assume that $\ov{D}\subset\Om\subset\ov{\Om}\subset B$.
For $\tau\in\C$ and $\kappa\in\C$ define $\tilde p(\tau,\kappa^2):=\sigma\pa_\nu u-\tau\pa_\nu\tilde u_0(\cdot,\kappa^2)$ on $\pa B$.
Then the following statements are true:

(i) If $\tau=\sigma$ and $\kappa^2=q/\sigma$, then $\tilde p(\tau,\kappa^2)\in{\rm Ran}\,G(\kappa^2,\Om)$.

(ii) $\tilde p(\tau,\kappa^2)\in{\rm Ran}\,G(\kappa^2,\Om)$ if and only if $[\sigma A(q/\sigma,\widehat B)-\tau A(\kappa^2,\widehat B)]f\in{\rm Ran}\,G(k^2,\Om)$.
\end{theorem}
%

The proof of Theorem \ref{thm-2} is similar to that of Theorem \ref{thm-1}.
Similarly, if $[\sigma A(q/\sigma,\widehat B)-\tau A(\kappa^2,\widehat B)]f\notin{\rm Ran}\,G(\kappa^2,\Om)$ whenever $(\tau,\kappa^2)\neq(\sigma,q/\sigma)$, then $(\sigma,q)$ can be recovered by checking whether $\tilde p(\tau,\kappa^2)$ belongs to ${\rm Ran}\,G(\kappa^2,\Om)$.
Moreover, we have the following theorem.

\begin{theorem}\label{thm250518-1}
Let $\widehat B$ be a domain contained in $\Om$ and $f\in H^s(\pa B)\ba H^{s+\varepsilon}(\pa B)$ for any $s\geq-\frac12$ and $\varepsilon>0$.
Let $u$ solve \eqref{-1}--\eqref{-3} and $v=\tilde u_0(\cdot,\kappa^2)$ solve \eqref{1'}--\eqref{3'} with $(k,\Om)=(\kappa,\widehat B)$.
Assume $\ov{D}\subset\Omega\subset\ov{\Om}\subset B$.
If $(\tau,\kappa^2)\neq(\sigma,q/\sigma)$, then $[\sigma A(q/\sigma,\widehat B)-\tau A(\kappa^2,\widehat B)]f\notin{\rm Ran}\,G(\kappa^2,\Om)$.
\end{theorem}
\begin{proof}

If neither $q/\sigma$ nor $\kappa^2$ is an eigenvalue of \eqref{5}--\eqref{6}, we can deduce from Theorems \ref{t3.3} and \ref{t3.5} that $A(q/\sigma,\widehat B)-A_0(q/\sigma):H^s(\pa B)\to H^{m-3/2}(\pa B)$ and $A(\kappa^2,\widehat B)-A_0(\kappa^2):H^s(\pa B)\to H^{m-3/2}(\pa B)$ are bounded for all $m\geq1$.
From the proof of Theorem \ref{53-thm1}, we conclude that
\be\label{55-0}
\begin{cases}
[\sigma A(q/\sigma,\widehat B)-\tau A(\kappa^2,\widehat B)]f\in H^{s-1}(\pa B)\ba H^{s-1+\varepsilon}(\pa B) & \text{if }\tau\neq\sigma,\\
[\sigma A(q/\sigma,\widehat B)-\tau A(\kappa^2,\widehat B)]f\in H^{s+1}(\pa B)\ba H^{s+1+\varepsilon}(\pa B) & \text{if }\tau=\sigma\text{ and }\kappa^2\neq q/\sigma.
\end{cases}
\en

If $\kappa^2$ is an eigenvalue of \eqref{5}--\eqref{6}, proceeding as in the proof of \cite[Theorem 2.33]{KH} the corresponding eigenfunction of \eqref{5}--\eqref{6} can be represented by
\ben
\sum_{n\in\Z}\alpha_nJ_n(\kappa|x|)e^{in\theta},
\enn
where $\theta\in[0,2\pi)$ is given by $x/|x|=(\cos\theta,\sin\theta)$.
Since there are only finitely many linearly independent eigenfunctions possible (see \cite[Theorem 5.1]{CK19}), we conclude that there exists an integer $N=N(\kappa^2)$ such that $\alpha_n=0$ and $J_n(\kappa R)\neq0$ for $|n|\geq N$.
In view of \eqref{53-1}, we define
\ben
f^{(N)}(x)=\sum_{|n|\geq N}c_ne^{in\theta},\quad x\in\pa B,
\enn
then $f-f^{(N)}\in C^\infty(\pa B)$.
It follows from Remark \ref{rem-} (iii) that
\be\label{55-}
[\sigma A(q/\sigma,\widehat B)-\tau A(\kappa^2,\widehat B)](f-f^{(N)})\in C^\infty(\pa B).
\en
Moreover,
\ben
u_0^{(N)}(\cdot,\kappa^2)=\sum_{|n|\geq N}c_n\frac{J_n(\kappa|x|)}{J_n(\kappa R)}e^{in\theta},\quad x\in B,
\enn
solves the Helmholtz equation with wavenumber $\kappa^2$ in $B$ and $u_0^{(N)}(\cdot,\kappa^2)=f^{(N)}$ on $\pa B$.
Therefore, for $N>\max\{N(q/\sigma),N(\kappa^2)\}$ we have
\ben
\sigma\pa_\nu u_0^{(N)}(\cdot,k^2)-\tau\pa_\nu u_0^{(N)}(\cdot,\kappa^2)=\sigma\sum_{|n|\geq N}c_n\frac{kJ'_n(kR)}{J_n(kR)}e^{in\theta}-\tau\sum_{|n|\geq N}c_n\frac{\kappa J'_n(\kappa R)}{J_n(\kappa R)}e^{in\theta},
\enn
where $k^2=q/\sigma$.
Note that $A(\kappa^2,\widehat B)f^{(N)}-\pa_\nu u_0^{(N)}(\cdot,\kappa^2)|_{\pa B}=-G(\kappa^2,\widehat B)[(\pa_\nu+\eta)u_0^{(N)}(\cdot,\kappa^2)|_{\pa\widehat B}]\in C^\infty(\pa B)$ and $A(k^2,\widehat B)f^{(N)}-\pa_\nu u_0^{(N)}(\cdot,k^2)|_{\pa B}=-G(k^2,\widehat B)[(\pa_\nu+\eta)u_0^{(N)}(\cdot,k^2)|_{\pa\widehat B}]\in C^\infty(\pa B)$.
Analogously to the proof of Theorem \ref{53-thm1}, we can deduce from \eqref{55-} that \eqref{55-0} remains valid.

It follows from \eqref{55-0} that $[\sigma A(q/\sigma,\widehat B)-\tau A(\kappa^2,\widehat B)]f\notin C^\infty(\pa B)$.
However, Theorem \ref{t3.3} (a) implies ${\rm Ran}\,G(k^2,\Om)\subset C^\infty(\pa B)$.
Therefore, $[\sigma A(q/\sigma,\widehat B)-\tau A(\kappa^2,\widehat B)]f\notin{\rm Ran}\,G(\kappa^2,\Om)$.
\end{proof}

\begin{remark}

(a) From the above proof, we see that a functions from ${\rm Ran}\,G(\kappa^2,\Om)$ is smooth and $[\sigma A(q/\sigma,\widehat B)-\tau A(\kappa^2,\widehat B)]f$ is not smooth provided $f$ is nonsmooth and $(\tau,\kappa^2)\neq(\sigma,q/\sigma)$.
{\bl Since a smooth transform does not change regularities, the proposed numerical method for recovering constant parameters is also applicable if the disk $B$ is replaced by any bounded domain with smooth boundary.}

{\x (b) In view of \cite[Theorem 2.2]{CK19}, the solution $v$ to \eqref{7}--\eqref{9} is analytic in $B\ba\ov{\Om}$.
Let $\widehat\Om$ be a domain with analytic boundary $\pa\widehat\Om$ such that $\Om\subset\ov{\Om}\subset\widehat\Om\subset\ov{\widehat\Om}\subset B$, for example $\widehat\Om$ is a disk centered at the center of $B$ with radius slightly less than that of $B$.
Then $v$ solves the Helmholtz equation in the analytic domain $B\ba\ov{\widehat\Om}$ with analytic Dirichlet boundary value on $\pa\widehat\Om$ and homogeneous Dirichlet boundary value on $\pa B$.
Therefore, $v$ is analytic in $\ov{B}\ba\widehat\Om$ and $\pa_\nu v|_{\pa B}\in C^\om(\pa B)$, where $C^\om(\pa B)$ denotes the set of all analytic functions defined on $\pa B$.
This shows ${\rm Ran}\,G(\kappa^2,\Om)\in C^\om(\pa B)$.}
If we can find a Dirichlet boundary value $f\in H^{-1/2}(\pa B)$ such that $[\sigma A(q/\sigma,\widehat B)-\tau A(\kappa^2,\widehat B)]f\notin C^\om(\pa B)$ whenever $(\tau,\kappa^2)\neq(\sigma,q/\sigma)$, then $\tilde p(\tau,\kappa^2)\in{\rm Ran}\,G(\kappa^2,\Om)$ if and only if $(\tau,\kappa^2)=(\sigma,q/\sigma)$.
Actually, if $f\notin C^\om(\pa B)$, then $[\sigma A_0(q/\sigma)-\tau A_0(\kappa^2)]f\notin C^\om(\pa B)$ provided $(\tau,\kappa^2)\neq(\sigma,q/\sigma)$.
Assume to the contrary that $[\sigma A_0(q/\sigma)-\tau A_0(\kappa^2)]f\in C^\om(\pa B)$, then the solution $v(\cdot,q/\sigma)$ to \eqref{1'}--\eqref{3'} with $(k^2,\Om)=(q/\sigma,\widehat B)$ and the solution $v(\cdot,\kappa^2)$ to \eqref{1'}--\eqref{3'} with $(k^2,\Om)=(\kappa^2,\widehat B)$ satisfies the following interior transmission problem
\ben
\Delta v(\cdot,q/\sigma)+q/\sigma v(\cdot,q/\sigma)=0 && \text{in }B\ba\ov{\widehat\Om},\\
\Delta v(\cdot,\kappa^2)+\kappa^2v(\cdot,\kappa^2)=0 && \text{in }B\ba\ov{\widehat\Om},\\
v(\cdot,q/\sigma)-v(\cdot,\kappa^2)=0,\;\sigma\pa_\nu v(\cdot,q/\sigma)-\tau\pa_\nu v(\cdot,\kappa^2)=g_{\pa B} && \text{on }\pa B,\\
v(\cdot,q/\sigma)-v(\cdot,\kappa^2)=f_{\pa\widehat\Om},\;\sigma\pa_\nu v(\cdot,q/\sigma)-\tau\pa_\nu v(\cdot,\kappa^2)=g_{\pa\widehat\Om} && \text{on }\pa\widehat\Om.
\enn
Note that $g_{\pa B}=[\sigma A_0(q/\sigma)-\tau A_0(\kappa^2)]f\in C^\om(\pa B)$, $f_{\pa\widehat\Om},g_{\pa\widehat\Om}\in C^\om(\pa\widehat\Om)$ due to Remark \ref{rem-} (v).
By carefully choosing the radius of the disk $\widehat\Om$, the uniqueness of above interior transmission problem can be guaranteed (see \cite[Section 6.3]{Cakoni14} and \cite{CGH2010} for the discreteness of transmission eigenvalues), and the well-posedness of above interior transmission problem can be established in a similar manner as in \cite[Section 4]{wcy} no matter $\sigma$ equals $\tau$ or not.
We conclude from the analyticity of the boundary and boundary data that $v(\cdot,q/\sigma),v(\cdot,\kappa^2)\in C^\om(B\ba\ov{\widehat\Om})$ and $v(\cdot,q/\sigma)|_{\pa B},v(\cdot,\kappa^2)|_{\pa B}\in C^\om(\pa B)$.
This contradicts to $v(\cdot,q/\sigma)|_{\pa B}=v(\cdot,\kappa^2)|_{\pa B}=f\notin C^\om(\pa B)$.
\end{remark}

\begin{remark}\label{rem-250518-1}
(a) Instead of by calculating $G(\kappa^2,\Om)$ directly, we obtain ${\rm Ran}\,G(\kappa^2,\Om)$ indirectly by calculating $A(\kappa^2,\Om)$ and $A(\kappa^2,\widetilde\Om)$ with $\widetilde\Om\subset\Om$ as mentioned in the previous section.
More precisely, we have ${\rm Ran}\,G(\kappa^2,\Om)={\rm Ran}[A(\kappa^2,\Om)-A(\kappa^2,\widetilde\Om)]_\#^{1/2}$ (see Theorems \ref{thm3.8_ID} and \ref{thm3.8m_ID}).
Noting that $A(\kappa^2,\Om)$, $A(\kappa^2,\widetilde\Om)$, and $p(\tau,\kappa^2)$ are equivalent to corresponding Cauchy data, we are able to recover the coefficients $\sigma$ and $q$ in a data-to-data manner.

(b) Theorems \ref{thm-2} and \ref{thm250518-1} can be viewed as a uniqueness result for the inverse problem to determine $\sigma$ and $q$ from a single pair of Cauchy data to \eqref{-1}--\eqref{-3}.

(c) Proceeding as above we conclude that $G(k^2,\Om)$ can be replaced by $G(k^2,\mho)$.
For numerical implementation, $G(k^2,\mho)$ can be obtained from \eqref{220623-1}.
Moreover, $\tilde p(\tau,\kappa^2)$ in Theorem \ref{thm-2} can be replaced by $\sigma\pa_\nu u-\tau\pa_\nu\tilde u_0(\cdot,\kappa^2)$ where $w=\tilde u_0(\cdot,\kappa^2)$ uniquely solves \eqref{m1}--\eqref{m2} with $k=\kappa$ and ${\rm supp}(1-n)=:\ov{\mho}=\widehat B$ satisfying $\ov{\widehat B}\subset\Om$.

(d) For an explicit example for $f\in H^s(\pa B)\ba H^{s+\varepsilon}(\pa B)$, we refer to \cite[Remark 3.4]{XH24}.
\end{remark}


\section{Reconstruction of the polygon}\label{sec4}

Since $\sigma$ and $q$ can be theoretically determined and numerically recovered by a single pair of Cauchy data as shown in the previous section, we will continue with the assumption that the constants $\sigma$ and $q$ are known.
This section will focus on uniqueness results and numerical methods for recovering the location and shape of $D$.
In view of \cite[Theorem 5.2 and Corollary 5.3]{CK19}, \cite[Theorem 5.5]{CK19} and \cite[Theorem 3.7 (iii)]{XH24}, we immediately obtain the following theorems.

\begin{theorem}
Let $c$ be the first positive zero of Bessel function $J_0$.
Suppose that $f\not\equiv0$ and $k\cdot{\rm diam}(D)<c$, where ${\rm diam}(D)$ denotes the diameter of $D$, then $D$ can be uniquely determined by a single pair of Cauchy data $(f,\pa_\nu u|_{\pa B})$ where $u$ solves the boundary value problem (\ref{1})--(\ref{3}).
\end{theorem}

If the solution $u$ to (\ref{1})--(\ref{3}) can be analytically extended to the domain $B$, then the extended function $u$ must be a Dirichlet eigenfunction over $D$.
Therefore, it is also possible to determine $D$ if $f$ is not the  boundary value of some function whose restriction to $D$ is a Dirichlet eigenfunction. Below we show the singularity of $u$ around corners of the polygonal obstacle $D$.

\begin{theorem}\label{thm220605-3}
Suppose that one of the following conditions {\bl holds}:
\begin{itemize}
\item[(i)] {\x $f$ is not identically zero on $\pa B$ and $f\notin\{v|_{\partial B}:\Delta v+k^2v=0\text{ in }B\text{ and }v|_{\pa D}=0\}$; $D$ is a convex polygon such that ${\rm diam}(D)<{\rm dist}(D,\pa B)$.}

\item[(ii)] $f(x)\neq 0$ for all $x\in \partial B$; $D$ is a convex polygon.

\item[(iii)] {\x $f\in H^s(\pa B)\ba H^{s+\varepsilon}(\pa B)$ for any $s\geq-\frac12$ and $\varepsilon>0$; $D$ is a convex polygon such that ${\rm diam}(D)<{\rm dist}(D,\pa B)$.}
\end{itemize}
Then the solution $u$ to (\ref{1})--(\ref{3}) cannot be analytically extended across any corner of $D$.
\end{theorem}

\begin{proof} (i) Since $f$ is not identically zero on $\pa B$, $u$ is not identically zero in $B\ba\ov{D}$.
Assume to the contrary that $u$ can be analytically across a corner $z$ of $D$.
By the Schwartz reflection principle of Helmholtz equation (see e.g., \cite{MaHu}), we deduce from ${\rm diam}(D)<{\rm dist}(D,\pa B)$ that $u$ can be analytically extended as a Dirichlet eigenfunction in $D$.
This contradicts the assumption on $f$.

(ii) {\x The proof in the second case is similar to \cite[Theorem 5.5]{CK19}.}

(iii) {\x Assume to the contrary that $u$ can be analytically across a corner $z$ of $D$.
By the Schwartz reflection principle of Helmholtz equation (see e.g., \cite{MaHu}), we can show that $u$ can be analytically extended across $\pa B$ and $f$ is analytic on $\pa B$ which contradicts $f\notin H^{s+\varepsilon}(\pa B)$.}
\end{proof}

{\x Theorem \ref{thm220605-3} can be viewed as uniqueness results that determining unknown Dirichlet polygon $D$ from a single pair of Cauchy data under some a priori assumptions.
Using the reflection principle for the Helmholtz equation, it is also possible to prove the results of Theorem \ref{thm220605-3} (ii) and (iii) for non-convex polygons.}
The above proofs of uniqueness results cannot be directly applied in numerical implementation.
The one-wave factorization method explored within this paper is based on the following theorem {\x similar to \cite[Theorem 3.8]{XH24}.}

\begin{theorem}\label{thm211027}
Let the assumptions of Theorem \ref{thm220605-3} be fulfilled.
Assume $\Omega$ is a convex domain such that $\widehat B\subset\Om\subset B$.
It holds that

(i) Assume further that $k^2$ is not an eigenvalue of \eqref{5}--\eqref{6}.
Then $D\subset\Omega$ if and only if $p(\sigma,q/\sigma)\in{\rm Ran}\,G(q/\sigma,\Omega)$, where $p(\cdot,\cdot)$ is defined in Theorem \ref{thm-1};

(ii) $D\subset\Omega$ if and only if $\tilde p(\sigma,q/\sigma)\in{\rm Ran}\,G(q/\sigma,\Omega)$, where $\tilde p(\cdot,\cdot)$ is defined in Theorem \ref{thm-2}.
\end{theorem}

\begin{remark}\label{rem1029-2}
(i) Theorem \ref{thm211027} remains valid with ${\rm Ran}\,G(q/\sigma,\Om)$ replaced by ${\rm Ran}\,G(q/\sigma,\mho)$.

(ii) A characterization of a convex polygon $D\subset B$ in terms of a single pair of Cauchy data on $\partial B$ is given as follows.
Denote by $\mathcal O$ the set of all convex Lipschitz domains that contained in the interior of $B$.
We conclude from Theorem \ref{thm211027} that
\ben
D=\bigcap\left\{\Omega\in\mathcal O:p(\sigma,q/\sigma)\in{\rm Ran}\,G(q/\sigma,\Omega)\right\}.
\enn
By Theorem \ref{thm3.8} we have
\ben
D=\bigcap\left\{\Omega\in\mathcal O:p(\sigma,q/\sigma)\in{\rm Ran}\,[A(q/\sigma,\Omega)-A_0(q/\sigma)]_\#^{1/2}\right\}
\enn
provided $q/\sigma$ is not an eigenvalue of \eqref{5}--\eqref{6}.
Moreover, by Theorem \ref{thm3.8_ID} we have
\ben
D=\bigcap\left\{\Omega\in\mathcal O:\tilde p(\sigma,q/\sigma)\in{\rm Ran}\,[A(q/\sigma,\Omega)-A(q/\sigma,\widetilde\Om)]_\#^{1/2}\right\}.
\enn
\end{remark}


{\x In our numerical implementations, it is more practical to firstly apply the one-wave factorization method to get an initial guess of the inclusion and then improve the inversion results by an optimization-based iteration approach.
We note that the effectiveness of iteration method depends heavily on a good initial guess. 
For details on iteration method, we refer the reader to \cite{BC2002}.
An analogous numerical implementation for Laplace equation can be found in \cite{XH24}.
For simplicity, we omit the numerical implementation of iteration method.}

\section{Numerical examples}\label{s4}
\setcounter{equation}{0}

{\bl Proceeding as in \cite{XH24}, one can prove the well-posedness of above boundary value problems by the integral equation method and express the Dirichlet-to-Neumann operator using boundary integral operators.}
{\x Similarly, a graded mesh on the boundary of the polygon $\pa D$ is used for high order convergence (see \cite[Section 3.6]{CK19}, \cite{QZZ} and \cite{XH24}).
This section is devoted to numerical implementation of the methods introduced in this paper.}

\subsection{Numerical examples for the factorization method}

In this subsection, we display several numerical examples {\bl to implement the proposed factorization method by using the Dirichlet-to-Neumann operator}.
Combining Remark \ref{0915} (iv), Theorems \ref{thm3.8} and \ref{zB} and Picard's theorem \cite[Theorem 4.8]{CK19}, we immediately obtain the following result.

\begin{theorem}\label{imaging0}
Under the assumptions of Theorems \ref{thm3.8} and \ref{zB}, the following statements hold:

(i) Denote by $(\lambda_n,\varphi_n)$ an eigensystem of $[A(k^2,\Om)-A_0(k^2)]_\#$.
Define
\ben
I_{obstacle}(z):=\left[\sum_{n}\frac{|(\pa_\nu\Psi_k(\cdot,z),\varphi_n)|^2}{|\lambda_n|}\right]^{-1},\quad z\in B,
\enn
where $(\cdot,\cdot)$ denotes the inner product in $L^2(\pa B)$.
Then $z\in\Om$ if and only if $I_{obstacle}(z)>0$.

(ii) Denote by $(\mu_n,\phi_n)$ an eigensystem of $[A(k^2,\mho)-A_0(k^2)]_\#$.
Define
\ben
I_{medium}(z):=\left[\sum_{n}\frac{|(\pa_\nu\Psi_k(\cdot,z),\phi_n)|^2}{|\mu_n|}\right]^{-1},\quad z\in B,
\enn
where $(\cdot,\cdot)$ denotes the inner product in $L^2(\pa B)$.
Then $z\in\mho$ if and only if $I_{medium}(z)>0$.
\end{theorem}

Combining Remark \ref{0915} (iv), Theorems \ref{thm3.8_ID}, \ref{thm3.8m_ID} and \ref{zB_ID} and Picard's theorem \cite[Theorem 4.8]{CK19}, we immediately obtain the following result analogous to Theorem \ref{imaging0}.

\begin{theorem}\label{imaging1}
Under the assumptions of Theorems \ref{thm3.8_ID}, \ref{thm3.8m_ID}, and \ref{zB_ID}, it holds that:

(i) Denote by $(\tilde\lambda_n,\tilde\varphi_n)$ an eigensystem of $[A(k^2,\Om)-A(k^2,\widetilde\Om)]_\#$.
Define
\ben
\widetilde I_{obstacle}(z):=\left[\sum_{n}\frac{|(\pa_\nu\widetilde\Psi_k(\cdot,z),\tilde\varphi_n)|^2}{|\tilde\lambda_n|}\right]^{-1},\quad z\in B,
\enn
where $(\cdot,\cdot)$ denotes the inner product in $L^2(\pa B)$.
Then $z\in\Om$ if and only if $\widetilde I_{obstacle}(z)>0$.

(ii) Denote by $(\tilde\mu_n,\tilde\phi_n)$ an eigensystem of $[A(k^2,\mho)-A(k^2,\widetilde\mho)]_\#$.
Define
\ben
\widetilde I_{medium}(z):=\left[\sum_{n}\frac{|(\pa_\nu\widetilde\Psi'_k(\cdot,z),\tilde\phi_n)|^2}{|\tilde\mu_n|}\right]^{-1},\quad z\in B,
\enn
where $(\cdot,\cdot)$ denotes the inner product in $L^2(\pa B)$.
Then $z\in\mho$ if and only if $\widetilde I_{medium}(z)>0$.
\end{theorem}

\begin{remark}
(i) For the case when impedance obstacle $\Om$ is replaced by a Dirichlet object in $B$, the analogue to Theorems \ref{imaging0} and \ref{imaging1} can be established in a similar manner.

(ii) The factorization method {\bl using the} Dirichlet-to-Neumann operator can be extended to inverse scattering problems {\bl with} near-field data, if the corresponding Dirichlet-to-Neumann operator can be represented in terms of near-field data.

(iii) {\bl Analogously to the limited aperture data case (see \cite[Section 2.3]{Kirsch08}), the factorization method introduced in Theorems \ref{imaging0} and \ref{imaging1} is applicable even if the Cauchy data on an open subset of $\pa B$ is available.}
\end{remark}

\begin{example}[Factorization method based on Dirichlet-to-Neumann operators]\label{Ex0819}
Let $B$ be a disk centered at the origin with radius $5$ and $k=1$.
The Dirichlet-to-Neumann operators $A(k^2,\Om)$, $A_0(k^2)$, $A(k^2,\widetilde\Om)$, $A(k^2,\mho)$, and $A(k^2,\widetilde\mho)$ are approximated by different $128\times128$ matrices, respectively.
The impedance coefficient on $\pa\widetilde\Om$ is $\eta=-ik$ and the refractive index of $\widetilde\mho$ is $(3-4i)^2$.
The numerical results of indicator functions $I_{obstacle}(z)$ and $I_{medium}(z)$ defined in Theorem \ref{imaging0} for different objects are shown in Figure \ref{ex0819} (a)--(e) and the numerical results of indicator functions $\widetilde I_{obstacle}(z)$ and $\widetilde I_{medium}(z)$ defined in Theorem \ref{imaging1} for different objects are shown in Figure \ref{ex0819} (f)--(j).
The object in Figure \ref{ex0819} (a) and (f) is a kite-shaped obstacle with  Dirichlet boundary condition, the object in Figure \ref{ex0819} (b) and (g) is a polygonal obstacle with Dirichlet boundary condition, the object in Figure \ref{ex0819} (c) and (h) is a peanut-shaped obstacle with Neumann boundary condition, the object in Figure \ref{ex0819} (d) and (i) is a medium with refractive index $1/4$, and the object in Figure \ref{ex0819} (e) and (j) is a medium with refractive index $4$.
\end{example}
\setcounter{subfigure}{0}
\begin{figure}[htbp]
  \centering
  \subfigure[]{
  \includegraphics[width=0.18\textwidth]{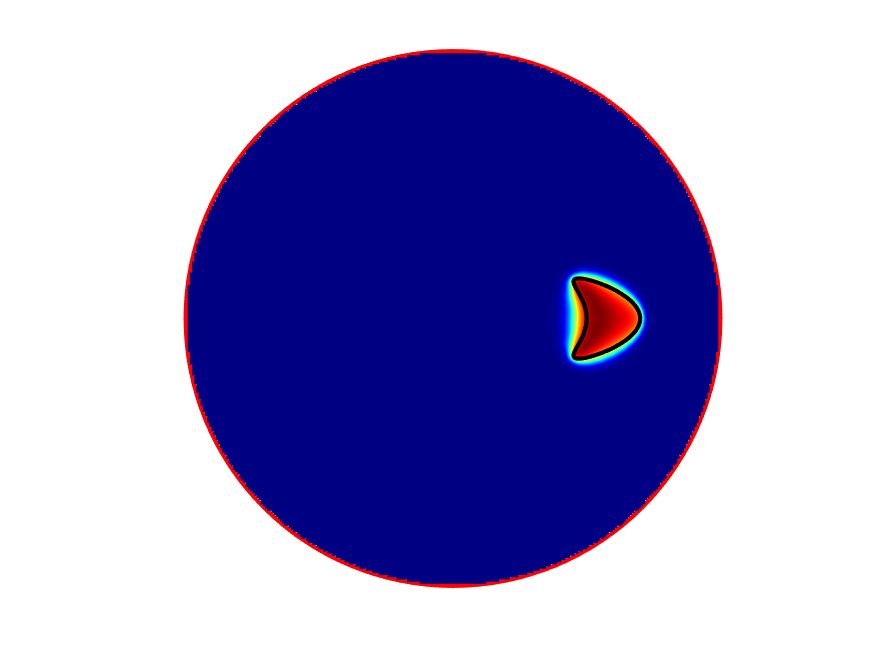}}
  \subfigure[]{
  \includegraphics[width=0.18\textwidth]{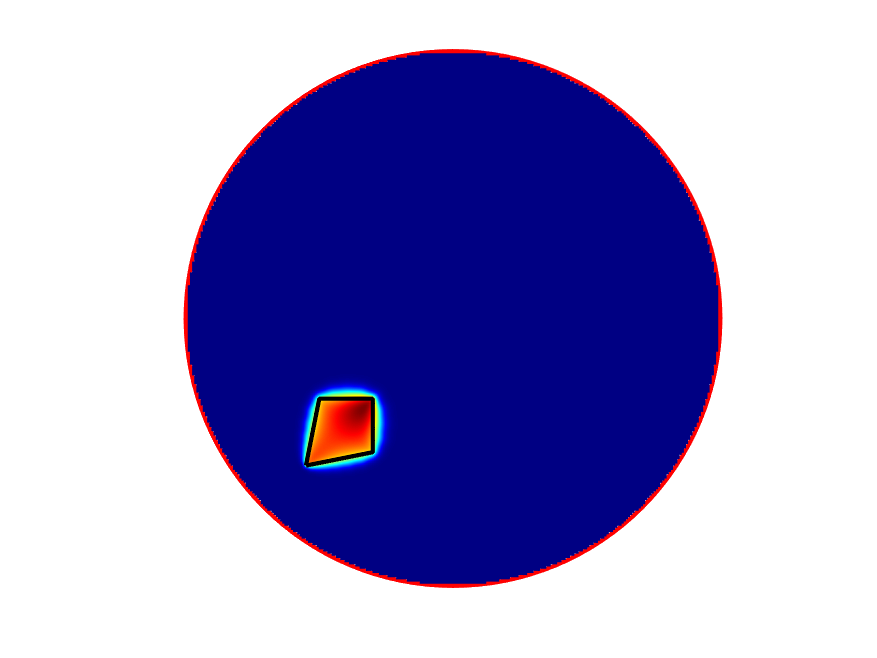}}
  \subfigure[]{
  \includegraphics[width=0.18\textwidth]{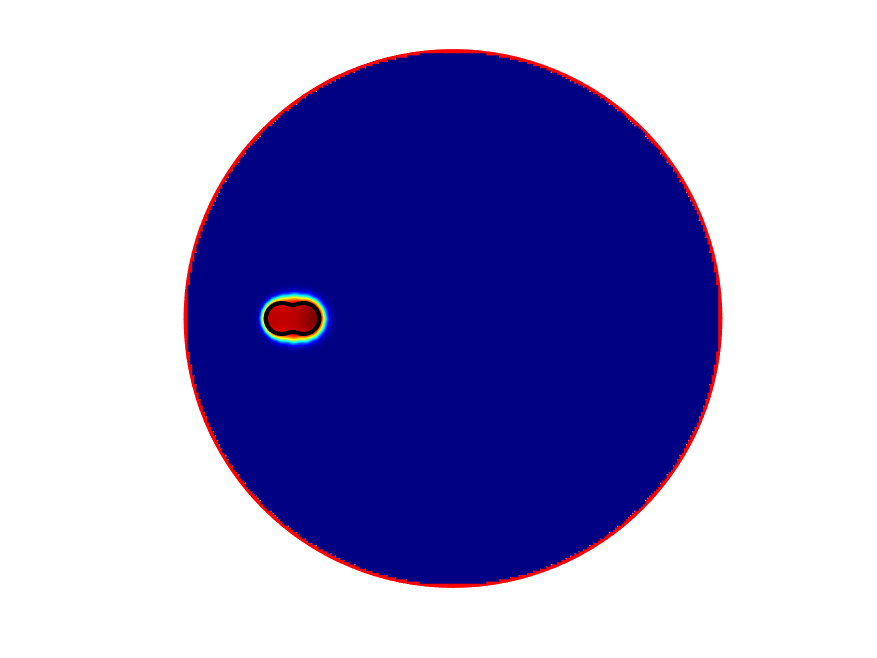}}
  \subfigure[]{
  \includegraphics[width=0.18\textwidth]{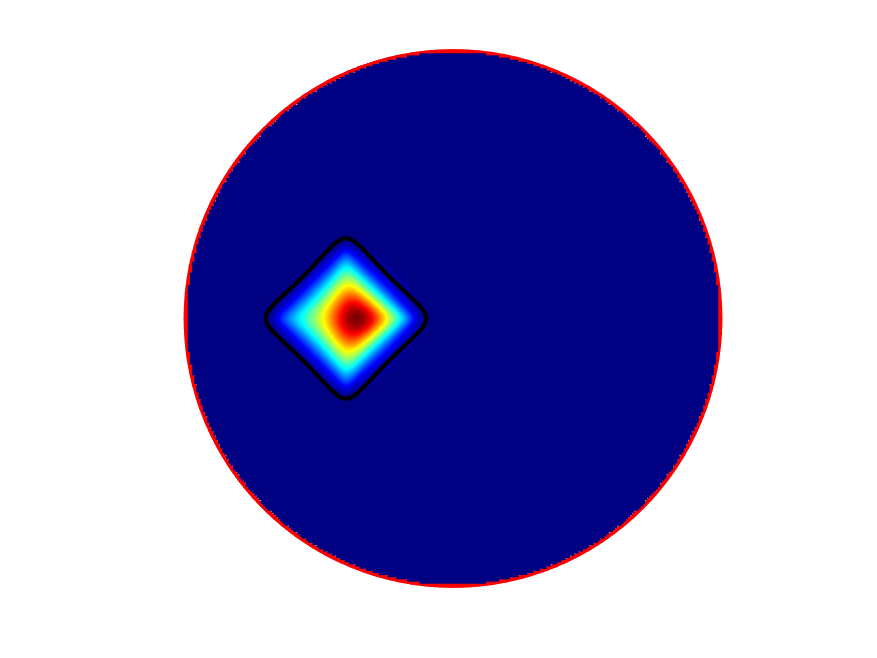}}
  \subfigure[]{
  \includegraphics[width=0.18\textwidth]{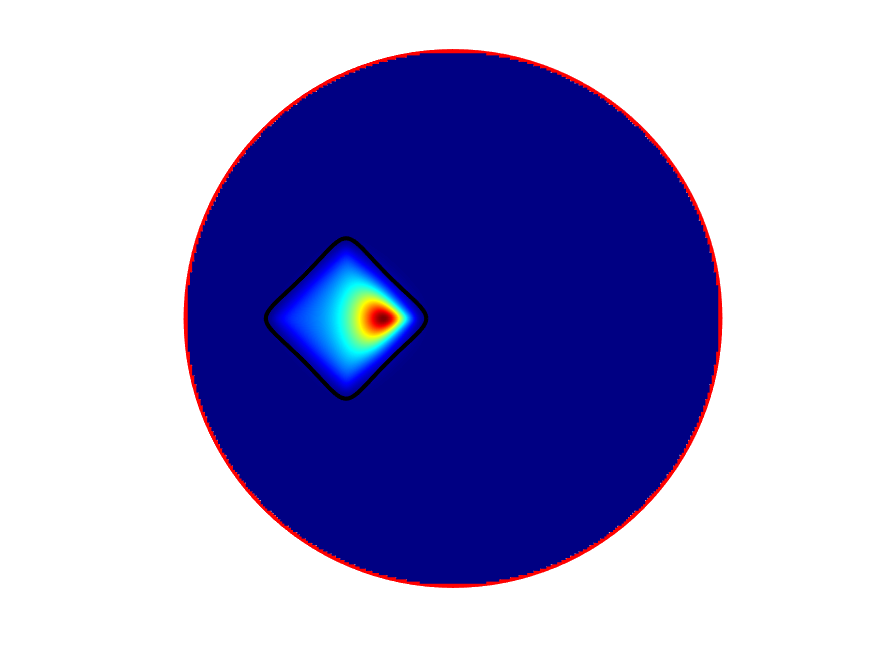}}\\
  \subfigure[]{
  \includegraphics[width=0.18\textwidth]{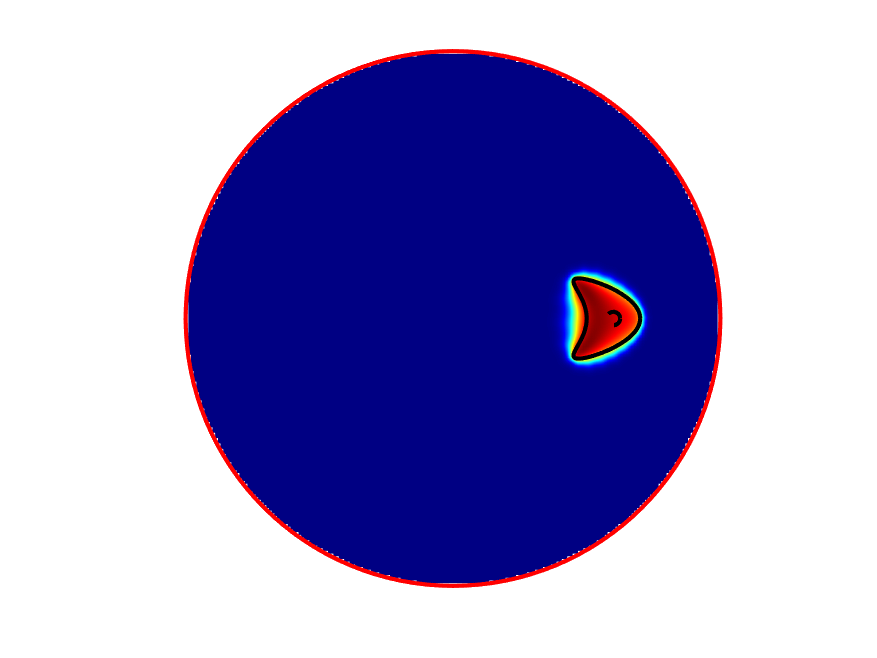}}
  \subfigure[]{
  \includegraphics[width=0.18\textwidth]{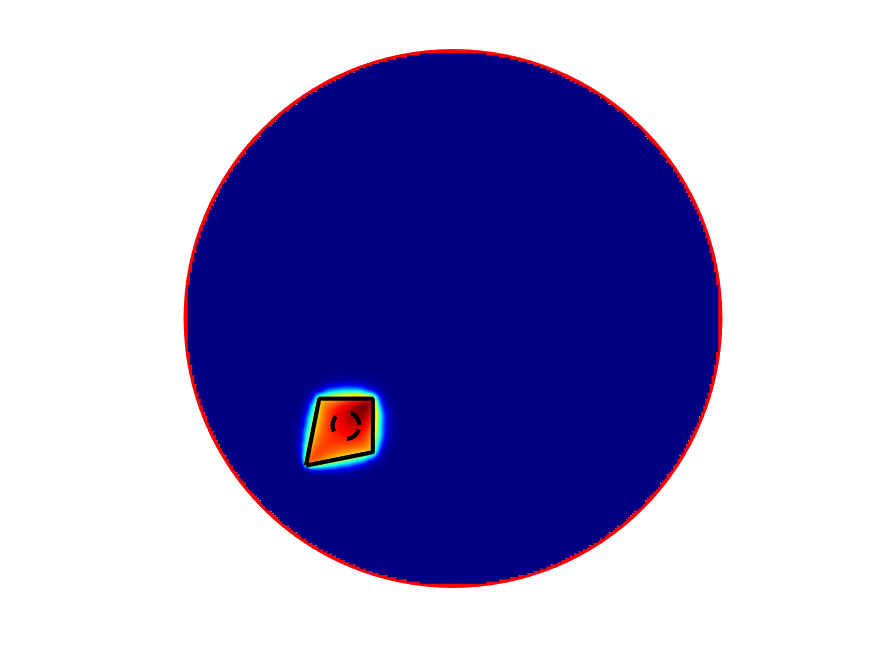}}
  \subfigure[]{
  \includegraphics[width=0.18\textwidth]{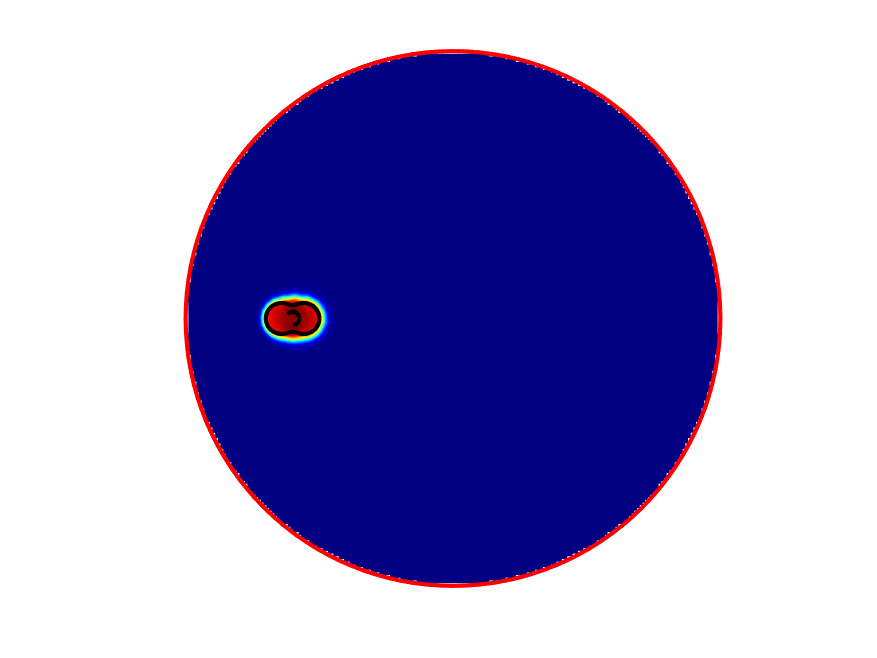}}
  \subfigure[]{
  \includegraphics[width=0.18\textwidth]{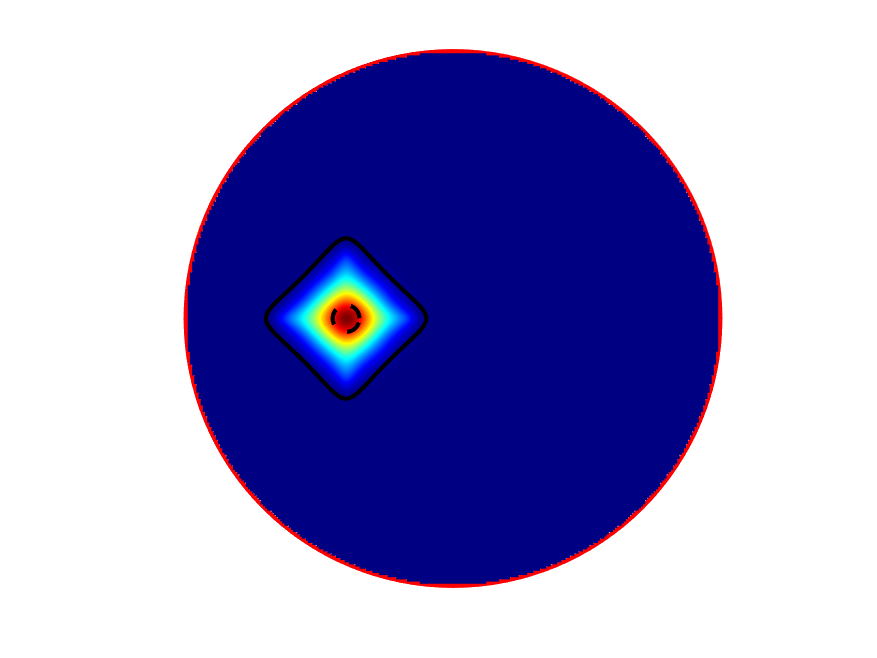}}
  \subfigure[]{
  \includegraphics[width=0.18\textwidth]{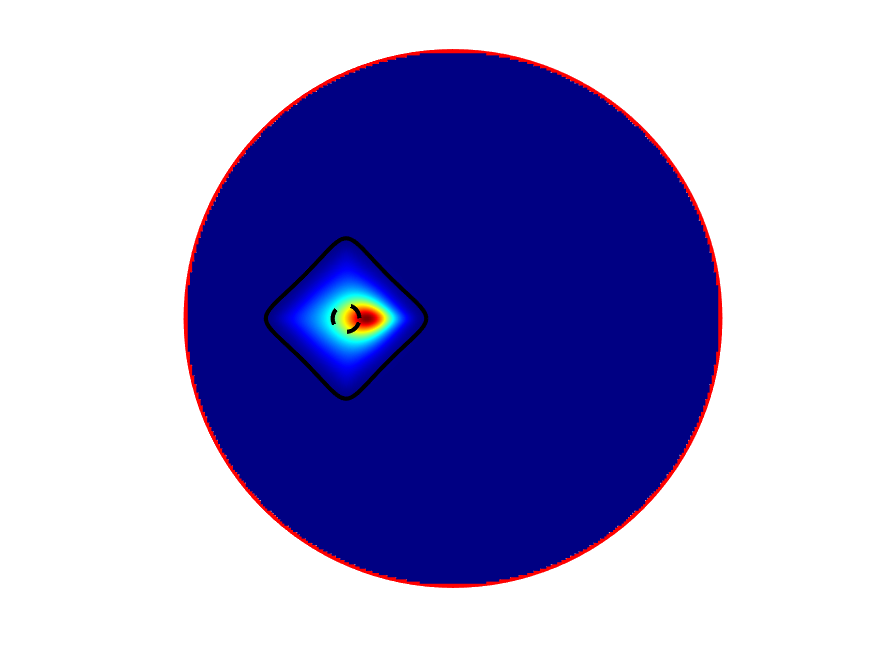}}
  \caption{Numerical examples for Example \ref{Ex0819}, where the red solid line, black solid line and black dashed line representing the disk $B$, the true shape of the unknown object $\Om$ (or $\mho$) and $\widetilde\Om$ inside obstacle $\Om$ (or $\widetilde\mho$ inside the medium with support $\ov{\mho}$), respectively.}\label{ex0819}
\end{figure}

\subsection{Numerical examples for recovering the coefficients}

In this subsection, we display several numerical examples for recovering the coefficients $\sigma$ and $q$ from a single pair of Cauchy data $(f,\tau\pa_\nu u)$ on $\pa B$.
In view of Theorems \ref{thm-1}, \ref{53-thm1}, \ref{thm-2}, \ref{thm250518-1}, Remark \ref{rem-250518-1} and the Picard's theorem (see \cite[Theorem 4.8]{CK19}), we define the following indicator functions to reconstruct the values of $\sigma$ and $q$:
\be\label{ind+}
I_1(\tau,\kappa^2):=\left[\sum_{n=1}^\infty\frac{|(p(\tau,\kappa^2),f_n)|^2}{|\lambda_n|}\right]^{-1},\\ \label{ind+tilde}
\widetilde I_1(\tau,\kappa^2):=\left[\sum_{n=1}^\infty\frac{|(\tilde p(\tau,\kappa^2),\tilde f_n)|^2}{|\tilde\lambda_n|}\right]^{-1},
\en
where $(\cdot,\cdot)$ denotes the inner product in $L^2(\pa B)$, $p(\tau,\kappa^2)$ is defined as in Theorem \ref{thm-1} with $f$ satisfying the condition in Theorem \ref{53-thm1}, $\tilde p(\tau,\kappa^2)$ is defined as in Theorem \ref{thm-2} with $f$ satisfying the condition in Theorem \ref{thm250518-1}, Lipschitz domain $\Om$ satisfies $\ov{D\cup\widetilde\Om}\subset\Om\subset\ov{\Om}\subset B$, and $(\lambda_n,f_n)$ and $(\tilde\lambda_n,\tilde f_n)$ are the eigensystems of $[A(\kappa^2,\Om)-A_0(\kappa^2)]_\#$ and $[A(\kappa^2,\Om)-A(\kappa^2,\widetilde\Om)]_\#$, respectively.
\begin{remark}
In view of Theorems \ref{thm3.8}, \ref{thm3.8_ID}, \ref{thm3.8m_ID} and Remark \ref{rem-250518-1} (c), we can define indicator functions similar to \eqref{ind+} and \eqref{ind+tilde} with obstacle $\Om$ replaced by the support $\ov{\mho}$ of a medium.
\end{remark}

\begin{figure}[htbp]
  \centering
  \subfigure[The geometry of \eqref{ind+} in Example \ref{ex1complex}.]{
  \includegraphics[width=0.4\textwidth]{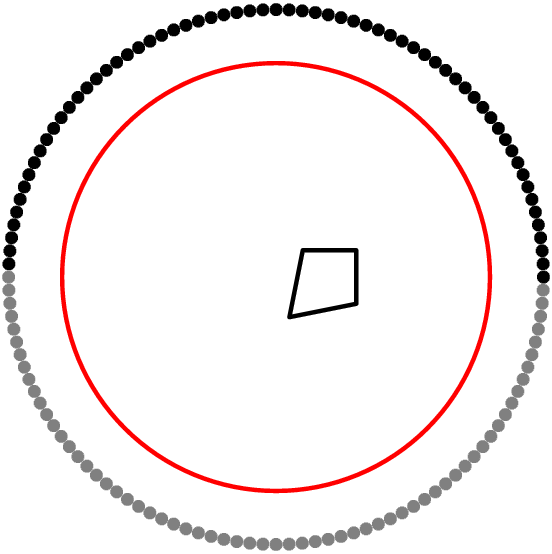}}
  \subfigure[The geometry of \eqref{ind+tilde} in Example \ref{ex1complex}.]{
  \includegraphics[width=0.4\textwidth]{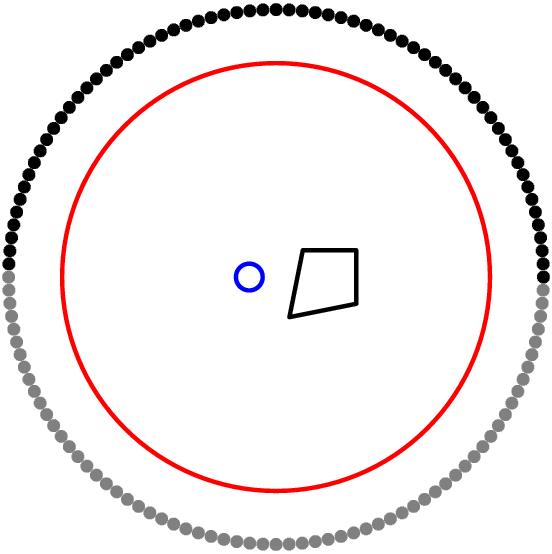}}
  \caption{The geometry of Example \ref{ex1complex}. The black line represents $\pa D$, the red line represents $\pa\Om$ (or $\pa\mho$), the circle with knots represents $\pa B$, and the blue line represents $\pa\widetilde\Om$ (or $\pa\widetilde\mho$). The boundary value $f$ takes value $1$ at black knots and takes value $0$ at gray knots on $\pa B$.}\label{ex_1_geo}
\end{figure}

\begin{example}[Reconstruction of coefficients]\label{ex1complex}
Let $B$ be a disk centered at the origin with radius $5$.
Let $D$ be a polygon with corners given by $(0.25,-0.75)$, $(1.5,-0.5)$, $(1.5,0.5)$, and $(0.5,0.5)$.
Let $\Om$ be an impedance obstacle and $\ov{\mho}$ be the support of a homogeneous medium. In this example, $\Om$ (or $\mho$) is assumed to be a disk centered at the origin with radius $4$ with impedance boundary condition $\pa_\nu u+i|\kappa|u=0$ on $\pa\Om$ (or wave number $|\kappa|(2+i)$ inside $\mho$) for each sampling value of $\kappa$, $\widetilde\Om$ (or $\widetilde\mho$) is set to be a disk centered at $(-0.5,0)$ with radius $0.25$ with impedance boundary condition $\pa_\nu u-i|\kappa|u=0$ on $\pa\widetilde\Om$ (or wave number $|\kappa|(2-i)$ inside $\widetilde\mho$) for each sampling value of $\kappa$
, and $\widehat B$ is set to be the same as $\widetilde\Om$ (or $\widetilde\mho$).
The geometry is shown in Figure \ref{ex_1_geo}.
The numerical results for \eqref{ind+} and \eqref{ind+tilde}, to simultaneously reconstruct two coefficients, are shown in Figure \ref{ex_1simultaneous}.
\end{example}

\setcounter{subfigure}{0}
\begin{figure}[htbp]
  \centering
  \subfigure[$I_1(\tau_j,\kappa_\ell)$ with $\Om$]{
  \includegraphics[width=0.23\textwidth]{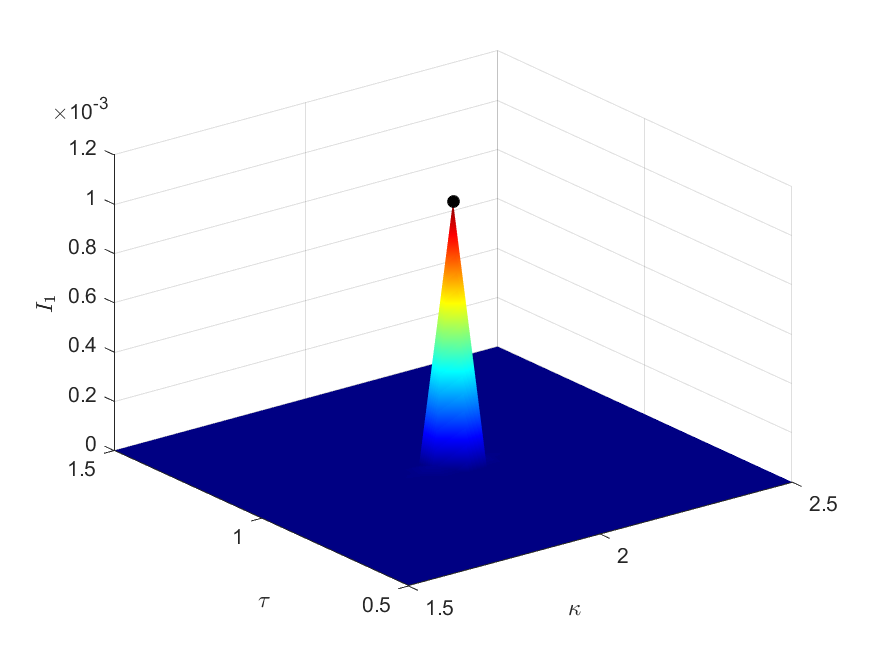}}
  \subfigure[$\ln(I_1(\tau_j,\kappa_\ell))$ with $\Om$]{
  \includegraphics[width=0.23\textwidth]{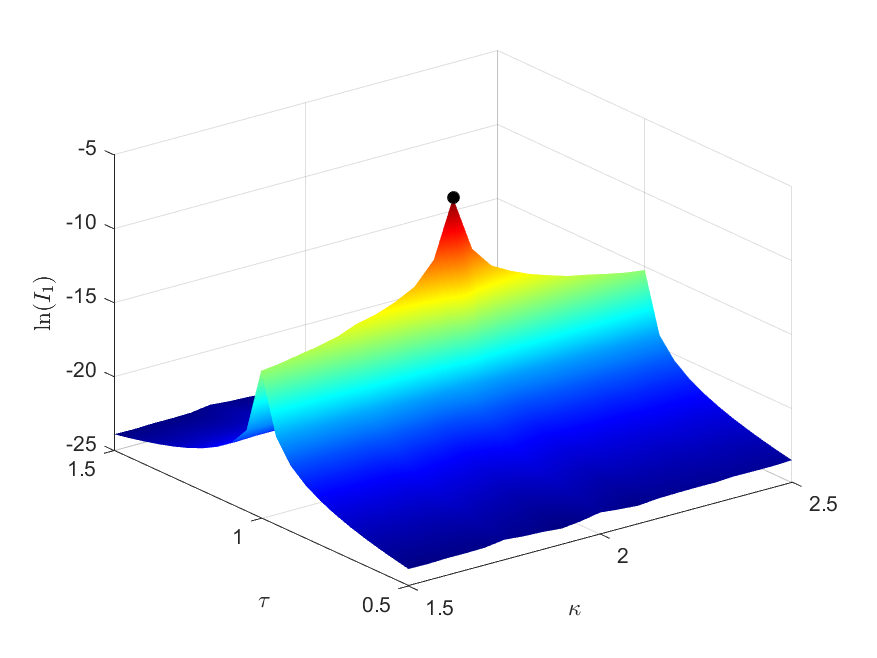}}
  \subfigure[$I_1(\tau_j,\kappa_\ell)$ with $\mho$]{
  \includegraphics[width=0.23\textwidth]{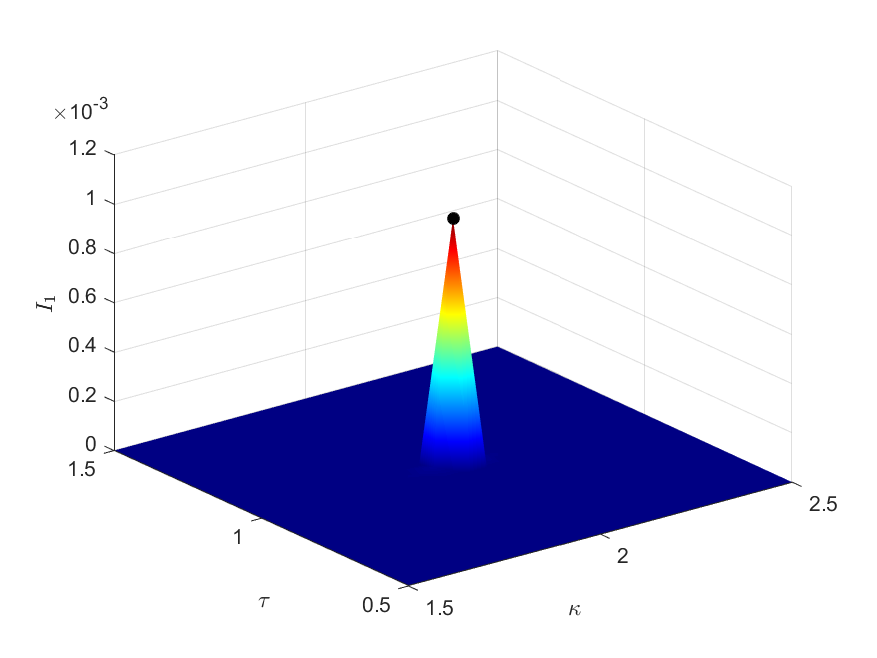}}
  \subfigure[$\ln(I_1(\tau_j,\kappa_\ell))$ with $\mho$]{
  \includegraphics[width=0.23\textwidth]{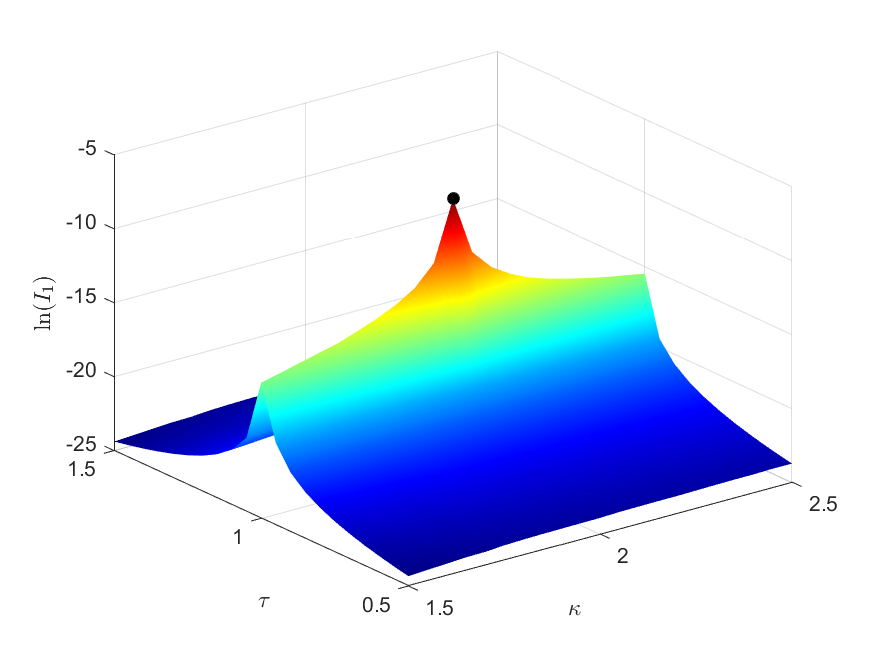}}
  \subfigure[$\widetilde I_1(\tau_j,\kappa_\ell)$ with $\Om$]{
  \includegraphics[width=0.23\textwidth]{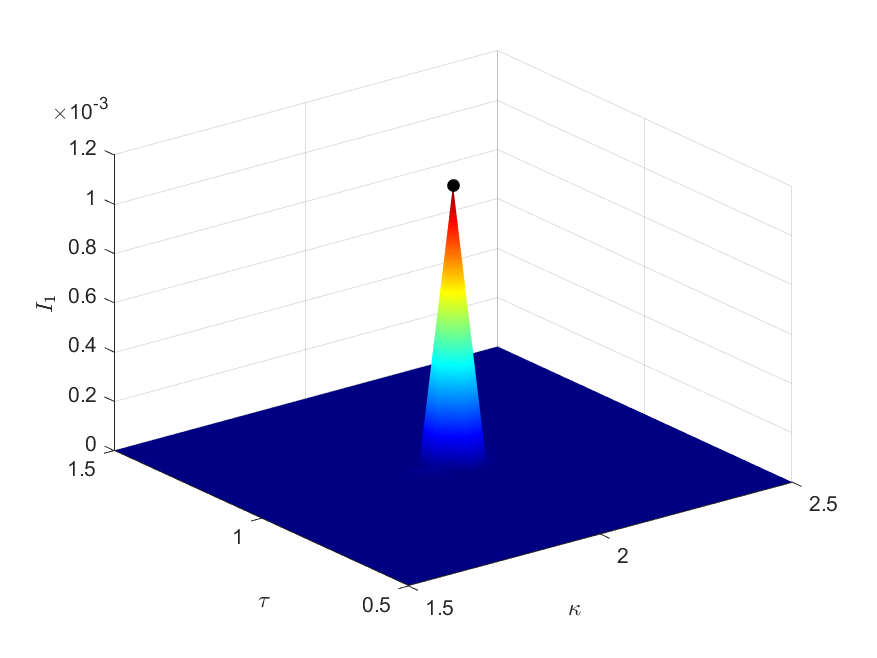}}
  \subfigure[$\ln(\widetilde I_1(\tau_j,\kappa_\ell))$ with $\Om$]{
  \includegraphics[width=0.23\textwidth]{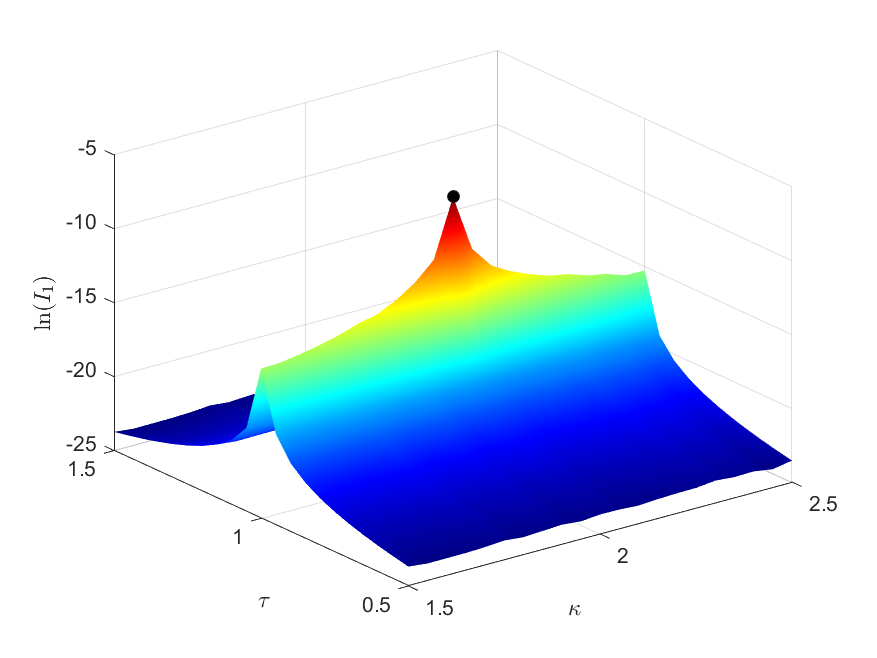}}
  \subfigure[$\widetilde I_1(\tau_j,\kappa_\ell)$ with $\mho$]{
  \includegraphics[width=0.23\textwidth]{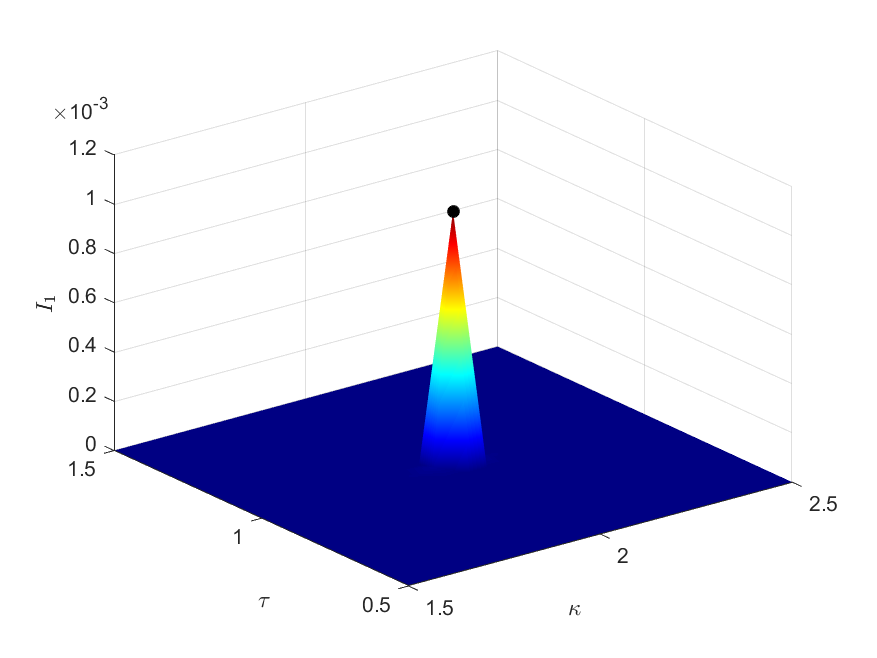}}
  \subfigure[$\ln(\widetilde I_1(\tau_j,\kappa_\ell))$ with $\mho$]{
  \includegraphics[width=0.23\textwidth]{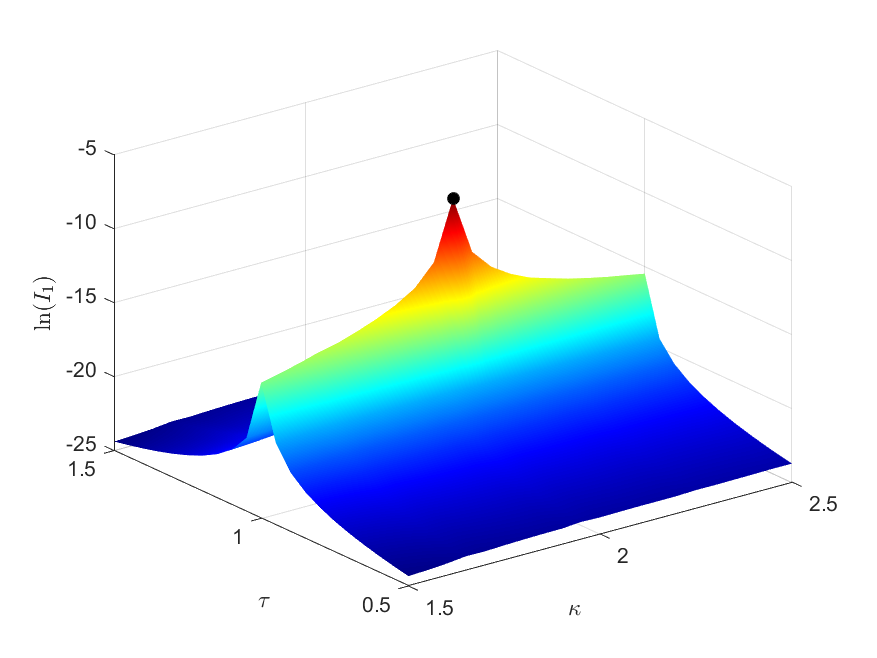}}
  \caption{Numerical results for Example \ref{ex1complex} to simultaneously reconstruct two coefficients $\sigma=1$ and $k=2$ ($q=4$).
  $(\lambda_n,f_n)$ in \eqref{ind+} is the eigensystem of $[A(\kappa^2,\Om)-A_0(\kappa^2)]_\#$ in (a), (b) and of $[A(\kappa^2,\mho)-A_0(\kappa^2)]_\#$ in (c), (d).
  $(\lambda_n,f_n)$ in \eqref{ind+tilde} is the eigensystem of $[A(\kappa^2,\Om)-A(\kappa^2,\widetilde\Om)]_\#$ in (e), (f) and of $[A(\kappa^2,\mho)-A(\kappa^2,\widetilde\mho)]_\#$ in (g), (h).
  $\tau_j=j/20+0.5$ and $\kappa_\ell=\ell/20+1.5$, $j,\ell=0,1\cdots,20$.
  The black dot represents the true value of $(\tau,k)$.}\label{ex_1simultaneous}
\end{figure}

\subsection{Numerical examples for recovering the polygonal obstacle}

We are now ready to recover a rough location and the convex hull of the inclusion $D$.
In this subsection, we assume the constant coefficients $\sigma$ and $q$ have been recovered in the previous subsection.
In view of Theorem \ref{thm211027}, Remark \ref{rem1029-2} and the Picard's theorem (see \cite[Theorem 4.8]{CK19}), we defined the following indicator functions to recover $D$ from the Cauchy data $(f,\pa_\nu u|_{\pa B})$:
\be\label{ind}
I_2(\Omega):=\left[\sum_{n=1}^\infty\frac{|(p(\sigma,q/\sigma),f_n)|^2}{|\lambda_n|}\right]^{-1},\\ \label{indtilde}
\widetilde I_2(\Omega):=\left[\sum_{n=1}^\infty\frac{|(\tilde p(\sigma,q/\sigma),\tilde f_n)|^2}{|\tilde \lambda_n|}\right]^{-1},
\en
where $\Omega$ is a convex Lipschitz domain that contained in $B$, $(\cdot,\cdot)$ denotes the inner product in $L^2(\pa B)$, $(\lambda_n,f_n)$ and $(\tilde\lambda_n,\tilde f_n)$ are the eigensystems of $[A(q/\sigma,\Omega)-A_0(q/\sigma)]_\#$ and $[A(q/\sigma,\Omega)-A(q/\sigma,\widetilde\Om)]_\#$, respectively.
$p(\sigma,q/\sigma)$ is defined as in Theorem \ref{thm-1} with $f$ satisfying the condition in Theorem \ref{53-thm1}, and $\tilde p(\sigma,q/\sigma)$ is defined as in Theorem \ref{thm-2} with $f$ satisfying the condition in Theorem \ref{thm220605-3}.
To indicator the value of $I_2(\Om^{(\ell)})$ for each $\Om^{(\ell)}\!\in\!\mathcal O$, $\pa\Omega^{(\ell)}$ will be plotted in the color given in terms of RGB values in $[0,1]\times[0,1]\times[0,1]$ as follows:
\be\label{1029-3}
\begin{cases}
(V^{(\ell)},1-V^{(\ell)},0) & \text{if}\;V^{(\ell)}\ge0,\\
(0,1+V^{(\ell)},-V^{(\ell)}) & \text{if}\;V^{(\ell)}<0,
\end{cases}
\quad\text{ with }
V^{(\ell)}:=2\times\frac{I_2(\Om^{(\ell)})-I_{min}}{I_{max}-I_{min}}-1,
\en
where $I_{max}\!:=\!\max_{\ell}\{I_2(\Om^{(\ell)})\}$ and $I_{min}\!:=\!\min_{\ell}\{I_2(\Om^{(\ell)})\}$.

\begin{remark}\label{rem220704}
In view of Theorem \ref{thm3.8m_ID} and Remark \ref{rem1029-2} (i), we can define indicator functions similar to \eqref{ind} and \eqref{indtilde} with obstacle $\Om$ replaced by the support $\ov{\mho}$ of a medium.
\end{remark}

\begin{example}\label{Ex2}
Let $B$ and $D$ be the same as in Example \ref{ex1complex}, and the coefficient $\sigma=1$ is known.
Let $\Om_P^{(\ell)}$ be a disk centered at $P$ with radius $\ell/10$ and corresponding $\widetilde\Om_P^{(\ell)}$ be a disk centered at $P$ with radius $\ell/20$.
The boundary condition on $\pa\Om_P^{(\ell)}$ is $\pa_\nu u+iku=0$, while the boundary condition on $\pa\widetilde\Om$ is $\pa_\nu u-iku=0$.
Set $\widehat B_P^{(\ell)}=\widetilde\Om_P^{(\ell)}$ for each $\ell$.
The numerical results for $I_2(\Om_P^{(\ell)})$ defined by \eqref{ind}, $\ell=5,\cdots,30$, with different locations of $P$ and different values of $k$ are shown in Figure \ref{ex3}, while the numerical results for $\widetilde I_2(\Om_P^{(\ell)})$ defined by \eqref{indtilde}, $\ell=5,\cdots,30$, with different locations of $P$ and different values of $k$ are shown in Figure \ref{tilde_ex3}.
Similarly, let $\mho_P^{(\ell)}$ be a disk centered at $P$ with radius $\ell/10$ and corresponding $\widetilde\mho_P^{(\ell)}$ be a disk centered at $P$ with radius $\ell/20$.
The wave number inside $\mho_P^{(\ell)}$ is $k_0=k(2+i)$, while the wave number inside $\widetilde\mho$ is $k(2-i)$.
Set $\widehat B_P^{(\ell)}=\widetilde\Om_P^{(\ell)}$ for each $\ell$.
The numerical results for $I_2(\mho_P^{(\ell)})$ defined by \eqref{ind}, $\ell=5,\cdots,30$, with different locations of $P$ are shown in Figure \ref{exmedium3}, while the numerical results for $\widetilde I_2(\mho_P^{(\ell)})$ defined by \eqref{indtilde}, $\ell=5,\cdots,30$, with different locations of $P$ are shown in Figure \ref{tilde_exmedium3}.
\end{example}

\begin{figure}[htbp]
  \centering
  \subfigure[$P\!\!=\!\!(\!-\!1.5\!,\!0\!),\!k\!\!=\!\!0.5$]{
  \includegraphics[width=0.18\textwidth]{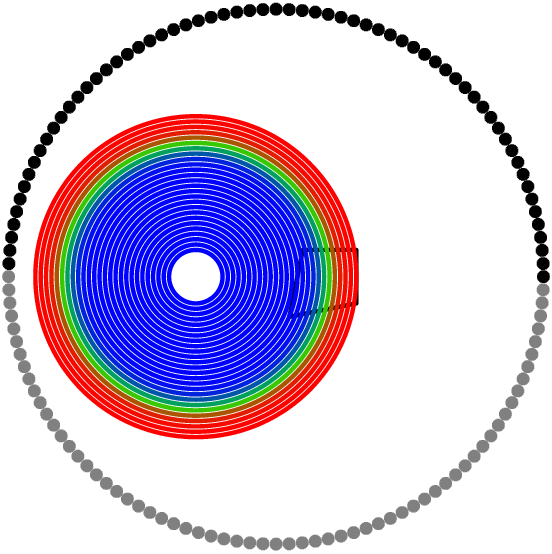}}
  \subfigure[$P\!\!=\!\!(-0.5,\!0),\!k\!=\!1$]{
  \includegraphics[width=0.18\textwidth]{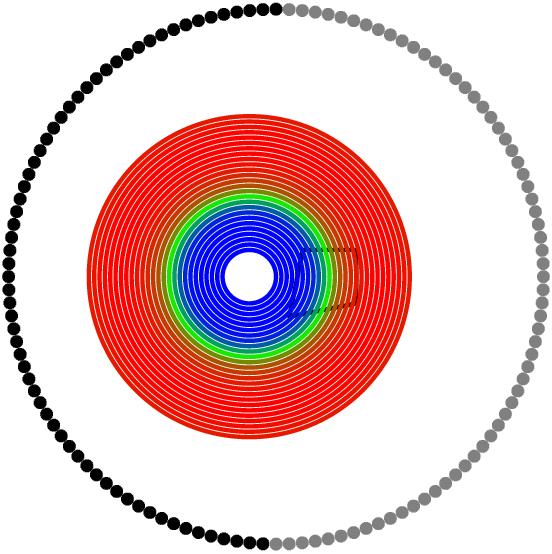}}
  \subfigure[$P=(0,0),k\!=\!2$]{
  \includegraphics[width=0.18\textwidth]{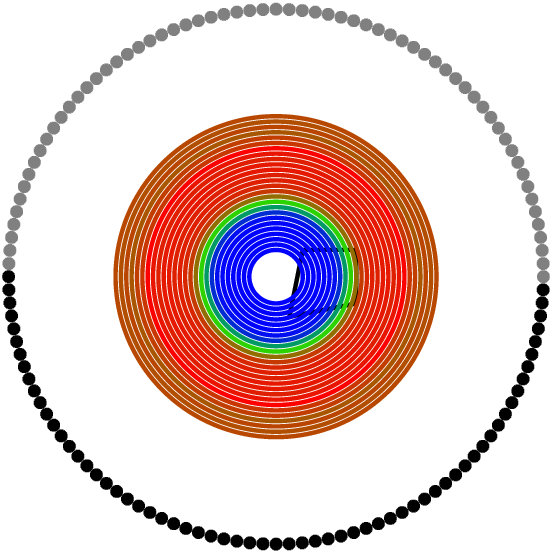}}
  \subfigure[$P\!=\!(0.5,0),k\!=\!4$]{
  \includegraphics[width=0.18\textwidth]{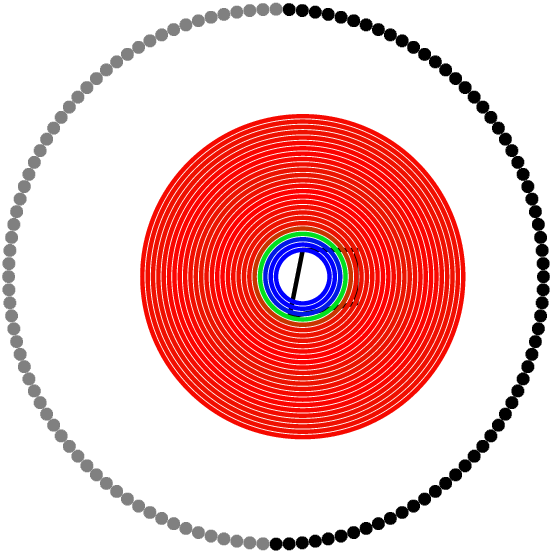}}
  \subfigure[$P\!=\!(1.5,0),k\!=\!8$]{
  \includegraphics[width=0.18\textwidth]{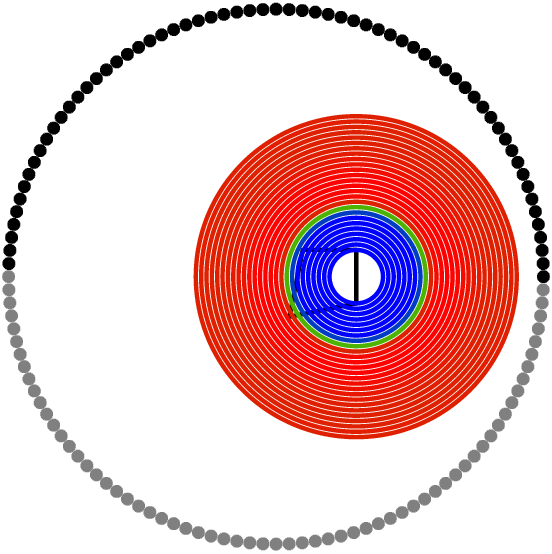}}\\
  \subfigure[$P\!\!=\!\!(\!-\!1.5\!,\!0\!),\!k\!\!=\!\!0.5$]{
  \includegraphics[width=0.18\textwidth]{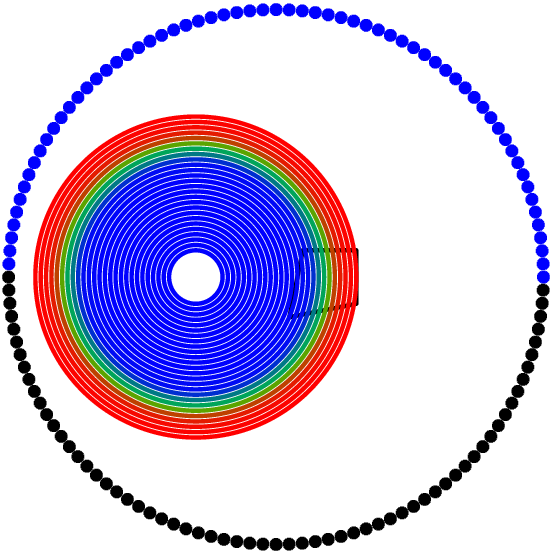}}
  \subfigure[$P\!\!=\!\!(-0.5,\!0),\!k\!=\!1$]{
  \includegraphics[width=0.18\textwidth]{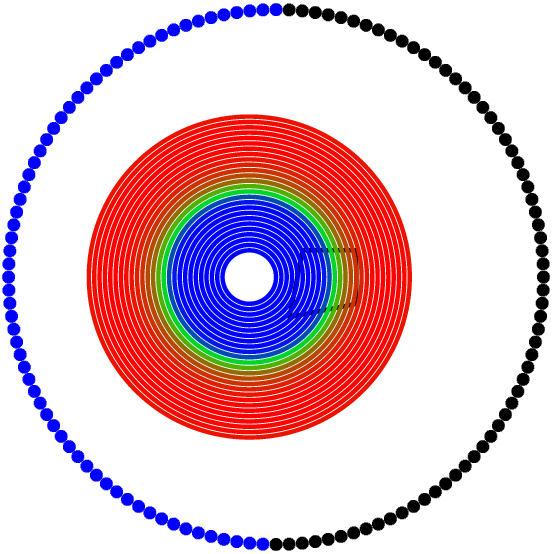}}
  \subfigure[$P=(0,0),k\!=\!2$]{
  \includegraphics[width=0.18\textwidth]{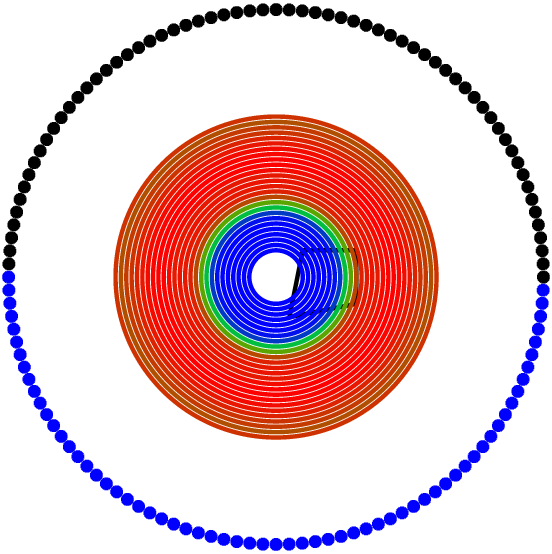}}
  \subfigure[$P\!=\!(0.5,0),k\!=\!4$]{
  \includegraphics[width=0.18\textwidth]{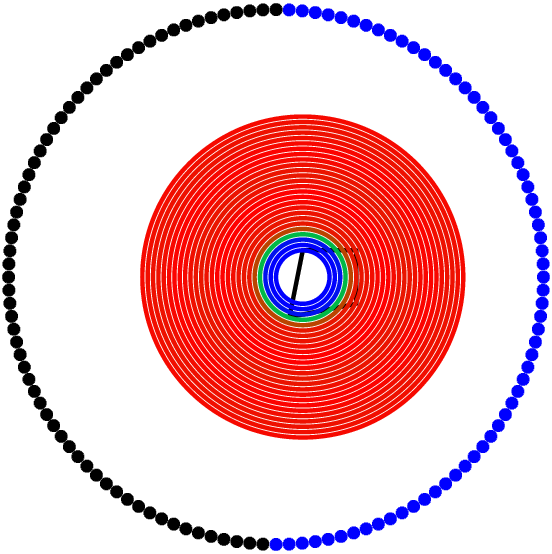}}
  \subfigure[$P\!=\!(1.5,0),k\!=\!8$]{
  \includegraphics[width=0.18\textwidth]{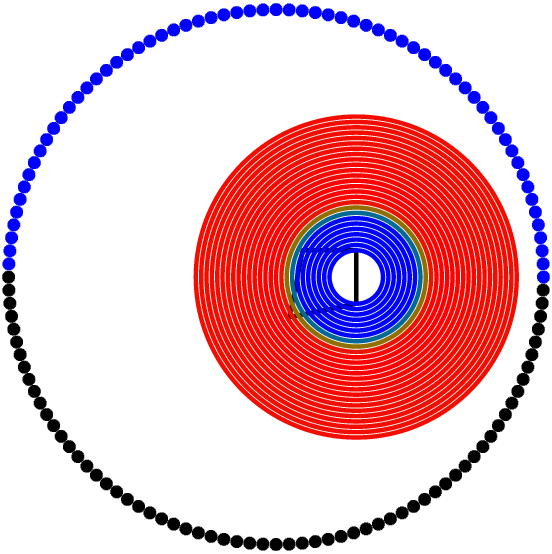}}
  \caption{Numerical results for Example \ref{Ex2}. The black line represents $\pa D$, the circle with knots represents $\pa B$. The boundary value $f$ takes value $0$, $1$, and $2$ at gray, black, and blue knots on $\pa B$, respectively. The colored circles represent $\pa\Om_P^{(\ell)}$ for different $P$ and $\ell$ with its color indicating the value of $I_2(\Om_P^{(\ell)})$ in the sense of (\ref{1029-3}).}\label{ex3}
\end{figure}

\begin{figure}[htbp]
  \centering
  \subfigure[$P\!\!=\!\!(\!-\!1.5\!,\!0\!),\!k\!\!=\!\!0.5$]{
  \includegraphics[width=0.18\textwidth]{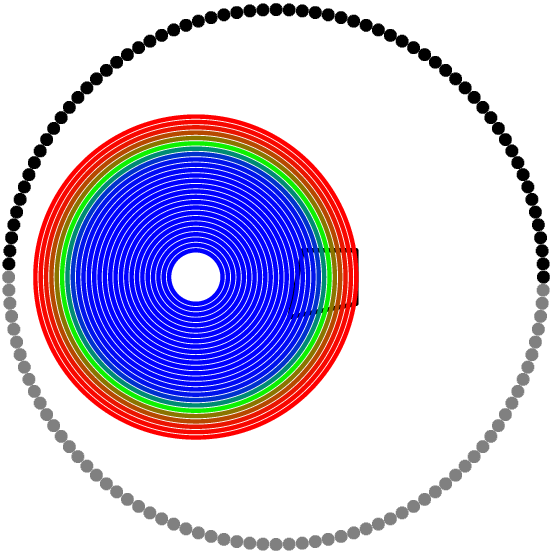}}
  \subfigure[$P\!\!=\!\!(-0.5,\!0),\!k\!=\!1$]{
  \includegraphics[width=0.18\textwidth]{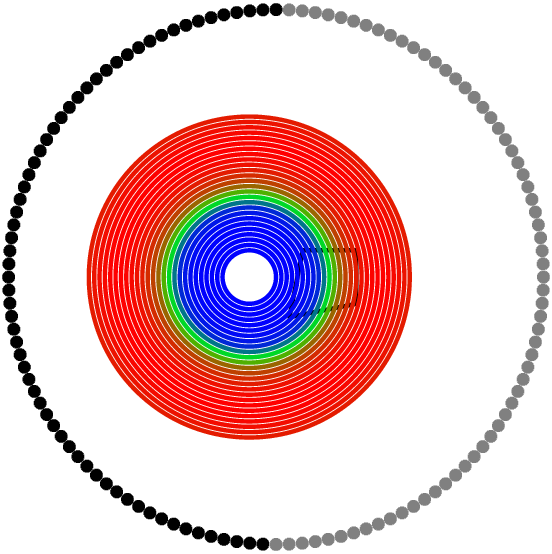}}
  \subfigure[$P=(0,0),k\!=\!2$]{
  \includegraphics[width=0.18\textwidth]{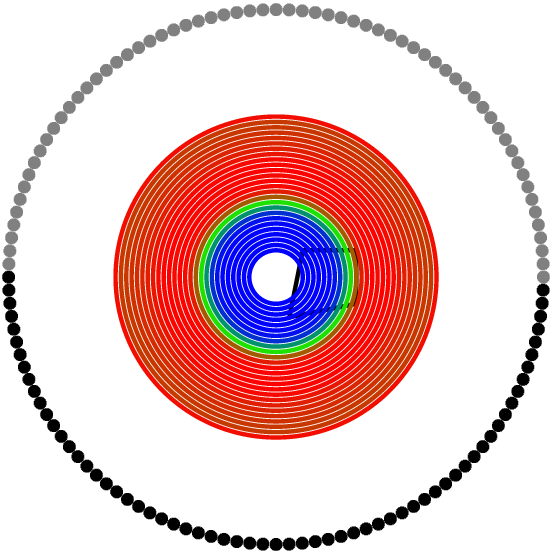}}
  \subfigure[$P\!=\!(0.5,0),k\!=\!4$]{
  \includegraphics[width=0.18\textwidth]{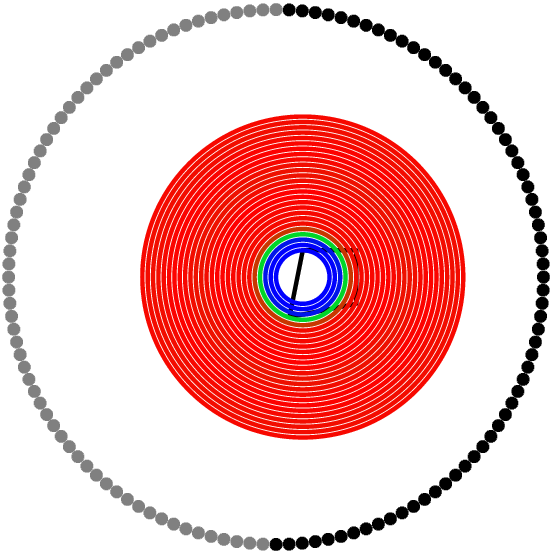}}
  \subfigure[$P\!=\!(1.5,0),k\!=\!8$]{
  \includegraphics[width=0.18\textwidth]{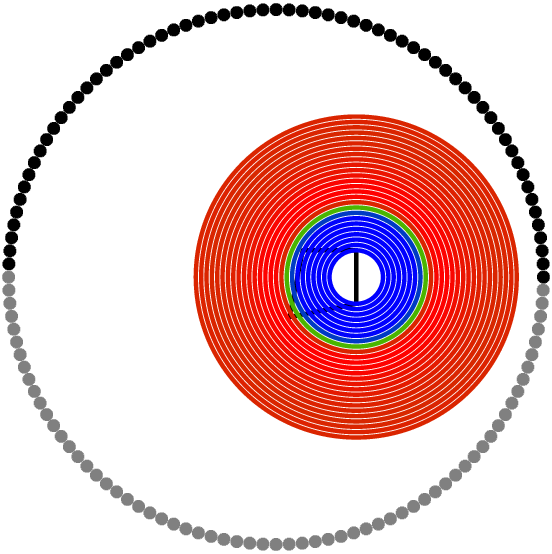}}\\
  \subfigure[$P\!\!=\!\!(\!-\!1.5\!,\!0\!),\!k\!\!=\!\!0.5$]{
  \includegraphics[width=0.18\textwidth]{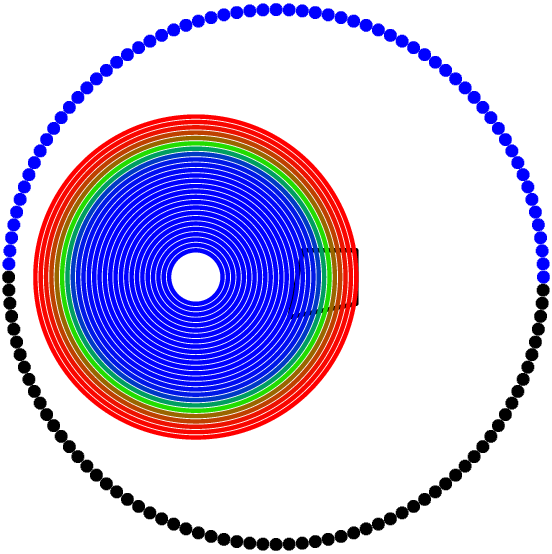}}
  \subfigure[$P\!\!=\!\!(-0.5,\!0),\!k\!=\!1$]{
  \includegraphics[width=0.18\textwidth]{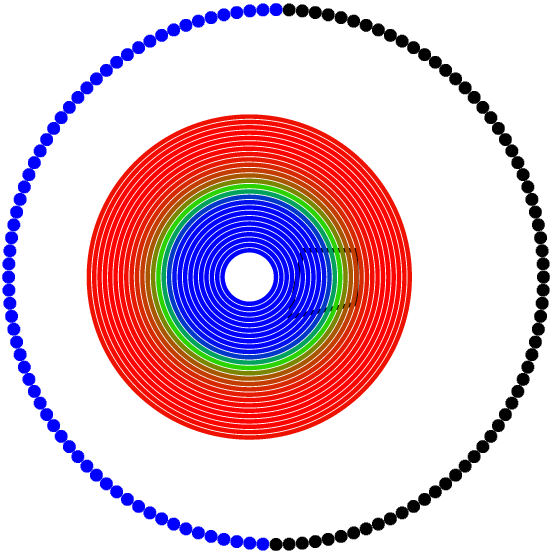}}
  \subfigure[$P=(0,0),k\!=\!2$]{
  \includegraphics[width=0.18\textwidth]{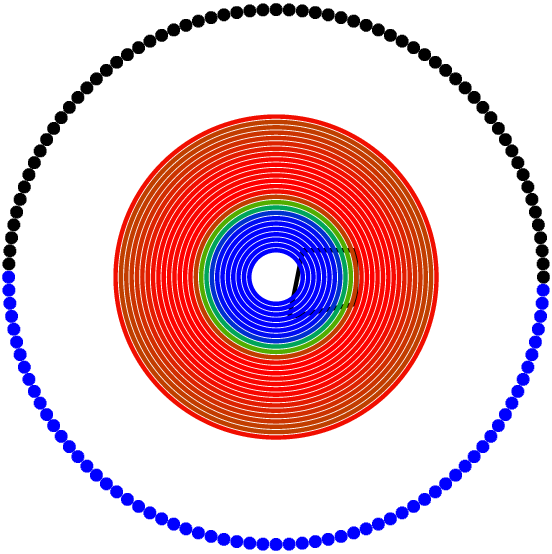}}
  \subfigure[$P\!=\!(0.5,0),k\!=\!4$]{
  \includegraphics[width=0.18\textwidth]{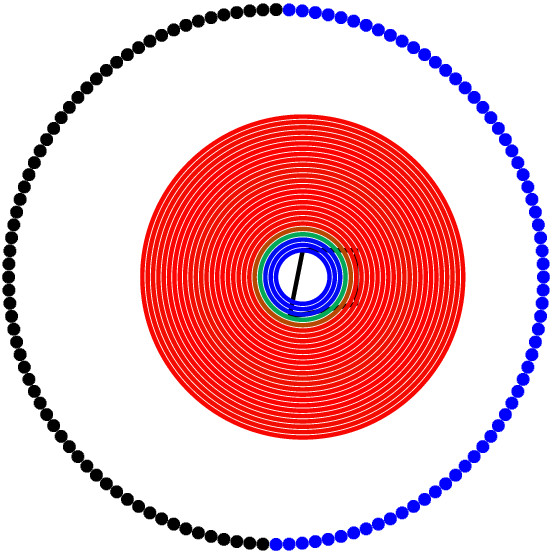}}
  \subfigure[$P\!=\!(1.5,0),k\!=\!8$]{
  \includegraphics[width=0.18\textwidth]{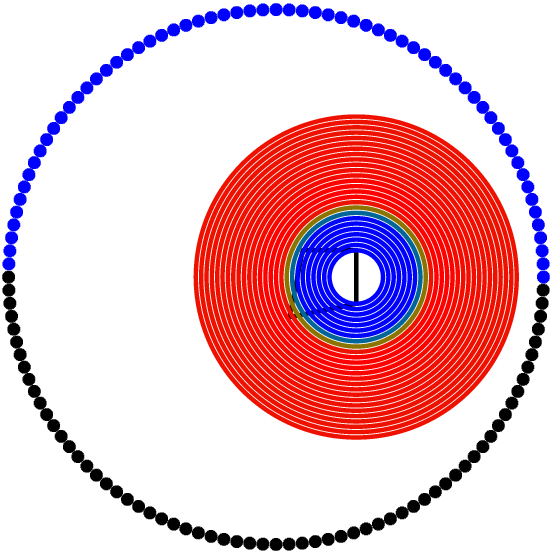}}
  \caption{Numerical results for Example \ref{Ex2}. The black line represents $\pa D$, the circle with knots represents $\pa B$. The boundary value $f$ takes value $0$, $1$, and $2$ at gray, black, and blue knots on $\pa B$, respectively. The colored circles represent $\pa\Om_P^{(\ell)}$ for different $P$ and $\ell$ with its color indicating the value of $\widetilde I_2(\Om_P^{(\ell)})$ in the sense of (\ref{1029-3}).}\label{tilde_ex3}
\end{figure}

\begin{figure}[htbp]
  \centering
  \subfigure[$P\!\!=\!\!(\!-\!1.5\!,\!0\!),\!k\!\!=\!\!0.5$]{
  \includegraphics[width=0.18\textwidth]{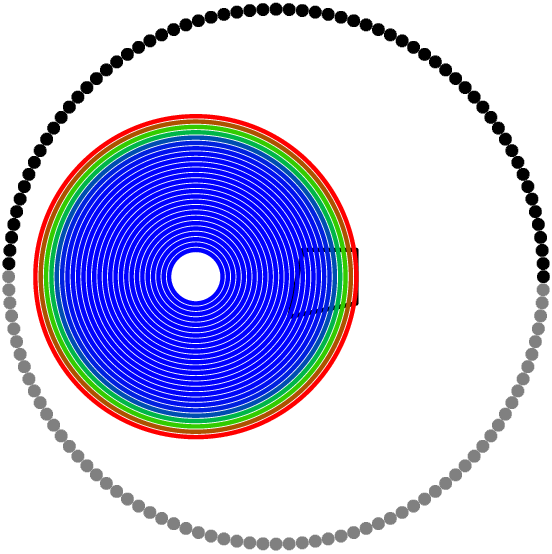}}
  \subfigure[$P\!\!=\!\!(-0.5,\!0),\!k\!=\!1$]{
  \includegraphics[width=0.18\textwidth]{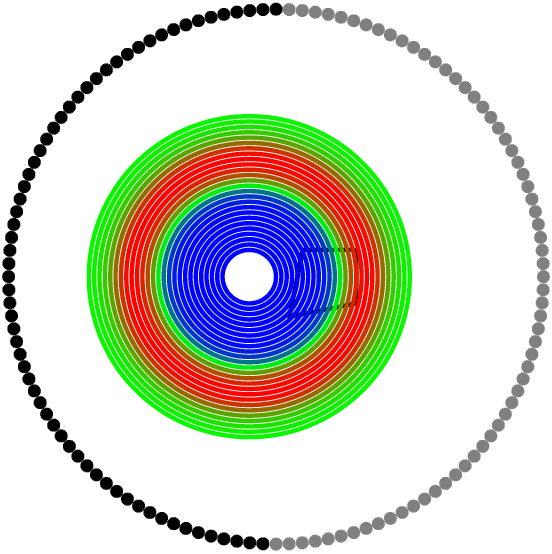}}
  \subfigure[$P=(0,0),k\!=\!2$]{
  \includegraphics[width=0.18\textwidth]{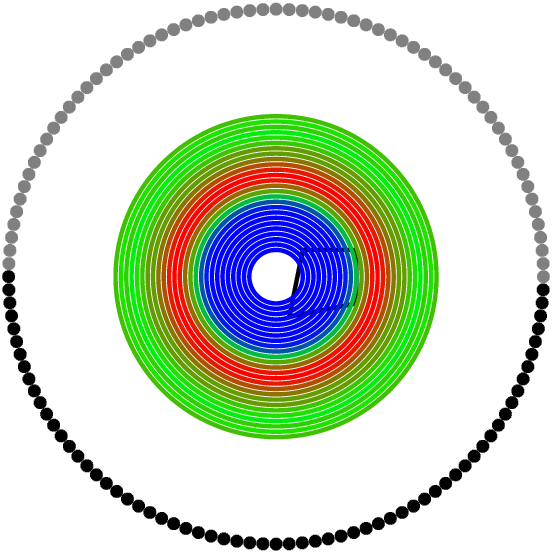}}
  \subfigure[$P\!=\!(0.5,0),k\!=\!4$]{
  \includegraphics[width=0.18\textwidth]{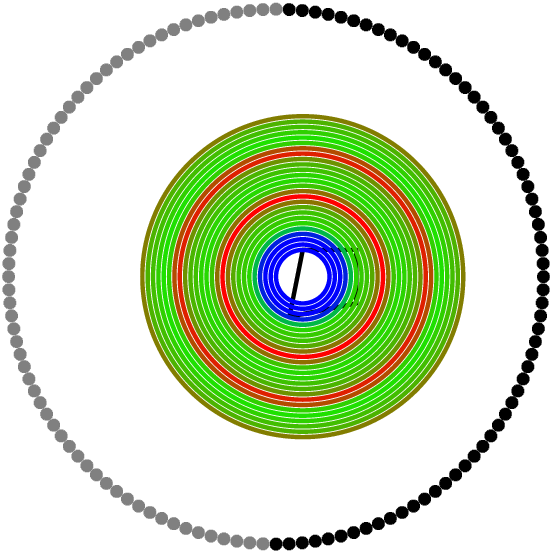}}
  \subfigure[$P\!=\!(1.5,0),k\!=\!6$]{
  \includegraphics[width=0.18\textwidth]{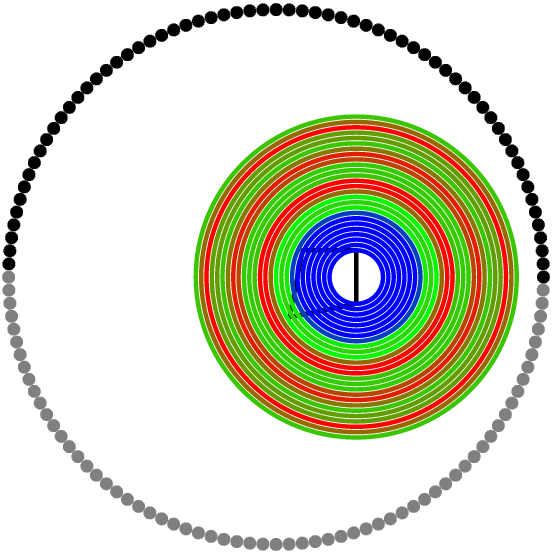}}\\
  \subfigure[$P\!\!=\!\!(\!-\!1.5\!,\!0\!),\!k\!\!=\!\!0.5$]{
  \includegraphics[width=0.18\textwidth]{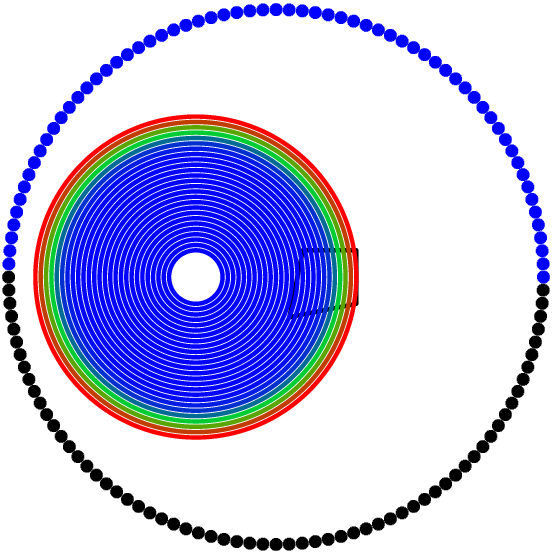}}
  \subfigure[$P\!\!=\!\!(-0.5,\!0),\!k\!=\!1$]{
  \includegraphics[width=0.18\textwidth]{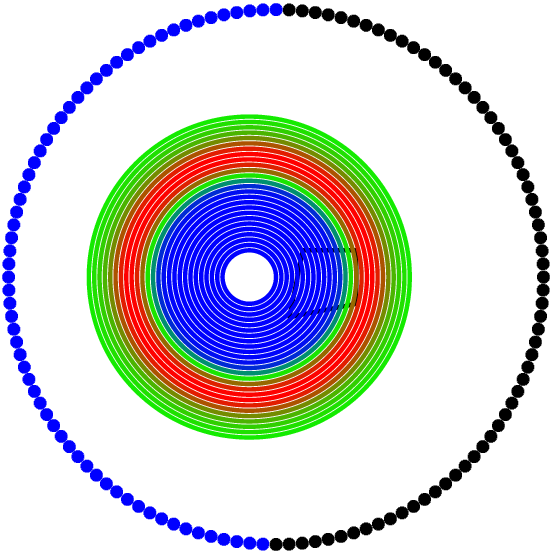}}
  \subfigure[$P=(0,0),k\!=\!2$]{
  \includegraphics[width=0.18\textwidth]{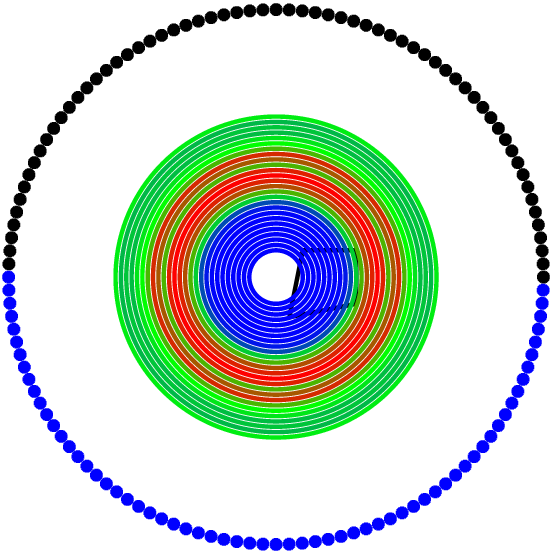}}
  \subfigure[$P\!=\!(0.5,0),k\!=\!4$]{
  \includegraphics[width=0.18\textwidth]{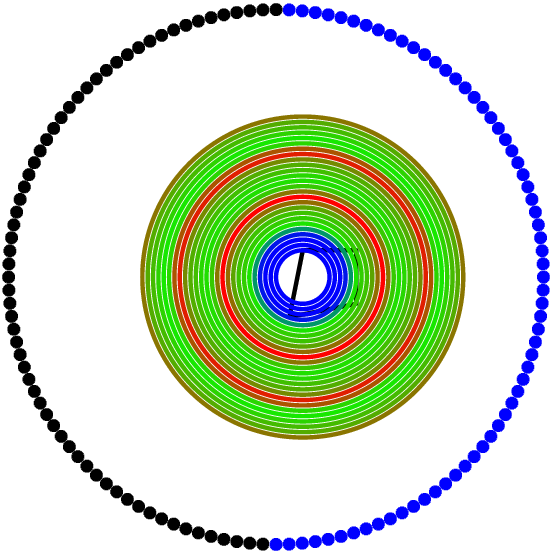}}
  \subfigure[$P\!=\!(1.5,0),k\!=\!6$]{
  \includegraphics[width=0.18\textwidth]{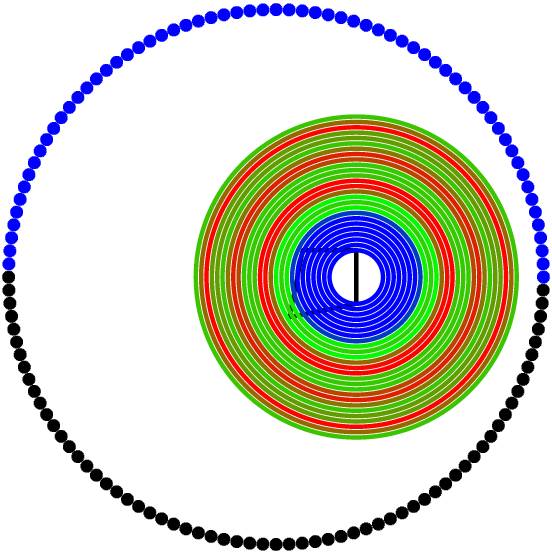}}
  \caption{Numerical results for Example \ref{Ex2}. The black line represents $\pa D$, the circle with knots represents $\pa B$. The boundary value $f$ takes value $0$, $1$, and $2$ at gray, black, and blue knots on $\pa B$, respectively. The colored circles represent $\pa\mho_P^{(\ell)}$ for different $P$ and $\ell$ with its color indicating the value of $I_2(\mho_P^{(\ell)})$ in the sense of (\ref{1029-3}).}\label{exmedium3}
\end{figure}

\begin{figure}[htbp]
  \centering
  \subfigure[$P\!\!=\!\!(\!-\!1.5\!,\!0\!),\!k\!\!=\!\!0.5$]{
  \includegraphics[width=0.18\textwidth]{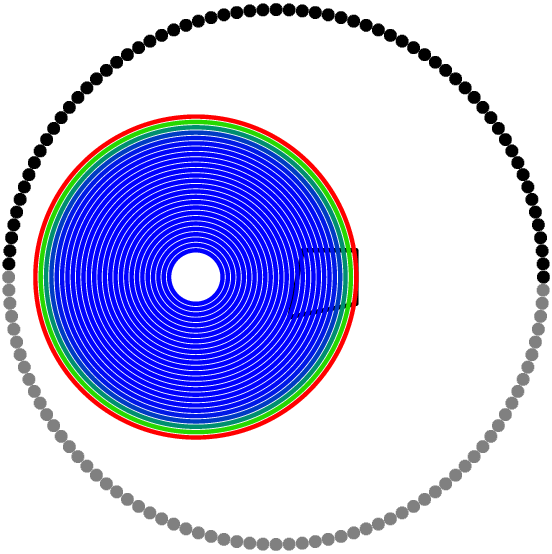}}
  \subfigure[$P\!\!=\!\!(-0.5,\!0),\!k\!=\!1$]{
  \includegraphics[width=0.18\textwidth]{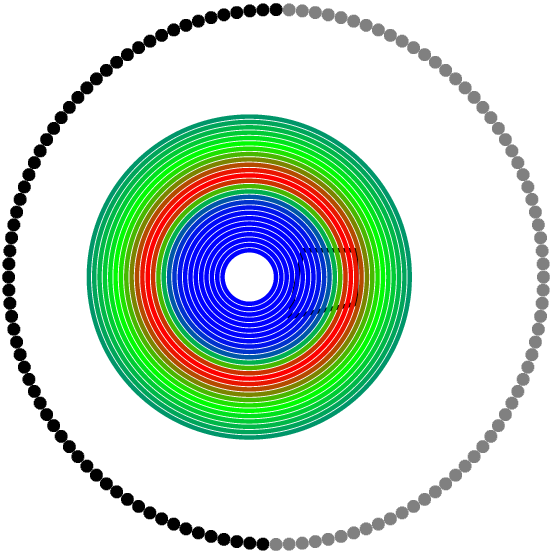}}
  \subfigure[$P=(0,0),k\!=\!2$]{
  \includegraphics[width=0.18\textwidth]{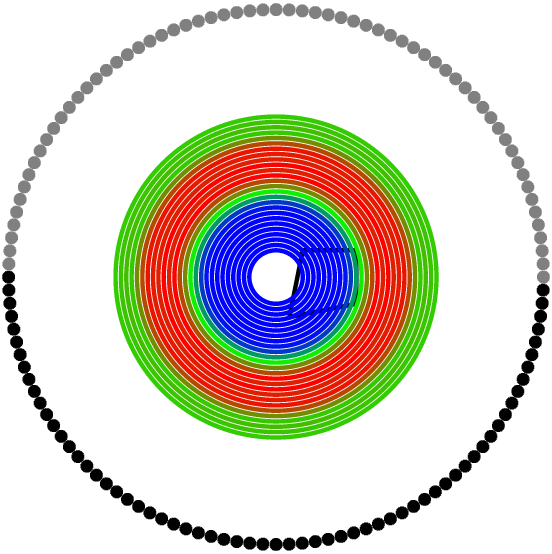}}
  \subfigure[$P\!=\!(0.5,0),k\!=\!4$]{
  \includegraphics[width=0.18\textwidth]{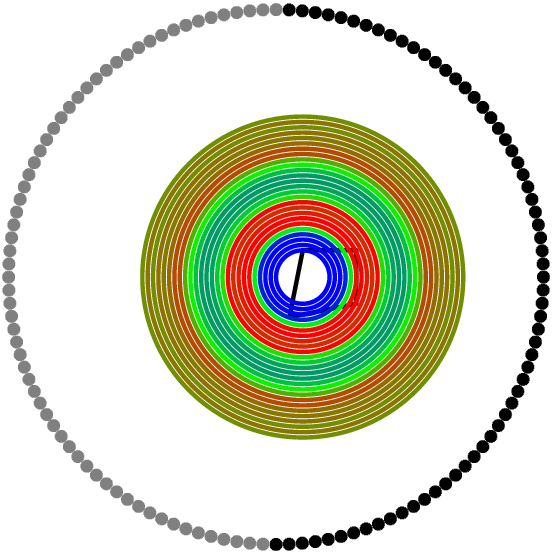}}
  \subfigure[$P\!=\!(1.5,0),k\!=\!6$]{
  \includegraphics[width=0.18\textwidth]{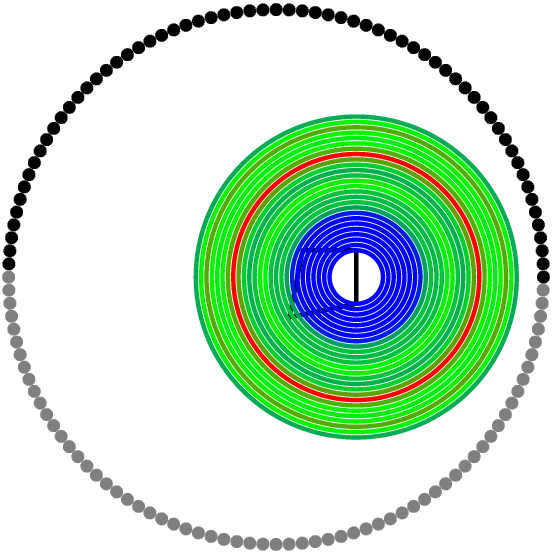}}\\
  \subfigure[$P\!\!=\!\!(\!-\!1.5\!,\!0\!),\!k\!\!=\!\!0.5$]{
  \includegraphics[width=0.18\textwidth]{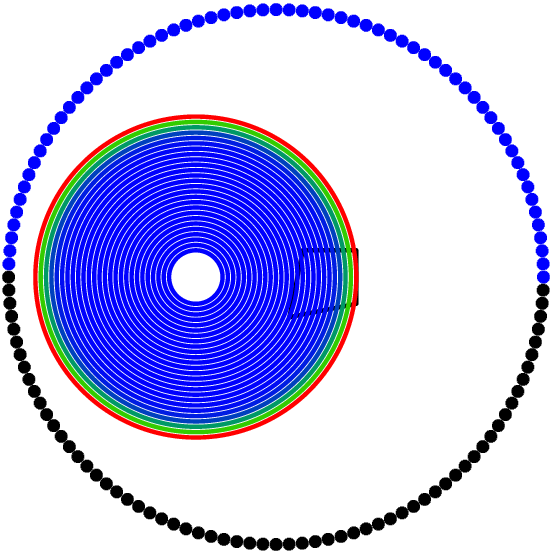}}
  \subfigure[$P\!\!=\!\!(-0.5,\!0),\!k\!=\!1$]{
  \includegraphics[width=0.18\textwidth]{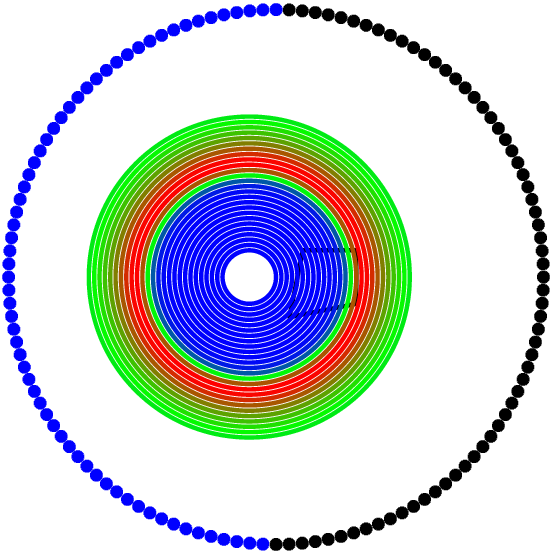}}
  \subfigure[$P=(0,0),k\!=\!2$]{
  \includegraphics[width=0.18\textwidth]{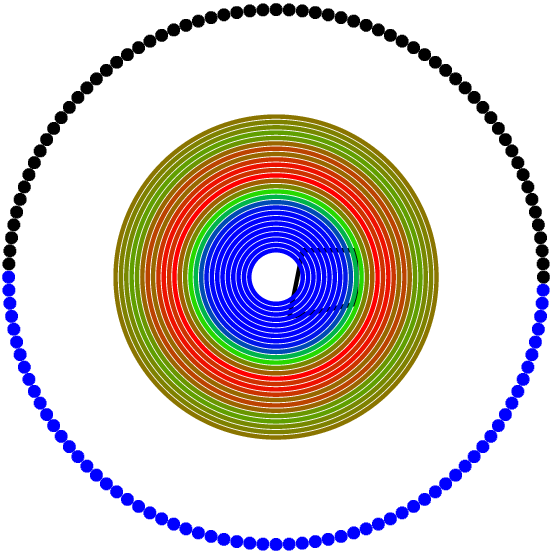}}
  \subfigure[$P\!=\!(0.5,0),k\!=\!4$]{
  \includegraphics[width=0.18\textwidth]{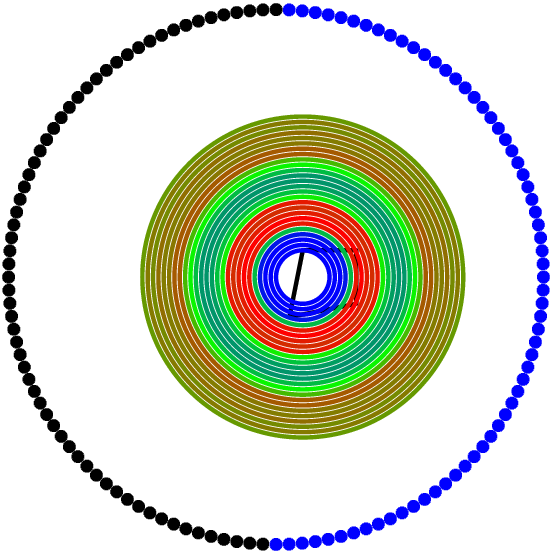}}
  \subfigure[$P\!=\!(1.5,0),k\!=\!6$]{
  \includegraphics[width=0.18\textwidth]{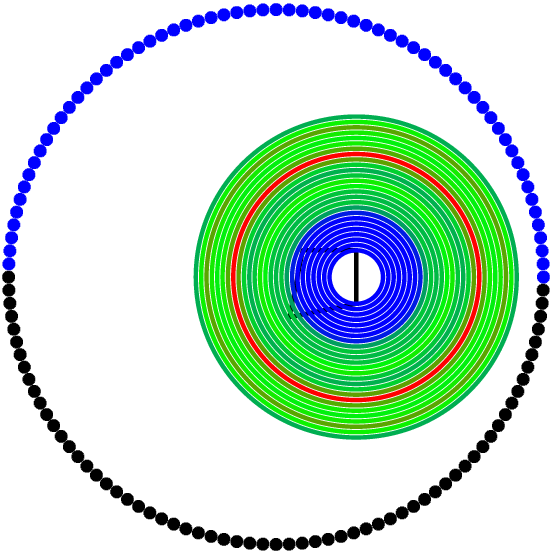}}
  \caption{Numerical results for Example \ref{Ex2}. The black line represents $\pa D$, the circle with knots represents $\pa B$. The boundary value $f$ takes value $0$, $1$, and $2$ at gray, black, and blue knots on $\pa B$, respectively. The colored circles represent $\pa\mho_P^{(\ell)}$ for different $P$ and $\ell$ with its color indicating the value of $\widetilde I_2(\mho_P^{(\ell)})$ in the sense of (\ref{1029-3}).}\label{tilde_exmedium3}
\end{figure}

From the numerical results shown in Figures \ref{ex3}--\ref{tilde_exmedium3}, we see that the rough location of $D$ can be reconstructed from a single pair of Cauchy data to (\ref{1})--(\ref{3}).
Moreover, we can imagine that Remark \ref{rem1029-2} can be numerically verified if the values of $I_2(\Om)$ for all possible domains $\Om$ in $\mathcal O$ are calculated.
For the reconstruction of the convex hull of the target shape, we refer to \cite{XH2024}.

\begin{example}\label{Ex4}
Let $D$ and $B$ be the same as in Example \ref{ex1complex}, and the coefficients $\sigma=1$ and $q=4$ ($k=2$) are known.
We compare the values of $I_2(\Om_{P_j}(r))$ defined by \eqref{ind} and $\widetilde I_2(\Om_{P_j}(r))$ defined by \eqref{indtilde} where $\Om_{P_j}(r)$ is a disk with radius $r\in\{1,1/2,1/4,1/8\}$ centered at $P_j:=(2rj_1-2,2rj_2-2)$ for $j=(j_1,j_2)\in\{0,1,\cdots,2/r\}^2$, and $\widetilde\Om_{P_j}(r/2)$ is a disk with radius $r/2$ centered at $P_j$.
The boundary condition on $\pa\Om_{P_j}(r)$ is $\pa_\nu u+iku=0$, while the boundary condition on $\pa\widetilde\Om_{P_j}(r/2)$ is $\pa_\nu u-iku=0$ for each $j$ and $r$.
Set $\widehat B=\widetilde\Om_{P_j}(r/2)$ for each $j$ and $r$.
The numerical results for $I_2(\Om_{P_j}(r))$ and $\widetilde I_2(\Om_{P_j}(r))$ are shown in Figure \ref{ex4} and Figure \ref{tilde_ex4}, respectively.
Similarly, the numerical results for $I_2(\mho_{P_j}(r))$ defined by \eqref{ind} and $\widetilde I_2(\mho_{P_j}(r))$ defined by \eqref{indtilde} with $\ov{\mho_{P_j}(r)}$ denoting the support of an medium with interior wave number $k(2+i)$ are shown in Figure \ref{exmedium4} and Figure \ref{tilde_exmedium4}, respectively, where $\mho_{P_j}(r)$ is also a disk with radius $r\in\{1,1/2,1/4,1/8\}$ centered at $P_j:=(2rj_1-2,2rj_2-2)$ for $j=(j_1,j_2)\in\{0,1,\cdots,2/r\}^2$, and $\widetilde\mho_{P_j}(r/2)$ is a disk with radius $r/2$ centered at $P_j$, the wave number inside $\widetilde\mho_{P_j}(r/2)$ is $k(2-i)$, $\widehat B=\widetilde\mho_{P_j}(r/2)$ for each $j$ and $r$.
\end{example}

\begin{figure}[htbp]
  \centering
  \subfigure[$r=1$]{
  \includegraphics[width=0.18\textwidth]{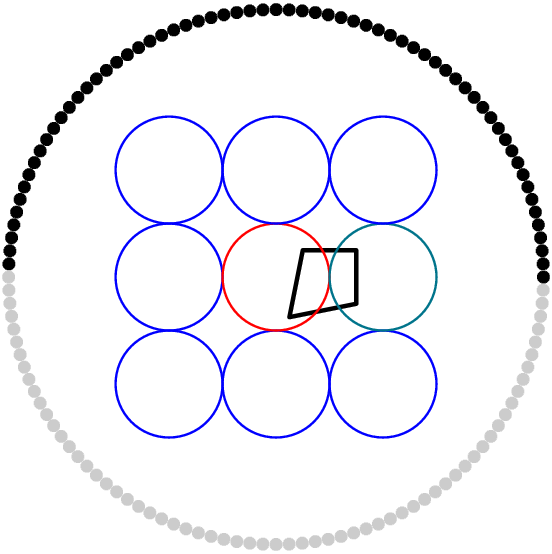}}
  \subfigure[$r=1/2$]{
  \includegraphics[width=0.18\textwidth]{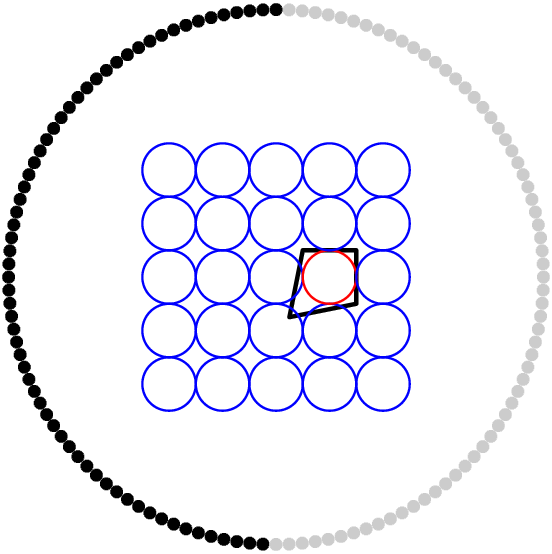}}
  \subfigure[$r=1/4$]{
  \includegraphics[width=0.18\textwidth]{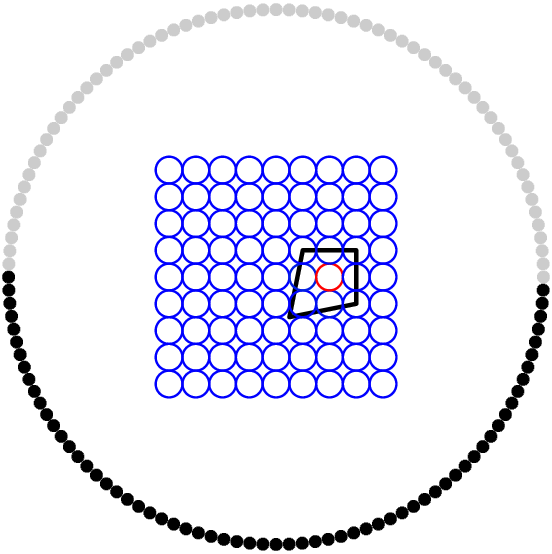}}
  \subfigure[$r=1/4$]{
  \includegraphics[width=0.18\textwidth]{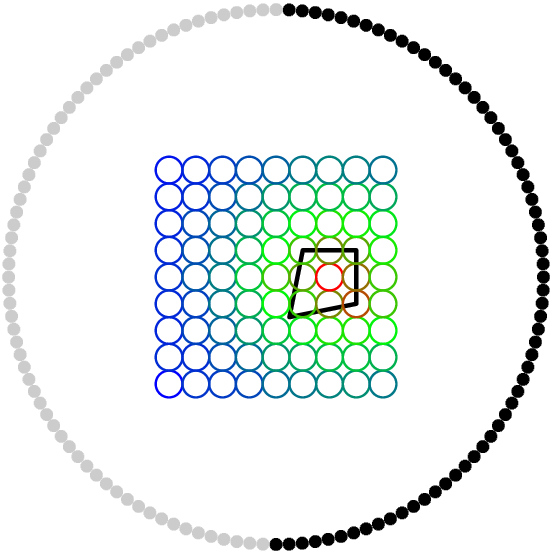}}
  \subfigure[$r=1/8$]{
  \includegraphics[width=0.18\textwidth]{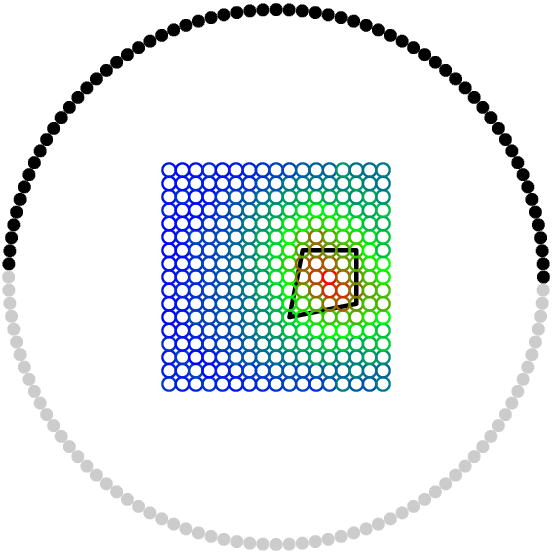}}\\
  \subfigure[$r=1$]{
  \includegraphics[width=0.18\textwidth]{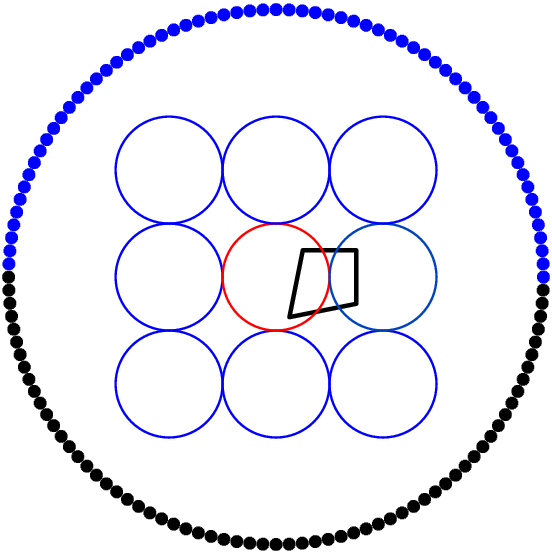}}
  \subfigure[$r=1/2$]{
  \includegraphics[width=0.18\textwidth]{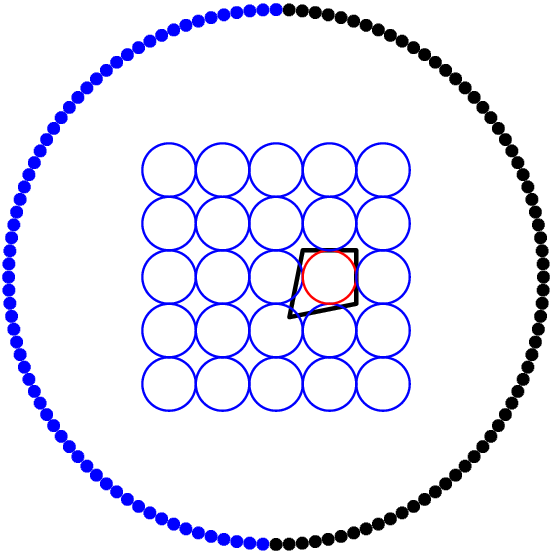}}
  \subfigure[$r=1/4$]{
  \includegraphics[width=0.18\textwidth]{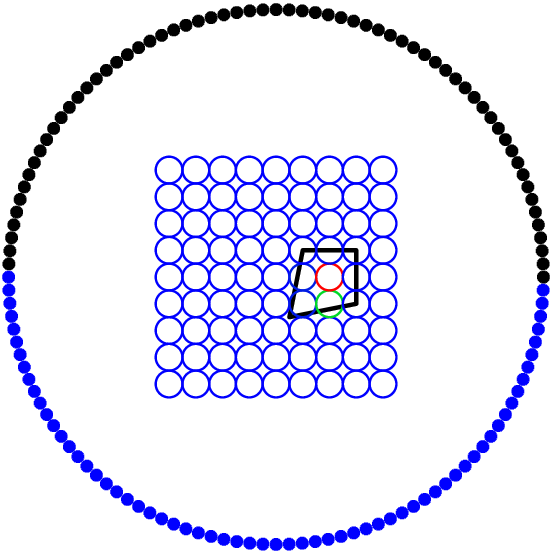}}
  \subfigure[$r=1/4$]{
  \includegraphics[width=0.18\textwidth]{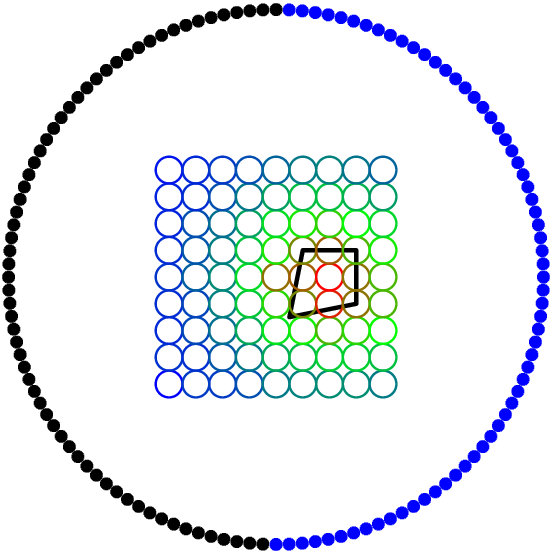}}
  \subfigure[$r=1/8$]{
  \includegraphics[width=0.18\textwidth]{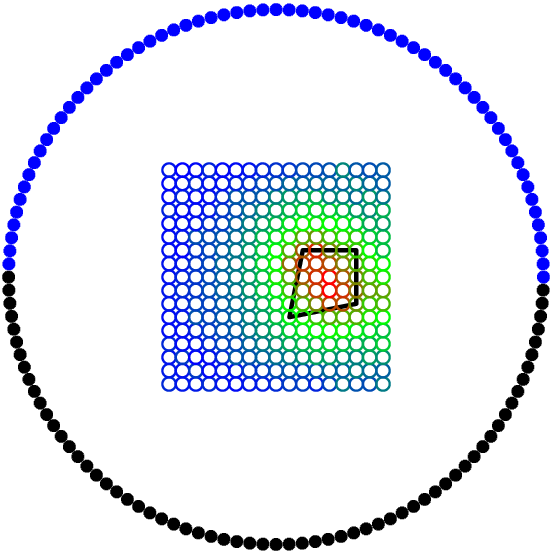}}
  \caption{Numerical results for Example \ref{Ex4}.
  The black line represents $\pa D$, the circle with knots represents $\pa B$.
  The boundary value $f$ takes value $0$, $1$, and $2$ at gray, black, and blue knots on $\pa B$, respectively.
  The colored circles represent $\pa\Om_P^{(\ell)}$ for different $P$ and $\ell$ with its color indicating the value of $I_2(\Om_P^{(\ell)})$ for (a)--(c) and (f)--(h), and $\ln(I_2(\Om_P^{(\ell)}))$ for (d), (e), (i) and (j), respectively, in the sense of (\ref{1029-3}).}\label{ex4}
\end{figure}

\begin{figure}[htbp]
  \centering
  \subfigure[$r=1$]{
  \includegraphics[width=0.18\textwidth]{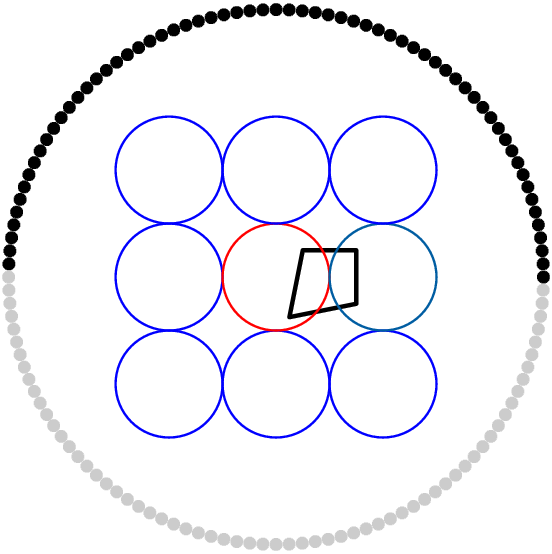}}
  \subfigure[$r=1/2$]{
  \includegraphics[width=0.18\textwidth]{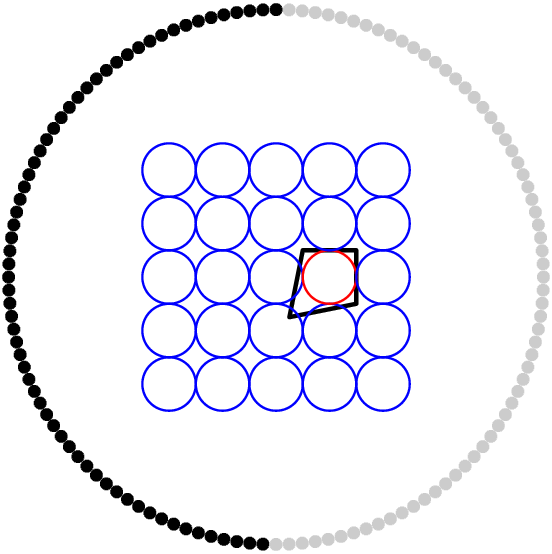}}
  \subfigure[$r=1/4$]{
  \includegraphics[width=0.18\textwidth]{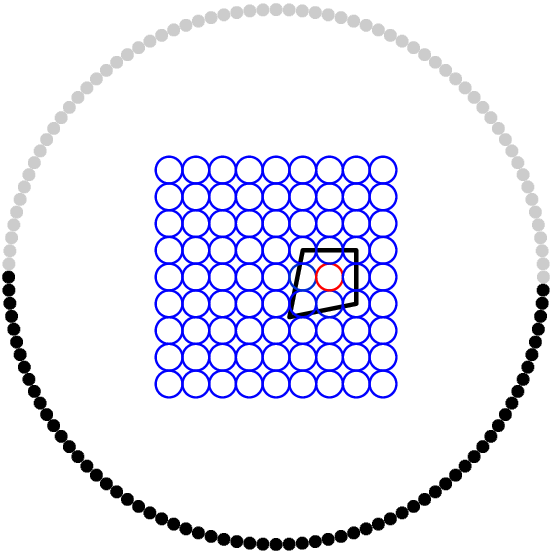}}
  \subfigure[$r=1/4$]{
  \includegraphics[width=0.18\textwidth]{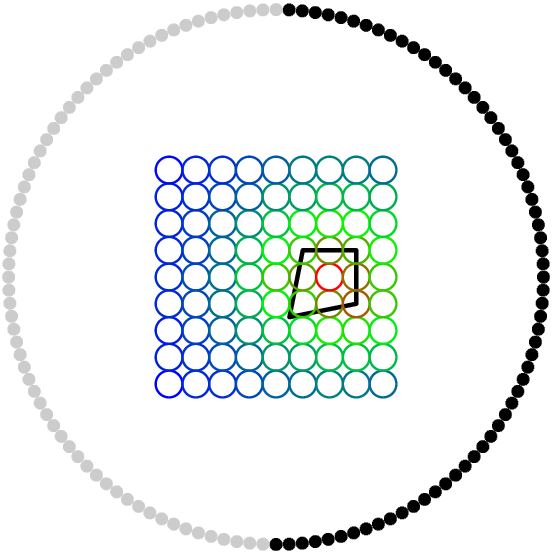}}
  \subfigure[$r=1/8$]{
  \includegraphics[width=0.18\textwidth]{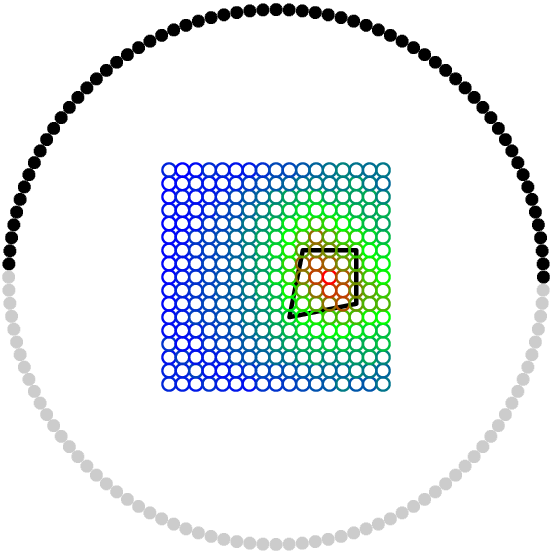}}\\
  \subfigure[$r=1$]{
  \includegraphics[width=0.18\textwidth]{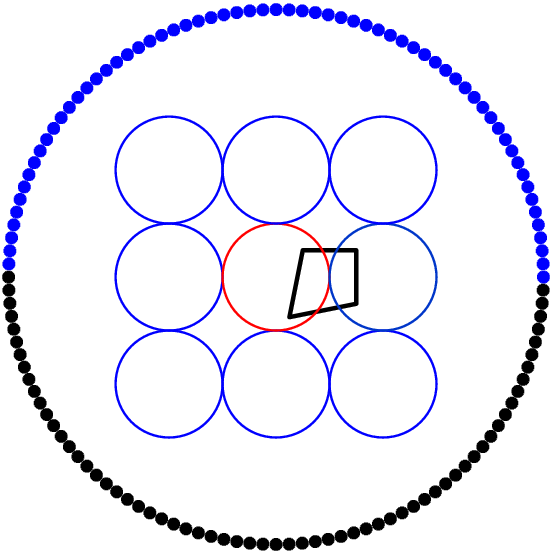}}
  \subfigure[$r=1/2$]{
  \includegraphics[width=0.18\textwidth]{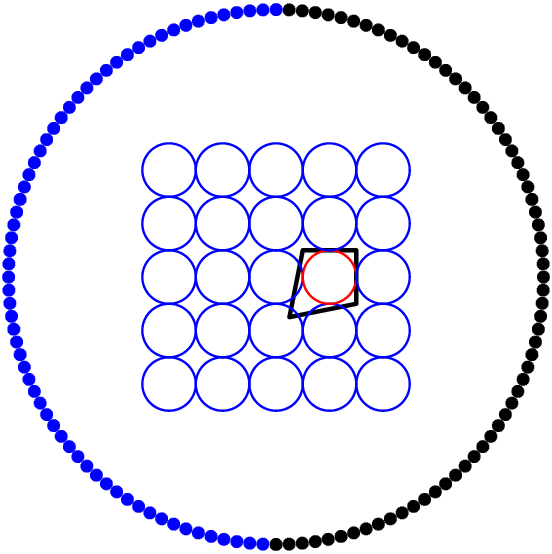}}
  \subfigure[$r=1/4$]{
  \includegraphics[width=0.18\textwidth]{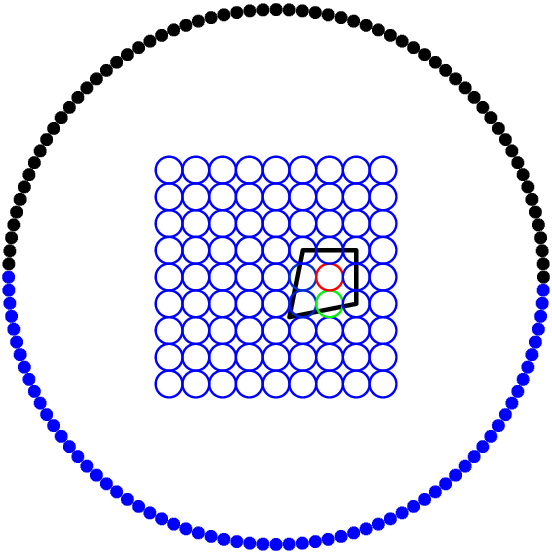}}
  \subfigure[$r=1/4$]{
  \includegraphics[width=0.18\textwidth]{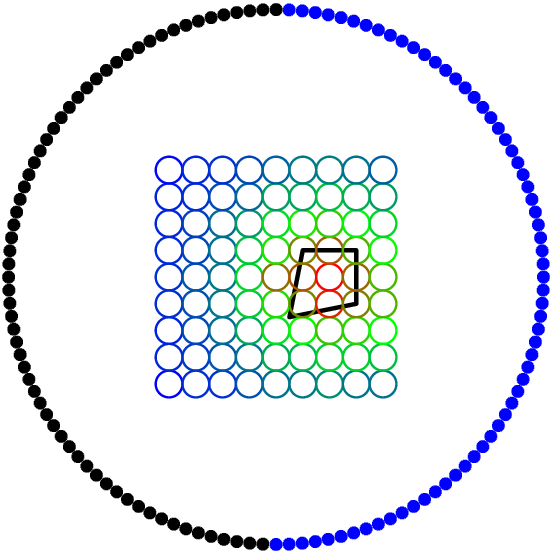}}
  \subfigure[$r=1/8$]{
  \includegraphics[width=0.18\textwidth]{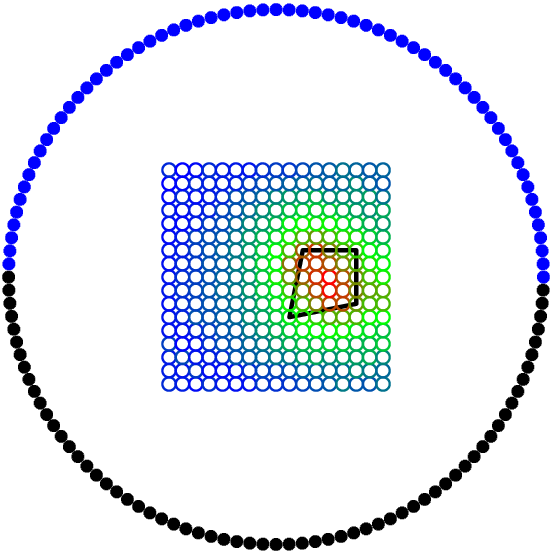}}
  \caption{Numerical results for Example \ref{Ex4}.
  The black line represents $\pa D$, the circle with knots represents $\pa B$.
  The boundary value $f$ takes value $0$, $1$, and $2$ at gray, black, and blue knots on $\pa B$, respectively.
  The colored circles represent $\pa\Om_P^{(\ell)}$ for different $P$ and $\ell$ with its color indicating the value of $\widetilde I_2(\Om_P^{(\ell)})$ for (a)--(c) and (f)--(h), and $\ln(\widetilde I_2(\Om_P^{(\ell)}))$ for (d), (e), (i) and (j), respectively, in the sense of (\ref{1029-3}).}\label{tilde_ex4}
\end{figure}

\begin{figure}[htbp]
  \centering
  \subfigure[$r=1$]{
  \includegraphics[width=0.18\textwidth]{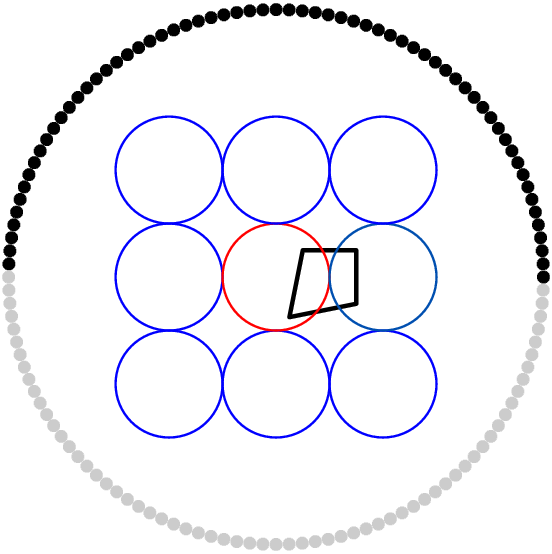}}
  \subfigure[$r=1/2$]{
  \includegraphics[width=0.18\textwidth]{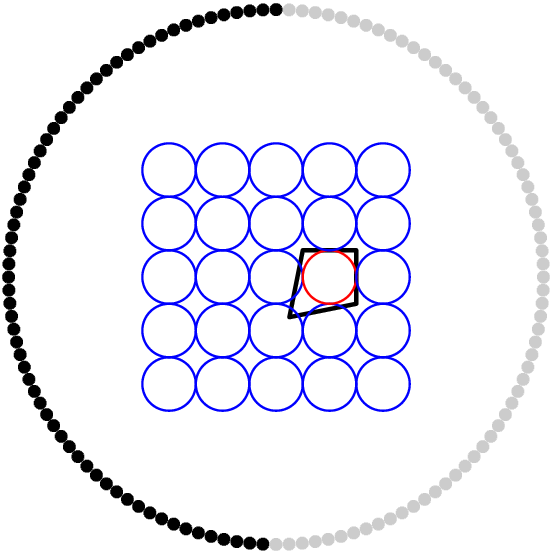}}
  \subfigure[$r=1/4$]{
  \includegraphics[width=0.18\textwidth]{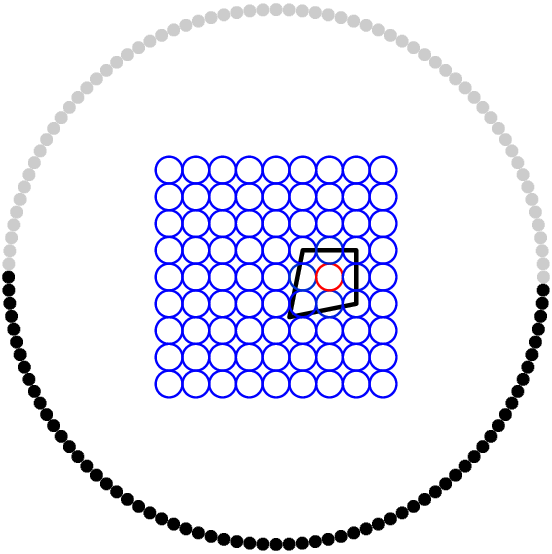}}
  \subfigure[$r=1/4$]{
  \includegraphics[width=0.18\textwidth]{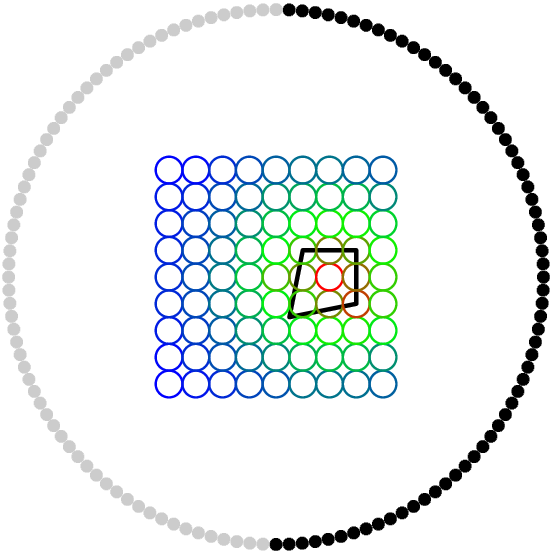}}
  \subfigure[$r=1/8$]{
  \includegraphics[width=0.18\textwidth]{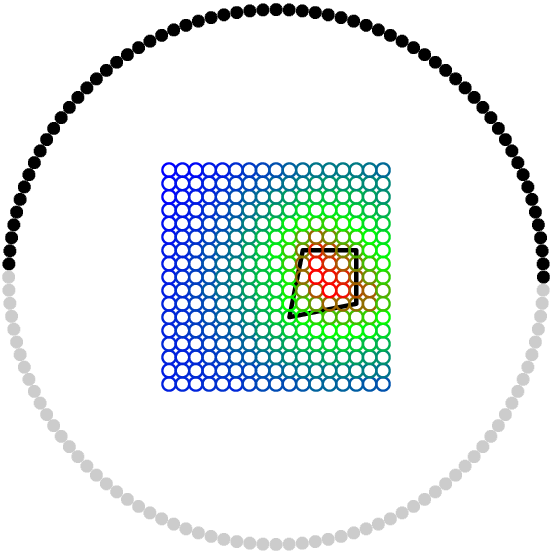}}\\
  \subfigure[$r=1$]{
  \includegraphics[width=0.18\textwidth]{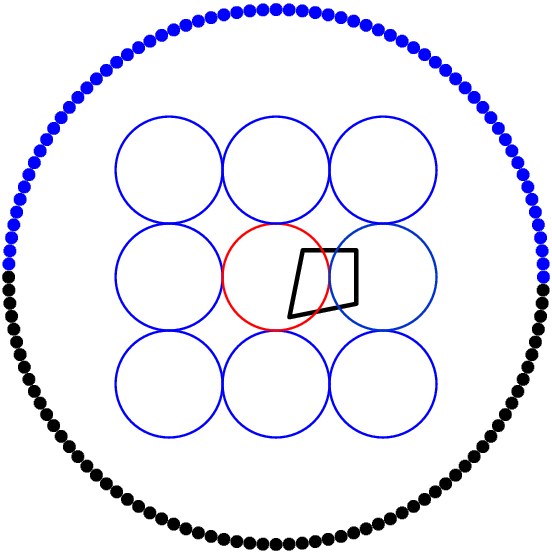}}
  \subfigure[$r=1/2$]{
  \includegraphics[width=0.18\textwidth]{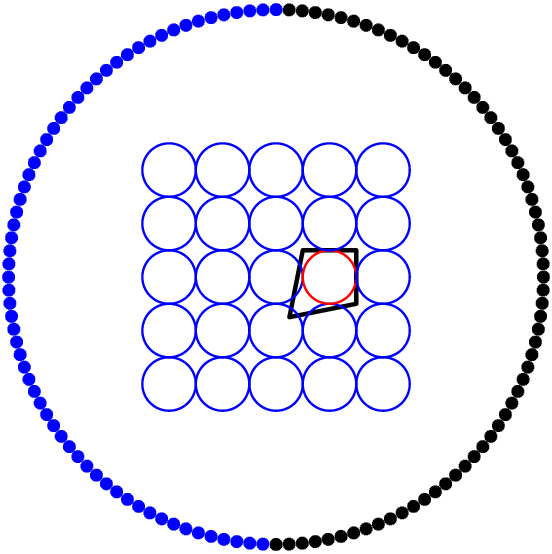}}
  \subfigure[$r=1/4$]{
  \includegraphics[width=0.18\textwidth]{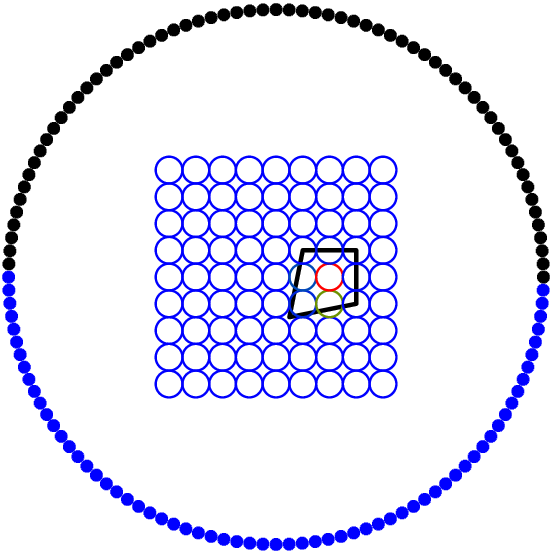}}
  \subfigure[$r=1/4$]{
  \includegraphics[width=0.18\textwidth]{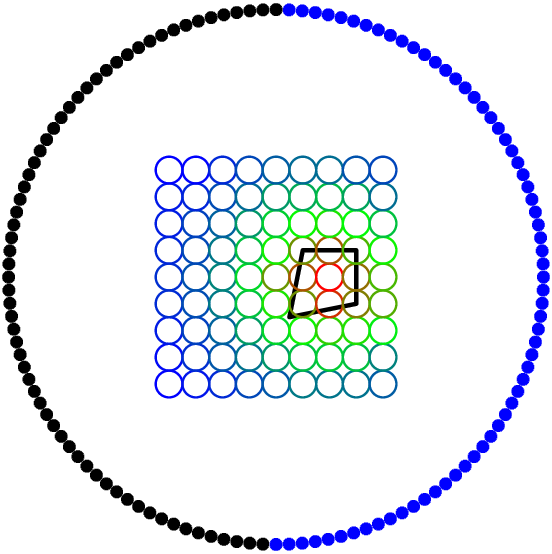}}
  \subfigure[$r=1/8$]{
  \includegraphics[width=0.18\textwidth]{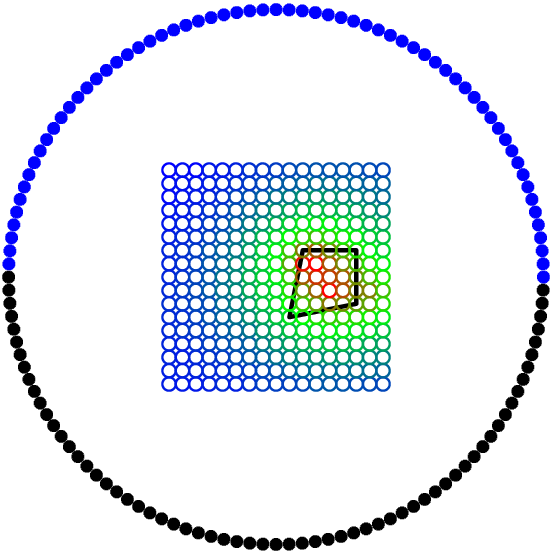}}
  \caption{Numerical results for Example \ref{Ex4}.
  The black line represents $\pa D$, the circle with knots represents $\pa B$.
  The boundary value $f$ takes value $0$, $1$, and $2$ at gray, black, and blue knots on $\pa B$, respectively.
  The colored circles represent $\pa\mho_P^{(\ell)}$ for different $P$ and $\ell$ with its color indicating the value of $I_2(\mho_P^{(\ell)})$ for (a)--(c) and (f)--(h), and $\ln(I_2(\mho_P^{(\ell)}))$ for (d), (e), (i) and (j), respectively, in the sense of (\ref{1029-3}).}\label{exmedium4}
\end{figure}

\begin{figure}[htbp]
  \centering
  \subfigure[$r=1$]{
  \includegraphics[width=0.18\textwidth]{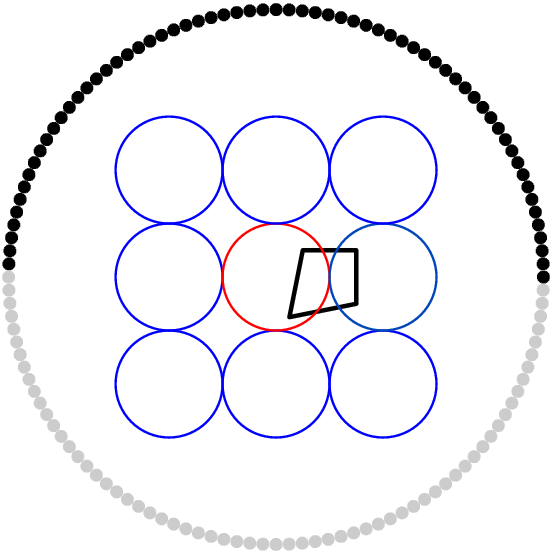}}
  \subfigure[$r=1/2$]{
  \includegraphics[width=0.18\textwidth]{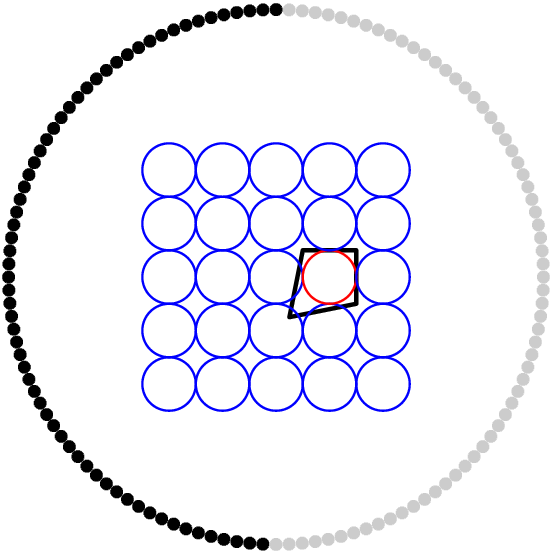}}
  \subfigure[$r=1/4$]{
  \includegraphics[width=0.18\textwidth]{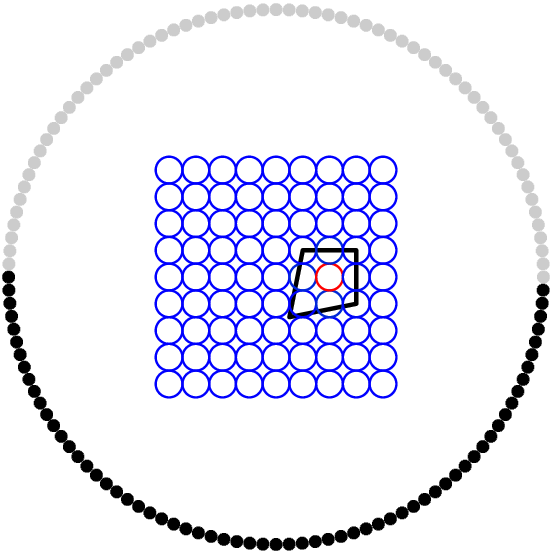}}
  \subfigure[$r=1/4$]{
  \includegraphics[width=0.18\textwidth]{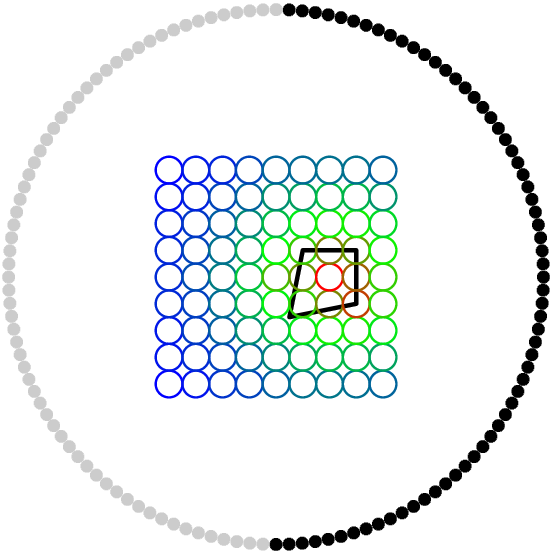}}
  \subfigure[$r=1/8$]{
  \includegraphics[width=0.18\textwidth]{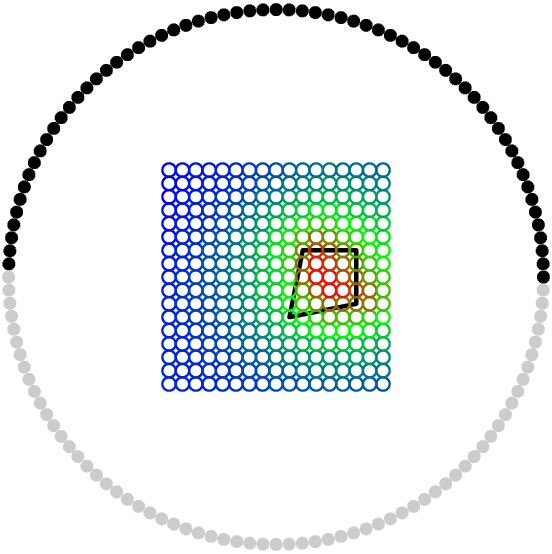}}\\
  \subfigure[$r=1$]{
  \includegraphics[width=0.18\textwidth]{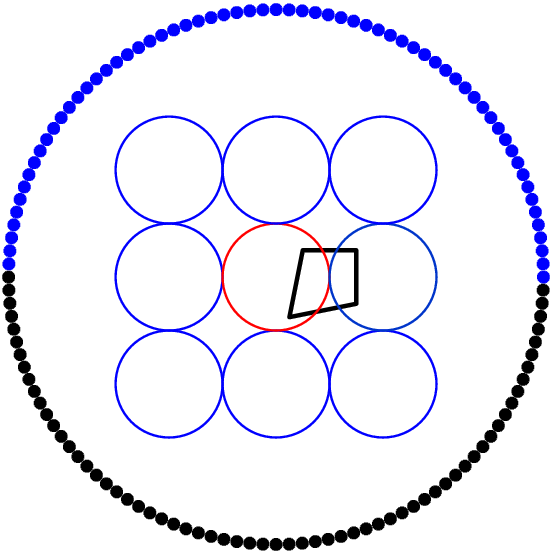}}
  \subfigure[$r=1/2$]{
  \includegraphics[width=0.18\textwidth]{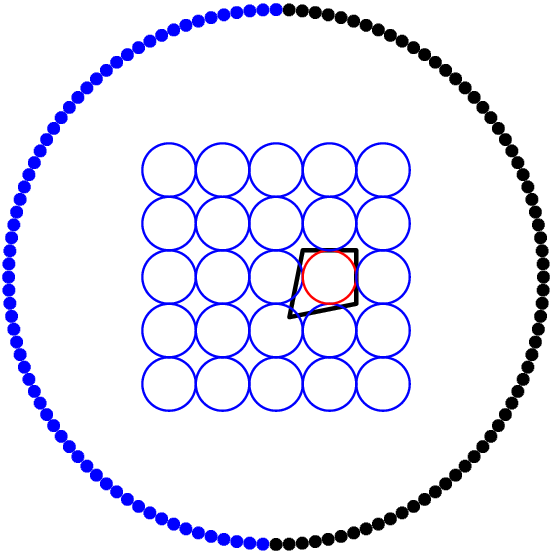}}
  \subfigure[$r=1/4$]{
  \includegraphics[width=0.18\textwidth]{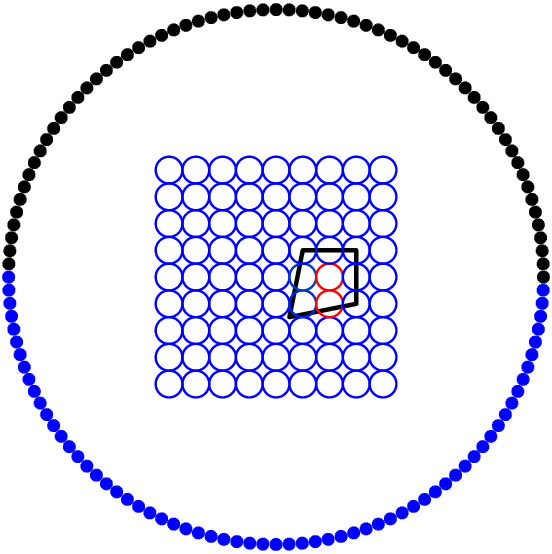}}
  \subfigure[$r=1/4$]{
  \includegraphics[width=0.18\textwidth]{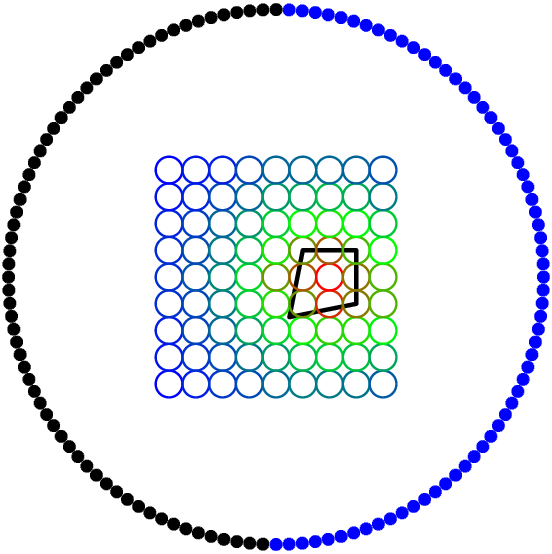}}
  \subfigure[$r=1/8$]{
  \includegraphics[width=0.18\textwidth]{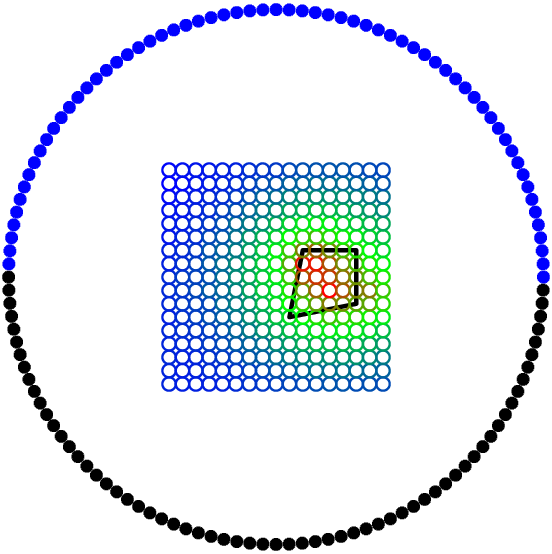}}
  \caption{Numerical results for Example \ref{Ex4}.
  The black line represents $\pa D$, the circle with knots represents $\pa B$.
  The boundary value $f$ takes value $0$, $1$, and $2$ at gray, black, and blue knots on $\pa B$, respectively.
  The colored circles represent $\pa\mho_P^{(\ell)}$ for different $P$ and $\ell$ with its color indicating the value of $\widetilde I_2(\mho_P^{(\ell)})$ for (a)--(c) and (f)--(h), and $\ln(\widetilde I_2(\mho_P^{(\ell)}))$ for (d), (e), (i) and (j), respectively, in the sense of (\ref{1029-3}).}\label{tilde_exmedium4}
\end{figure}

We are not able to give a rigourous explanation to the sampling scheme in Example \ref{Ex4}.
Intuitively, we think in some sense the distance between ${\rm Ran}\,A(k^2,\Om)$ and ${\rm Ran}\,A(k^2,D)$ (or the distance between ${\rm Ran}\,A(k^2,\Om)$ and $\pa_\nu u|_{\pa B}$ with $u$ solving (\ref{-1})--(\ref{-3})) is close provided $\Om$ is close to $D$.
We guess this property also holds with $\Om$ replaced by $\mho$.

\section{Conclusion}\label{s5}
\setcounter{equation}{0}

In this paper, we have proved uniqueness results of an inverse boundary value problem to determine the constant coefficients $\sigma$ and $q$ as well as the Dirichlet polygon $D$ from a single pair of Cauchy data under some a priori assumptions on $D$ and the Dirichlet boundary value $f$ on $\pa B$.
We also propose a sampling method to simultaneously reconstruct the constant coefficients $\sigma$ and $q$, and the rough location and convex hull of the polygon $D$.
{\x Obviously, the numerical results can be improved by iteration methods.}
All these numerical methods are based on a single pair of Cauchy data.
It remains interesting extend this method to the case when $D$ consists of several connect components.

\section*{Acknowledgements}

The work of Xiaoxu Xu is partially supported by National Natural Science Foundation of China grant 12201489, Shaanxi Fundamental Science Research Project for Mathematics and Physics (Grant No.23JSQ025), Young Talent Fund of Association for Science and Technology in Shaanxi, China (Grant No.20240504), the Young Talent Support Plan of Xi'an Jiaotong University, the Fundamental Research Funds for the Central Universities grant xzy012022009.
The work of Guanghui Hu is partially supported by National Natural Science Foundation of China grant 12071236 and the Fundamental Research Funds for Central Universities in China grant 63213025.


\begin{thebibliography}{10}

\bibitem{BLS}
Eemeli Bl\aa sten, Lassi P\"{a}iv\"{a}rinta, and John Sylvester.
\newblock Corners always scatter.
\newblock {\em Comm. Math. Phys.}, 331(2):725--753, 2014.

\bibitem{BC2002}
Marius Bochniak and Fioralba Cakoni.
\newblock Domain sensitivity analysis of the acoustic far-field pattern.
\newblock {\em Math. Methods Appl. Sci.}, 25(7):595--613, 2002.

\bibitem{Cakoni14}
Fioralba Cakoni and David Colton.
\newblock {\em A qualitative approach to inverse scattering theory}, volume 188
  of {\em Applied Mathematical Sciences}.
\newblock Springer, New York, 2014.

\bibitem{CGH2010}
Fioralba Cakoni, Drossos Gintides, and Houssem Haddar.
\newblock The existence of an infinite discrete set of transmission eigenvalues.
\newblock {\em SIAM J. Math. Anal.}, 42(1):237--255, 2010.

\bibitem{cbj}
Beiji Chen.
\newblock The factorization method for inverse acoustic and electromagnetic
  scattering problems at fixed frequency.
\newblock {\em Undergraduate Thesis. School of Mathematics and Statistics,
  Xi'an Jiaotong University.}, 2023.

\bibitem{CK19}
{David} {Colton} and {Rainer} {Kress}.
\newblock {\em {I}nverse acoustic and electromagnetic scattering theory},
  volume~93 of {\em Applied Mathematical Sciences}.
\newblock Springer, Switzerland AG, fourth edition, 2019.

\bibitem{CK83}
David~L. Colton and Rainer Kress.
\newblock {\em Integral equation methods in scattering theory}.
\newblock Pure and Applied Mathematics (New York). John Wiley \& Sons, Inc.,
  New York, 1983.
\newblock A Wiley-Interscience Publication.

\bibitem{Costabel}
Martin Costabel.
\newblock Boundary integral operators on {L}ipschitz domains: elementary
  results.
\newblock {\em SIAM J. Math. Anal.}, 19(3):613--626, 1988.

\bibitem{ElHu2015}
Johannes Elschner and Guanghui Hu.
\newblock Corners and edges always scatter.
\newblock {\em Inverse Problems}, 31(1):015003, 17, 2015.

\bibitem{hu2018}
Johannes Elschner and Guanghui Hu.
\newblock Acoustic scattering from corners, edges and circular cones.
\newblock {\em Arch. Ration. Mech. Anal.}, 228(2):653--690, 2018.

\bibitem{EH19}
Johannes Elschner and Guanghui Hu.
\newblock Uniqueness and factorization method for inverse elastic scattering
  with a single incoming wave.
\newblock {\em Inverse Problems}, 35(9):094002, aug 2019.

\bibitem{Evans}
Lawrence~C. Evans.
\newblock {\em Partial Differential Equations}, volume~19 of {\em Graduate
  Students in Mathematics}.
\newblock American Mathematical Society, 2010.

\bibitem{GT}
David Gilbarg and Neil~S. Trudinger.
\newblock {\em Elliptic partial differential equations of second order}.
\newblock Classics in Mathematics. Springer-Verlag, Berlin, 2001.
\newblock Reprint of the 1998 edition.

\bibitem{HL}
Guanghui Hu and Jingzhi Li.
\newblock Inverse source problems in an inhomogeneous medium with a single
  far-field pattern.
\newblock {\em SIAM J. Math. Anal.}, 52(5):5213--5231, 2020.

\bibitem{Ik1998}
Masaru Ikehata.
\newblock Reconstruction of the shape of the inclusion by boundary
  measurements.
\newblock {\em Comm. Partial Differential Equations}, 23(7-8):1459--1474, 1998.

\bibitem{I99}
Masaru Ikehata.
\newblock Enclosing a polygonal cavity in a two-dimensional bounded domain from
  {C}auchy data.
\newblock {\em Inverse Problems}, 15(5):1231--1241, 1999.

\bibitem{Ik1999}
Masaru Ikehata.
\newblock Reconstruction of a source domain from the {C}auchy data.
\newblock {\em Inverse Problems}, 15(2):637--645, 1999.


\bibitem{Kirsch98}
Andreas Kirsch.
\newblock Characterization of the shape of a scattering obstacle using the
  spectral data of the far field operator.
\newblock {\em Inverse Problems}, 14(6):1489--1512, dec 1998.

\bibitem{Kirsch05}
Andreas Kirsch.
\newblock The factorization method for a class of inverse elliptic problems.
\newblock {\em Math. Nachr.}, 278(3):258--277, 2005.

\bibitem{Kirsch08}
Andreas Kirsch and Natalia Grinberg.
\newblock {\em The factorization method for inverse problems}, volume~36 of
  {\em Oxford Lecture Series in Mathematics and its Applications}.
\newblock Oxford University Press, Oxford, 2008.

\bibitem{KH}
Andreas Kirsch and Frank Hettlich.
\newblock {\em The mathematical theory of time-harmonic {M}axwell's equations},
  volume 190 of {\em Applied Mathematical Sciences}.
\newblock Springer, Cham, 2015.
\newblock Expansion-, integral-, and variational methods.

\bibitem{Kress14}
Rainer Kress.
\newblock {\em Linear integral equations}, volume~82 of {\em Applied
  Mathematical Sciences}.
\newblock Springer, New York, third edition, 2014.

\bibitem{KS}
Steven Kusiak and John Sylvester.
\newblock The scattering support.
\newblock {\em Comm. Pure Appl. Math.}, 56(11):1525--1548, 2003.

\bibitem{Lin2021}
Yi-Hsuan Lin, Gen Nakamura, Roland Potthast, and Haibing Wang.
\newblock Duality between range and no-response tests and its application for
  inverse problems.
\newblock {\em Inverse Probl. Imaging}, 15(2):367--386, 2021.

\bibitem{Lin2023}
Yi-Hsuan Lin, Gen Nakamura, Roland Potthast, and Haibing Wang.
\newblock Corrigendum to ``{D}uality between range and no-response tests and
  its application for inverse problems''.
\newblock {\em Inverse Probl. Imaging}, 17(4):907, 2023.

\bibitem{LS2018}
Juan Liu and Jiguang Sun.
\newblock Extended sampling method in inverse scattering.
\newblock {\em Inverse Problems}, 34(8):085007, jun 2018.

\bibitem{P2003}
D.~Russell Luke and Roland Potthast.
\newblock The no response test---a sampling method for inverse scattering
  problems.
\newblock {\em SIAM J. Appl. Math.}, 63(4):1292--1312, 2003.

\bibitem{MaHu}
Guanqiu Ma and Guanghui Hu.
\newblock Factorization method with one plane wave: from model-driven and
  data-driven perspectives.
\newblock {\em Inverse Problems}, 38(1):Paper No. 015003, 26, 2022.

\bibitem{WM}
William McLean.
\newblock {\em Strongly elliptic systems and boundary integral equations}.
\newblock Cambridge University Press, Cambridge, 2000.

\bibitem{Monk}
Peter Monk.
\newblock {\em Finite element methods for {M}axwell's equations}.
\newblock Numerical Mathematics and Scientific Computation. Oxford University
  Press, New York, 2003.

\bibitem{NP2013}
Gen Nakamura and Roland Potthast.
\newblock {\em Inverse modeling}.
\newblock IOP Expanding Physics. IOP Publishing, Bristol, 2015.
\newblock An introduction to the theory and methods of inverse problems and
  data assimilation.

\bibitem{rangetest}
Roland Potthast, John Sylvester, and Steven Kusiak.
\newblock A `range test' for determining scatterers with unknown physical
  properties.
\newblock {\em Inverse Problems}, 19(3):533--547, 2003.

\bibitem{QZZ}
Fenglong Qu, Bo~Zhang, and Haiwen Zhang.
\newblock A novel integral equation for scattering by locally rough surfaces
  and application to the inverse problem: the {N}eumann case.
\newblock {\em SIAM J. Sci. Comput.}, 41(6):A3673--A3702, 2019.

\bibitem{Sun2023}
Shiwei Sun, Gen Nakamura, and Haibing Wang.
\newblock Numerical studies of domain sampling methods for inverse boundary
  value problems by one measurement.
\newblock {\em J. Comput. Phys.}, 485:Paper No. 112099, 18, 2023.

\bibitem{wcy}
Chengyu Wu, and Jiaqing Yang.
\newblock The obstacle scattering for the biharmonic equation.
\newblock {\em Inverse Problems}, 41(10):105003, 2025.

\bibitem{XH24}
Xiaoxu Xu and Guanghui Hu.
\newblock An inverse obstacle problem with a single pair of {C}auchy data:
  Laplace's equation case.
\newblock {\em arXiv:2406.01503}, 2024.

\bibitem{XH2024}
Xiaoxu Xu, Guanqiu Ma, and Guanghui Hu.
\newblock Detection of a piecewise linear crack with one incident wave.
\newblock {\em arXiv:2405.05179}, 2024.

\end{thebibliography}
\end{document}